\documentclass[11pt]{article}
\usepackage[utf8]{inputenc}
\usepackage{geometry}
\usepackage{algorithm}
\usepackage{algorithmicx}
\usepackage{algpseudocode}
\usepackage{stmaryrd} 
\usepackage{subcaption}
\usepackage[whole]{bxcjkjatype}
\usepackage{graphicx}[dvipdfmx]
\usepackage{natbib}
\usepackage{hyperref}
\hypersetup{
    colorlinks = true,
	citecolor=blue,
	linkcolor=blue
}
\usepackage{wrapfig}
\usepackage{authblk}
\usepackage{enumerate}

\usepackage{fourier}
\usepackage{amsmath,amssymb,amsthm,bbm}
\theoremstyle{definition}
\newtheorem{theorem}{Theorem}
\newtheorem*{theorem*}{Theorem}
\newtheorem{definition}[theorem]{Definition}

\newtheorem{lemma}[theorem]{Lemma}
\newtheorem{proposition}[theorem]{Proposition}
\newtheorem{remark}[theorem]{Remark}

\newtheorem{assumption}{Assumption}

\newtheorem{corollary}[theorem]{Corollary}

\newcommand{\bs}{\boldsymbol}
\newcommand{\diff}{\mathrm{d}}

\newcommand{\internal}{\text{Int}}
\renewcommand{\hat}{\widehat}
\renewcommand{\tilde}{\widetilde}

\newcommand{\dhauss}{d_{\text{Haus.}}}
\newcommand{\llb}{\llbracket}
\newcommand{\rrb}{\rrbracket}
\newcommand{\f}{\mathcal{F}}
\newcommand{\flip}{\f^{\text{L}}}

\newcommand{\flipinv}{\flip_{\text{INV}}}

\newcommand{\lebesgue}{\mathcal{L}}
\newcommand{\vertinfty}[1]{\VERT #1 \VERT_{L^{\infty}}}
\renewcommand{\hat}{\widehat}
\renewcommand{\tilde}{\widetilde}
\newcommand{\shikaku}{\square}
\newcommand{\shikakuimage}{\Diamond}

\newcommand{\prisk}{\mathsf{Q}}
\newcommand{\priskinv}{\prisk_{\mathrm{INV}}}\newcommand{\prisklinv}{\prisk_{\mathrm{INV}}^\ddagger}

\newcommand{\risk}{\mathsf{R}}
\newcommand{\riskinv}{\risk_{\mathrm{INV}}}\newcommand{\risklinv}{\risk_{\mathrm{INV}}^\ddagger}

\title{
Minimax Analysis for Inverse Risk in \\ Nonparametric Planer Invertible Regression
}
\author[1,3]{Akifumi Okuno}
\author[2,1,3]{Masaaki Imaizumi}
\affil[1]{The Institute of Statistical Mathematics}
\affil[2]{The University of Tokyo}
\affil[3]{RIKEN Center for Advanced Intelligence Project}

\date{\empty}

\begin{document}

\maketitle

\begin{abstract}
    We study a minimax risk of estimating inverse functions on a plane, while keeping an estimator is also invertible. Learning invertibility from data and exploiting an invertible estimator are used in many domains, such as statistics, econometrics, and machine learning. Although the consistency and universality of invertible estimators have been well investigated, analysis of the efficiency of these methods is still under development. In this study, we study a minimax risk for estimating invertible bi-Lipschitz functions on a square in a $2$-dimensional plane. 
    We first introduce two types of $L^2$-risks to evaluate an estimator which preserves invertibility. 
    Then, we derive lower and upper rates for minimax values for the risks associated with inverse functions. 
    For the derivation, we exploit a representation of invertible functions using level-sets. Specifically, to obtain the upper rate, we develop an estimator asymptotically almost everywhere invertible, whose risk attains the derived minimax lower rate up to logarithmic factors. 
    The derived minimax rate corresponds to that of the non-invertible bi-Lipschitz function, which shows that the invertibility does not reduce the complexity of the estimation problem in terms of the rate. 
\end{abstract}


\section{Introduction}

\subsection{Background}

Learning invertible structures from data is a problem encountered in several fields, from more classical to modern ones, where an invertible function is a typical shape-constraint of functions. A traditional and well-known application is the \textit{nonparametric calibration problem}: in a nonparametric regression problem with an unknown invertible function, one estimates an input covariate corresponding to an observed response variable. This problem has been studied by \citet{knafl1984nonparametric}, \citet{osborne1991statistical}, \citet{chambers1993bias}, \citet{gruet1996nonparametric}, \citet{tang2011two} and \citet{tang2015two}, and applied in the fields of biology and medicine \citep{tang2011two,tang2015two}. A different application in econometrics is the \textit{nonparametric instrumental variable}, developed by \citet{newey2003instrumental} and \citet{horowitz2011applied}. This is an ill-posed problem with conditional expectations. For instance, \citet{krief2017direct} studies the estimation by direct usage of inverse functions. Another application that has been developed rapidly in recent years is a framework for \textit{normalizing flow} used for generative models in machine learning, developed by \citet{rezende2015variational} and \citet{dinh2017density}. A related problem is the analysis of latent independent components using nonlinear invertible maps~\citep{dinh2014nice,hyvarinen2016unsupervised}. Under this problem, an observed data distribution is regarded as a transformation of a latent variable by an unknown invertible function, and this function is estimated by an invertible estimator to reconstruct the latent variable (for a review, see \citet{kobyzev2020normalizing}). Several methods have been developed for estimating invertible functions, for example, \citet{dinh2014nice}, \citet{papamakarios2017masked}, \citet{kingma2016improved}, \citet{huang2018neural}, \citet{de2020block} and \citet{ho2019flow++}.

In the univariate case ($d=1$), an error of the invertible estimators has been actively analyzed. In this case, the estimation of invertible functions is related to estimating strictly monotone functions, and there are many related studies in the field of isotonic regression (for a general introduction, see \citet{groeneboom2014nonparametric}). \citet{tang2011two,tang2015two} and \citet{gruet1996nonparametric} study an estimation for an input point $\Bar{\bs x} =\bs f^{-1}(\bs t) \in [-1,1]$ corresponding to an observed output $\bs t \in \mathbb{R}$ with an invertible function $\bs f$. 
Specifically, \citet{tang2011two} shows that a pointwise estimator $\hat{\bs x}$, which is based on the estimation of monotone functions, achieves a parametric convergence rate $|\hat{\bs x} - \bar{\bs x}| = O_P(n^{-1/2})$, where $n$ is the number of observations. 
They also establish an asymptotic distribution of the estimator $\Bar{\bs x}$. \citet{krief2017direct} develops an estimator $\tilde{\bs f}$ for an unknown invertible function $\bs f_*$, which is written as a conditional expectation with an $r$-times continuously differentiable distribution function, and study its convergence in terms of a sup-norm $\|\cdot\|_{L^\infty}$ as $\mathbb{E}[\|{\bs f_*} - \tilde{\bs f}\|_{L^\infty}^2] = O(n^{-2r/(2r+1)})$. Because this rate is slower than the minimax optimal rate on (even a non-invertible) $r$-differentiable functions, it is suggested that this rate does not achieve optimality.

For the multivariate ($d\geq 2$) case, there are few studies on the rate of errors, because a multivariate invertible function may not be represented by a simple monotone function as the univariate case. Several studies for normalizing flows  show the universality of each developed flow model (e.g., \citet{huang2018neural,jaini2019sum,teshima2020coupling}). However, these studies do not discuss efficiency, and only a few have investigated a volume of approximation errors of simple flow models~\citep{pmlr-v108-kong20a}.

The minimax rate of risk is a specific measure describing an effect of shape constraints such as invertibility, and one primary interest is whether shape constraints change the minimax rate.
It is studied that some shape constraints change the minimax rate to the parametric rate $O(n^{-1/2})$, such as unimodal \citep{bellec2018sharp}, convex \citep{guntuboyina2015global}, or log-concave \citep{kim2018adaptation}, whereas the ordinary rate without shape constraints is $O(n^{-r/(2r+d)})$ with an input dimension $d$ and smoothness $r$ of a target function.
Furthermore, even in the invertible setting, \citet{tang2015two} achieved the parametric rate for the pointwise estimator.
In contrast, the monotonicity constraint does not change the rate, that is, \cite{low2002estimating} shows that the nonparametric rate  appears in the estimation of monotone functions.
Based on these contrastive facts, whether the invertible constraint improves $L^2$-risk is an open question to clarify the efficiency of invertible function estimation.

\subsection{Problem Setting}

We consider a nonparametric planer regression problem with an invertible bi-Lipschitz function, and study an invertible estimator for the problem. We set the input dimension as $d=2$ and define $I:=[-1,1]$. 
We consider a set of invertible and bi-Lipschitz functions as
\begin{align*}
    \flipinv
    :=
    \left\{
        \bs f:I^2 \to I^2
        \mid \forall \bs y \in I^2,
        !\exists \bs x \in I^2 \text{ s.t. }\bs f(\bs x)=\bs y, \text{bi-Lipschitz}
    \right\}
\end{align*}
where $!\exists$ denotes unique existence, and a function $\bs f$ is called bi-Lipschitz if $L^{-1}\|\bs x-\bs x'\|_2 \le \|\bs f(\bs x)-\bs f(\bs x')\|_2 \le L\|\bs x-\bs x'\|_2$ holds for some $L\geq 1$ for any $\bs x,\bs x' \in I^2$. 
The bi-Lipschitz property is reasonable in dealing with invertible functions, because
$\bs f \in \flipinv$ holds if and only if $\bs f^{-1} \in \flipinv$ holds
(see Lemma~\ref{lem:bi_lipschitz}).
Note that invertible and continuous function is called homeomorphism. 

Assume we have observations $\mathcal{D}_n:=\{(\bs X_i,\bs Y_i)\}_{i=1}^{n} \subset I^2 \times \mathbb{R}^2$ that independently and identically follow the regression model for $i=1,...,n$:
\begin{align}
    \bs Y_i =\bs f_*(\bs X_i)+\bs \varepsilon_i,
    \quad 
    \bs \varepsilon_i \overset{\text{i.i.d.}}{\sim} N_2(\bs 0,\sigma^2 \bs I_2) \label{def:model}
\end{align}
with a true function $\bs f_* \in \flipinv$ and $\sigma^2>0$. Let $P_{\bs X}$ be a marginal measure of $\bs X_i$, and we assume that $P_{\bs X}$ has a 
density function which is positive and bounded on $I^2$.

\subsection{Analysis Framework with Inverse Risk}

The goal is to investigate the difficulty in estimating invertible functions by invertible estimators. 
To this end, we define two risks; (i) an \textit{inverse risk} to evaluate both an estimation error and invertibility of estimators, and (ii) an \textit{$L^2$-risk for an inverse}.
Preliminary, for any $\bs y \in I^2$, $\bar{\bs f}_n^{\ddagger}(\bs y)$ denotes $\bs x \in I^2$ if it  satisfies $\bar{\bs f}_n(\bs x)=\bs y$ uniquely, and some constant vector $\bs c \in \mathbb{R}^2 \setminus I^2$ otherwise. Namely, $\bar{\bs f}_n^{\ddagger}$ represents a  quasi-inverse of the function $\bar{\bs f}_n$ (that can be defined to not entirely-invertible functions). 

\textbf{(i) Inverse risk}: 
as the first risk, we develop an inverse $L^2$-risk as 
\[
    \riskinv(\bar{\bs f}_n,\bs f_*)
    =
    \mathbb{E}_n\left[
        \priskinv(\bar{\bs f}_n,\bs f_*)
    \right],
\]
where $\mathbb{E}_n$ denotes the expectation with respect to the observations $\mathcal{D}_n$, 
\[
    \priskinv(\bar{\bs f}_n,\bs f_*)
    :=
    \VERT \bar{\bs f}_n - \bs f_*  \VERT_{L^2(P_{\bs X})}^2
    +
    \psi\left( 
    \VERT \bar{\bs f}_n^{\ddagger} - \bs f_*^{-1}  \VERT_{L^2(P_{\bs X})}
    \right)
\]
denotes the predictive inverse $L^2$-risk, $\VERT \bs f \VERT_{L^2(P_{\bs X})} := (\sum_{j=1}^2 \int |f_j|^2 \diff P_{\bs X})^{1/2}$ is an $L^2$-norm for vector-valued functions, 
and 
\begin{align}
\psi \in \Psi:=\left\{
    \psi:\mathbb{R}_{\ge 0} \to \mathbb{R}_{\ge 0} \text{ is continuous, increasing, and }\psi(0)=0
    \right\}
    \label{eq:Psi}
\end{align}
denotes a non-negative penalty function. 
In our upper-bound analysis, we consider $\psi(z)=z^4$ and $\psi(z)=z^2$. 
By virtue of the penalty term $\psi(\cdot)$,  $\riskinv(\bar{\bs f}_n,\bs f_*) \to 0$ indicates both that $\bar{\bs f}_n$ is almost everywhere invertible and that  $\bar{\bs f}_n$ and $\bar{\bs f}_n^{\ddagger}$ are consistent estimators.
Using this risk, we can discuss constructing invertible estimators in the context of nonparametric regression. 

\textbf{(ii) $L^2$-risk for inverse}:
As the second risk, more simply, we define an $L^2$-risk for an inverse of $\bs f_*$.
It is defined as the following form:
\begin{align*}
    \risklinv(\bar{\bs f}_n,\bs f_*)
    :=
    \mathbb{E}_n\left[
        \prisklinv(\bar{\bs f}_n,\bs f_*)
    \right],
    \quad \text{where} \quad
    \prisklinv (\bar{\bs f}_n,\bs f_*) :=  \VERT \bar{\bs f}_n^{\ddagger} - \bs f_*^{-1} \VERT_{L^2(P_X)}^2.
\end{align*}
This risk is not only designed simply to evaluate the estimation error of the inverse function $\bs f_*^{-1}$, but also considers whether the estimator $\bar{\bs f}_n$ is invertible, since it utilizes the modified inverse $\bar{\bs f}_n^{\ddagger}$.

Then, we study the minimax inverse risk and the minimax $L^2$-risk for inverses of the regression problem, that is, we consider the following value
\begin{align*}
    \inf_{\bar{\bs f}_n} 
    \sup_{\bs f_* \in \flipinv}
    \riskinv(\bar{\bs f}_n, \bs f_*),
\end{align*}
and that with $\risklinv (\bar{\bs f}_n,\bs f_*)$.
Here, the infimum with respect to $\bar{\bs f}_n$ is taken over all measurable estimators, depending on $\mathcal{D}_n$. Note that this minimax inverse risk is related to an ordinary minimax risk without the invertibility of estimators, that is, $\inf_{\bar{\bs f}_n} \sup_{\bs f_* \in \flipinv} \riskinv(\bar{\bs f}_n, \bs f_*) \geq \inf_{\bar{\bs f}_n} \sup_{\bs f_* \in \flipinv} \risk(\bar{\bs f}_n, \bs f_*)$ holds with the ordinary $L^2$-risk $ \risk(\bar{\bs f}_n, \bs f_*) =  \VERT  \bar{\bs f}_n - \bs f_* \VERT_{L^2(P_{\bs X})}^2$.

\subsection{Approach and Results}

Our analysis depends on the representation of invertible functions by level-sets. 
For an invertible function $\bs f=(f_1,f_2) \in \flipinv$, 
we represent its inverse as
\begin{align}
    \bs f^{-1}(\bs y)
    =  L_{f_1}(y_1) \cap  L_{f_2}(y_2)
    \label{intro:level_set_rep}
\end{align}
where $L_{f_j}(y_j):=\{\bs x \in I^2 \mid f_j(\bs x) = y_j\}$ is a level-set for $y_j \in I$ and $j=1,2$. In this form, we can characterize invertibility of $\bs f$ by assuring the uniqueness of the intersection in \eqref{intro:level_set_rep}. This result allows the analysis of the smoothness and composition of an invertible estimator.

Our first main result is a lower bound of the minimax inverse risk and the minimax $L^2$-risk for inverses based on the developed representation. Specifically, we show that with $d=2$ and any $\psi \in \Psi$:
\begin{align*}
    \min\left\{\inf_{\bar{\bs f}_n} 
    \sup_{\bs f_* \in \flipinv}
    \riskinv(\bar{\bs f}_n, \bs f_*), \inf_{\bar{\bs f}_n} 
    \sup_{\bs f_* \in \flipinv}
    \risklinv(\bar{\bs f}_n, \bs f_*)\right\} \gtrsim n^{-2/(2+d)},
\end{align*}
where $\gtrsim$ denotes an asymptotic inequality up to constants, and $\inf_{\bar{\bs f}_n}$ takes infimum over all the possible estimators depending on $\mathcal{D}_n$. 
This rate corresponds to a minimax rate of estimating (not necessarily invertible) bi-Lipschitz functions.

This result gives a negative answer to the question of whether invertibility improves the minimax optimal rate to the parametric rate. That is, the family of functions restricted to be invertible is still sufficiently complicated, and no rate improvement occurs for $L^2$-risk when estimating it. 

Our second main result is an upper bound of the minimax risks.
To derive the bound, we develop a novel estimator for $\bs f_*$, and derive an upper bound on the inverse risk that corresponds to the lower bound. This estimator employs an arbitrary estimator of $\bs f_*$ minimax optimal in the sense of the standard $L^2$ risk, and amends it to be asymptotically almost everywhere invertible, so as to inherit the rate of convergence. As a result, for $d=2$ and $\psi(z)=z^4$, we obtain
\begin{align*}
    \inf_{\bar{\bs f}_n } 
    \sup_{\bs f_* \in \flipinv}
    \riskinv(\bar{\bs f}_n, \bs f_*) 
    \asymp 
    n^{-2/(2+d)},
\end{align*}
where $\asymp$ denotes the asymptotic equality up to the constants and logarithmic factors in $n$. 
While the above result considers the 4th power penalty $\psi(z)=z^4$ due to the pathological example shown in Supplement~\ref{subsec:twist}, the pathological example does not appear if the Lipschitz constant of $\bs f,\bs f^{-1}$ is less than $L = 2^{1/4} \approx 1.19$: for another penalty  $\psi(z)=z^2$, we also prove that 
\begin{align*}
    \inf_{\bar{\bs f}_n } 
    \sup_{\bs f_* \in \flipinv \cap \flip(2^{1/4})}
    \riskinv(\bar{\bs f}_n, \bs f_*) 
    \asymp
    \inf_{\bar{\bs f}_n } 
    \sup_{\bs f_* \in \flipinv \cap \flip(2^{1/4})}
    \risklinv(\bar{\bs f}_n, \bs f_*) 
    \asymp 
    n^{-2/(2+d)},
\end{align*}
with $\flip(L):=\{\bs f:I^2 \to I^2 \mid \bs f,\bs f^{-1}\text{ are }L\text{-Lipschitz}\}$. Similar to the above discussion, these results state that the learning invertibility problem has the same minimax rate for estimating bi-Lipschitz functions.

\subsection{Symbols and Notations}
$[n]:=\{1,2,\ldots,n\}$ for $n \in \mathbb{N}$. $\mathbbm{1}\{\cdot\}$ denotes an indicator function. For $p\in [0,\infty]$, the norm of vector $\bs x = (x_1,...,x_d)$ is defined as $\| \bs x \|_p := (\sum_{j} x_j^p)^{1/p}$. For a function $f: S \to \mathbb{R}$ and a set $S' \subseteq S$, we define $f(S'):= \{f(x) \mid x \in S'\}$. For a base measure $Q$, $\|f\|_{L^p(Q)} := (\int_S |f(\bs x)|^p \diff Q(\bs x))^{1/p}$ denotes an $L^p$-norm. For a vector-valued function $\bs f(\cdot) = (f_1(\cdot),...,f_d(\cdot)): S \to \mathbb{R}^d$, $\VERT \bs f \VERT_{L^p(Q)} := (\sum_{j=1}^d \int_S |f_j(\bs x)|^p \diff Q(\bs x))^{1/p}$ denotes its norm. When $Q$ is the Lebesgue measure, we simply write  $\|f\|_{L^p}$ and $\VERT f \VERT_{L^p}$. Specifically, $\vertinfty{\bs f}=\max_{j \in [d]}\sup_{\bs x \in S}|f_j(\bs x)|$. For any set $S \subset \mathbb{R}^d$, its boundary is expressed as $\partial S := \{\bs x \in \mathbb{R}^d \mid B_{\varepsilon}(\bs x) \subset S \, \text{ for some }\varepsilon>0\}$. $\mathbb{D}^d:=\{\bs x \in \mathbb{R}^d \mid \|\bs x\|_2 \le 1\}$ is a unit ball, and $\mathbb{S}^{d-1}=\{\bs x \in \mathbb{R}^d \mid \|\bs x\|_2=1\} \, (=\partial \mathbb{D}^d)$ denotes its surface. For two sets $X,X' \subset \mathbb{R}^d$, $\dhauss(X,X'):=\max\{\min_{x \in X} \max_{x' \in X'}\|x-x'\|_2,\min_{x' \in X'} \max_{x \in X}\|x-x'\|_2\}$ denotes Hausdorff distance. $\pm$ represents a simultaneous relation concerning a simultaneous sign inversion; for instance, $a(\pm 1)=b(\pm 1)$ means that both $a(1)=b(1)$ and $a(-1)=b(-1)$ hold, but does not mean that $a(1)=b(-1)$, $a(-1)=b(1)$.

\subsection{Organization}

The remainder paper is organized as follows. 
In Section~\ref{sec:level-set_representation}, we characterize invertible functions by their level-sets. In Section~\ref{sec:lower_bound_analysis}, we provide a minimax lower bound for inverse risk. We develop an invertible estimator, and prove that an upper bound of the risk by the estimator attains the lower bound up to logarithmic factors in Section~\ref{sec:upper_bound_analysis}. Supporting Lemmas, propositions and proofs of Theorems are listed in Appendix.

\section{Level-Set Representation on Invertible Function}
\label{sec:level-set_representation}

We consider a representation of invertible functions using the notion of level-sets, which will be used in our main results. 
That is, we describe an inverse of functions by an intersection of level-sets of coordinates of the functions. 
This approach is different from the commonly used representation of invertible functions by monotonicity \citep{krief2017direct}, local approximation \citep{tang2011two,tang2015two}, or Hessian normalization \citep{rezende2015variational,dinh2017density}.

We consider a vector-valued function $\bs f: I^2 \to I^2$ with its coordinate-wise representation $\bs f(\bs x) = (f_1(\bs x), f_2(\bs x))$ for $f_j: I^2 \to I$. For $j=1,2$, we define a level-set of $f_j$ for $y_j \in I$ as
\[
L_{f_j}(y_j):=\left\{\bs x \in I^2 \mid f_j(\bs x) = y_j \right\}.
\]
The notion of level-sets represents a slice of functions, whose shape depends on the nature of these functions. Then, we define the \textit{level-set representation} of $\bs f(\bs x)$.
\begin{definition}[Level-set representation]
    For a function $\bs f = (f_1,f_2): I^2 \to I^2$ and $\bs y \in I^2$, the level-set representation is defined as
    \begin{align}
        \bs f^{\dagger}(\bs y) := L_{f_1}(y_1) \cap L_{f_2}(y_2).
        \label{eq:inverse_intersection}
    \end{align}
\end{definition}
\noindent
This term is defined with an output-wise level-set of the function $\bs f$. The existence and nature of the intersection of $ \bs f^{\dagger}(\bs y)$ depends on the nature of $\bs f$. Then, the property of $ \bs f^{\dagger}(\bs y)$ explains the invertibility of $\bs f$. 
\begin{proposition}[Level-set representation for an invertible function] \label{prop:equiv_invertible_levelset}
$\bs f: I^2 \to I^2$ is invertible if and only if $\bs f^\dagger(\bs y)$ exists and uniquely determined for all $\bs y \in I^2$. Furthermore, if $\bs f$ is invertible, we have
\[
    \bs f^{-1}(\bs y) = \bs f^{\dagger}(\bs y).
\]
\end{proposition}
\noindent
From this result, if $\bs f$ is invertible, there exists a corresponding level-set representation. Additionally, the level-set has tractable geometric properties, which are useful for future analyses. We discuss the properties of level-sets in the next section.

We illustrate level-sets $L_{f_1},L_{f_2}$ in Figure~\ref{fig:level_set_intersection}. The orange and blue lines represent $L_{f_1}(y_1)$ and $L_{f_2}(y_2)$, respectively; $\bs x = \bs f^{-1}(\bs y)$ coincides with the intersection $L_{f_1}(y_1)\cap L_{f_2}(y_2)$ as described in eq.~(\ref{eq:inverse_intersection}). 

\begin{figure}[!ht]
\centering
\includegraphics[width=0.5\textwidth]{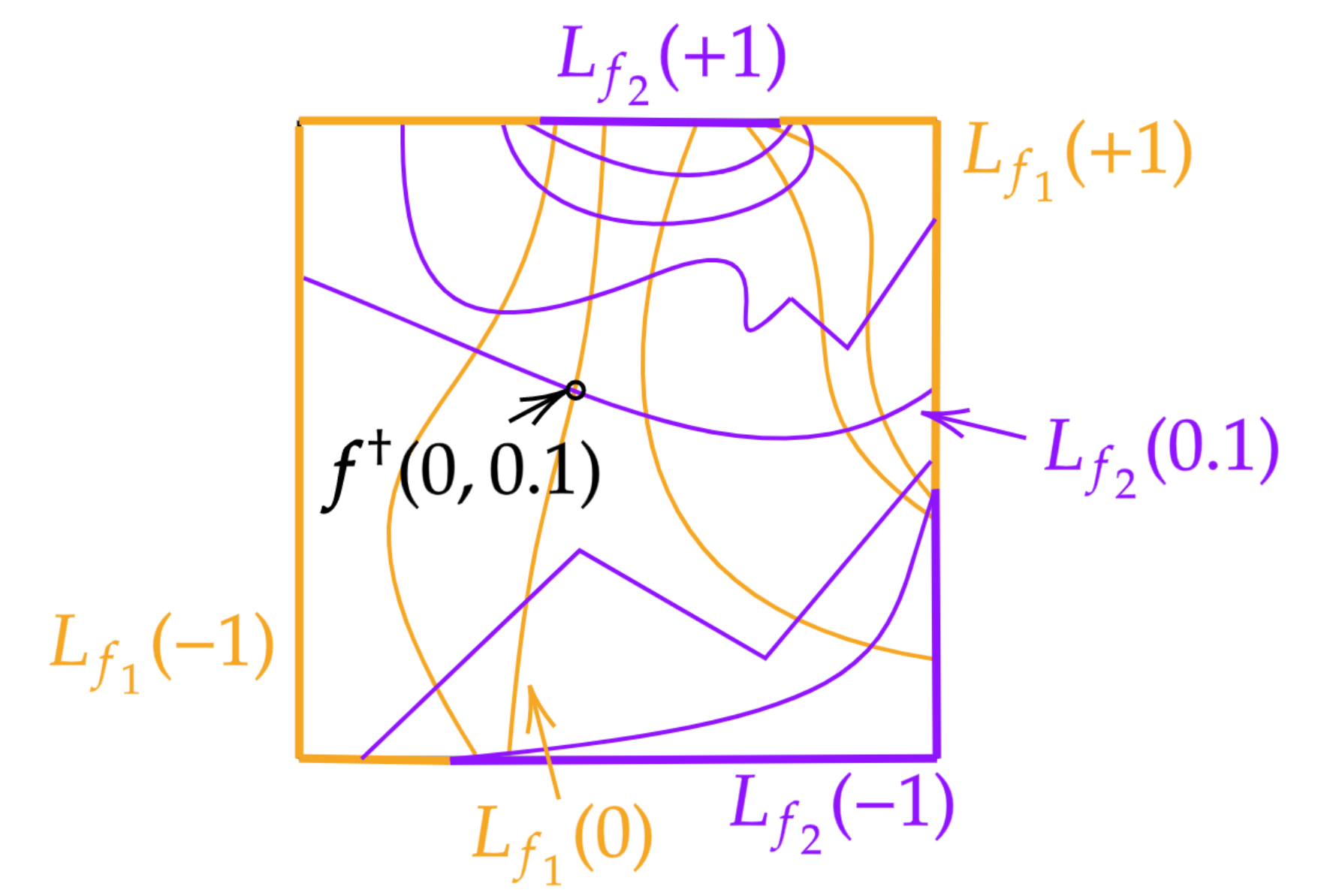}
\caption{Level-sets $L_{f_1}(y_1)$ (orange) and $L_{f_2}(y_2)$ (purple) in $I^2$ for $\bs f \in \flipinv$. These provide a level-set representation $\bs f^{\dagger}$ of $\bs f$, and the uniqueness of the intersection (black dot) of each level-set ensures invertibility, yielding $\bs f^{-1}(\bs y)=\bs f^{\dagger}(\bs y)$.}
\label{fig:level_set_intersection}
\end{figure}

\subsection{Property of Level-Set by Invertible Function}

We consider an invertible function $\bs f \in \flipinv$, where level-sets $L_{f_j}(y_j)$ have some geometric properties that are critical for the analyses on minimax inverse risk in Sections~\ref{sec:lower_bound_analysis} and \ref{sec:upper_bound_analysis}. 
All results in this section are rigorously proven in Appendix \ref{sec:supporting_lemmas}.

A level-set has a parameterization with a parameter $\alpha \in I$: 

\begin{lemma} \label{lem:parameterization} 
For $\bs f \in \flipinv$, the following holds for each $y \in I$:
\begin{align*}
    &L_{f_1}(y)
    =
    \bigcup_{\alpha \in I}
    \bs f^{-1}(y,\alpha), \mbox{~and~}L_{f_2}(y)
    =
    \bigcup_{\alpha \in I}
    \bs f^{-1}(\alpha,y).
\end{align*}
\end{lemma}
\noindent 
This parameterization guarantees the smoothness of level-sets, together with the Lipschitz property of $\bs f$. This property prohibits a ``sharp fluctuation'' in level-set $L_{f_j}$, as shown in Figure~\ref{fig:sharp_fluctuation}. 

\begin{figure}[!ht]
\centering
\includegraphics[width=0.65\textwidth]{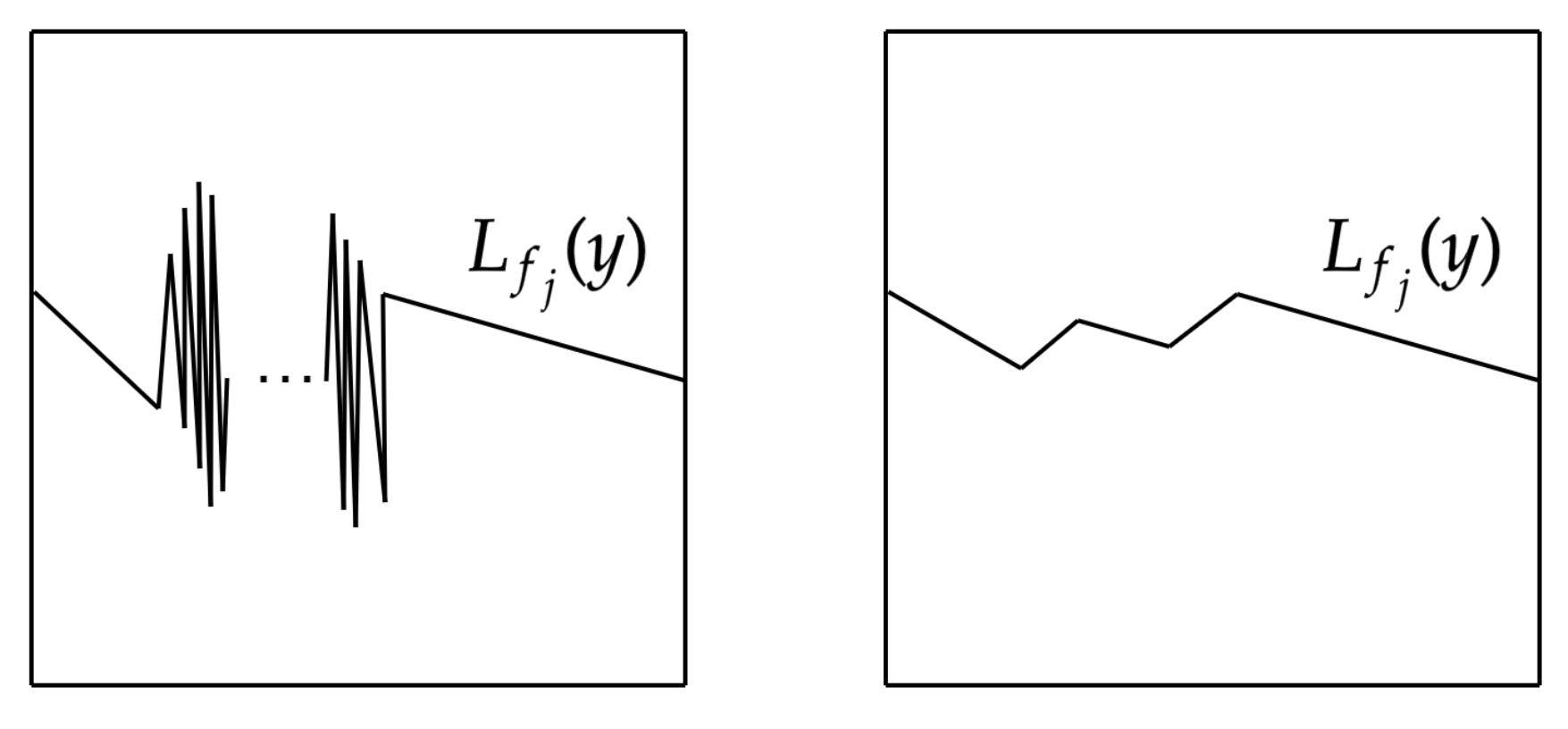}
\caption{Level-sets in $I^2$. [Left] $L_{f_j}(y)$ \textit{without} the Lipschitz continuity of $f_j$. [Right] $L_{f_j}(y)$ \textit{with} the Lipschitz continuity of $h_j$. If $f_j$ is Lipschitz continuous, the (excessively) sharp fluctuation along with one direction, shown in the left panel, does not appear. This property is clarified by parameterization (Lemma \ref{lem:parameterization}).}
\label{fig:sharp_fluctuation}
\end{figure}

Furthermore, level-set $L_{f_j}(y)$ is continuously shifted with respect to $y_j \in I$; more specifically, there exists $C \in (0,\infty)$ such that 
\[
    \dhauss(L_{f_j}(y),L_{f_j}(y')) \le \exists C |y-y'|
\]
holds for all $y,y' \in I$ (see Lemma~\ref{lemma:Hausdorff_Lipschitz} in Appendix~\ref{sec:supporting_lemmas}). The level-sets at $y=\pm 1$ are also properly included in the boundary of domain $I^2$: $L_{f_j}(\pm 1) \subset \partial I^2$ (see Lemma~\ref{lem:homeomorphism_boundary} in Appendix~\ref{sec:supporting_lemmas}). 

Whereas the above representation is for identifying the inverse function $\bs f^{-1}$, the level-set representation for the inverse function recovers the original function $\bs f$ itself: Lemma~\ref{lem:parameterization} which proves $L_{h_1}(x_1)=\bs f(x_1,I),L_{h_2}(x_2)=\bs f(I,x_2)$ with $\bs f(\bs x) = \bs h^{-1}(\bs x) = \bs h^{\dagger}(\bs x)$ leads to
\begin{align} 
    \bs f(\bs x)=\bs f(x_1,I) \cap \bs f(I,x_2).
    \label{eq:grid_representation}
\end{align}
As $\bs f(x_1,I)$ and $\bs f(I,x_2)$ are ($1$-dimensional) curve, they can be regarded as a kind of (skewed) ``grid" of the square $I^2$, identifying the unique point $\bs y=\bs f(\bs x)$ by their intersection. We employ this grid-like level-set representation for constructing an invertible estimator in Section~\ref{subsec:estimator}.

\section{Lower Bound Analysis}
\label{sec:lower_bound_analysis}

We develop a lower bound for the minimax risk. 
The direction of the proof is to utilize the $L^2$-risk $\risk(\hat{\bs f}_n, \bs f_*)$ and to develop a certain subset of invertible bi-Lipschitz functions $\f(\{\Xi^2_k\}_k) \subset \flipinv$ as follows:
\begin{align}
    \inf_{\bar{\bs f}_n} 
    \sup_{\bs f_* \in \flipinv}
    \riskinv(\bar{\bs f}_n, \bs f_*)\geq
    \inf_{\bar{\bs f}_n} 
    \sup_{\bs f_* \in \flipinv}
    \risk(\bar{\bs f}_n, \bs f_*) \geq \inf_{\bar{\bs f}_n}
    \sup_{\bs f \in \f(\{\Xi^2_k\}_k)}
    \risk(\bar{\bs f}_n, \bs f). \label{ineq:risks}
\end{align}
Then, we derive a lower bound on the right-hand side by two techniques: (i) the level-set representation developed in Section \ref{sec:level-set_representation}, and (ii) the information-theoretic approach for minimax risk (e.g., Section 2 in \citet{tsybakov2008introduction}).

\subsection{Minimax Lower Bound of the Inverse Risk}
\label{subsec:lower_bound}

We derive the minimax lower bound for the inverse risk: applying the information-theoretic approach to the subset $\f(\{\Xi^2_k\}_k) \subset \flipinv$ shown in Section~\ref{subsec:subset_of_flipinv} yields the following theorem.
\begin{theorem} 
\label{thm:main_lower}
Let $\psi \in \Psi$. 
For $d=2$, there exists $C_* > 0$ such that we have
\begin{align*}
\inf_{\bar{\bs f}_n} 
    \sup_{\bs f_* \in \flipinv}
    \risk(\bar{\bs f}_n, \bs f_*)
     \geq  C_* n^{-2/(2 + d)}.
\end{align*}
\end{theorem}
\noindent 
See Section~\ref{subsec:subset_of_flipinv} for the proof outline, and Appendix~\ref{sec:proofs_for_lower_bound_analysis} for details. 
This lower bound on the rate indicates that imposing invertibility on the true function does not improve estimation efficiency in the minimax sense. This is because the lower rate $n^{-2/(2 + d)}$ is identical to the rate for estimating (non-invertible) Lipschitz functions (see \cite{tsybakov2008introduction}). Although set $\flipinv$ is smaller than a set of Lipschitz functions, we find that the estimation difficulty is equivalent in this sense.

We also derive a lower bound for an inverse risk based on the above results. 
By the relation \eqref{ineq:risks}, the following result holds without proof:
\begin{corollary}
Let $\psi \in \Psi$. 
For $d=2$, there exists $C_* > 0$ such that we have 
\begin{align*}
\inf_{\bar{\bs f}_n} 
    \sup_{\bs f_* \in \flipinv}
    \riskinv(\bar{\bs f}_n, \bs f_*)
     \geq  C_* n^{-2/(2 + d)}.
\end{align*}
\end{corollary}
\noindent
This result implies that the efficiency of estimators preserving invertibility, such as normalizing flow, coincides with that of the estimation without invertibility in this sense.

Moreover, we also develop a lower bound on the $L^2$-risk for the inverse functions: we obtain the following theorem: 
\begin{theorem} 
\label{thm:sub_lower}
For $d=2$, there exists $C_* > 0$ such that we have
\begin{align*}
\inf_{\bar{\bs f}_n} 
    \sup_{\bs f_* \in \flipinv}
    \risklinv(\bar{\bs f}_n,\bs f_*)
     \geq  C_* n^{-2/(2 + d)}.
\end{align*}
\end{theorem}
This result is simply obtained by leveraging the bi-Lipschitz property of $\bs f_*$ and the result of Theorem \ref{thm:main_lower}.
Given that this rate corresponds to the minimax rate of estimation error for Lipschitz continuous functions, this result also shows that the invertible property does not improve the rate as in the previous example.


\subsection[Construction of Subset of Flipinv]{Proof Outline: Construction of Subset of \texorpdfstring{$\flipinv$}{flipinv}}
\label{subsec:subset_of_flipinv}
Applying an information-theoretic approach to the subset $\flipinv$ constructed below proves Theorem~\ref{thm:main_lower}. 
The important technical point is to use the level-set representation developed in Section \ref{sec:level-set_representation} to guarantee the invertibility of functions in $\flipinv$.

We first define a set of functions $\Xi^2_k$ for $k \in \{1,2\}$ as follows.
Let $m \in \mathbb{N}$ and let $M>2m$. Using a hyperpyramid-type basis function $\Phi:\mathbb{R}^2 \to [ 0,1]$ 
\[
    \Phi(\bs x)
    =
    \begin{cases}
    \min_{\tilde{\bs x} \in \partial I^2}\|\bs x-\tilde{\bs x}\|_2
    & (\bs x \in I^2) \\
    0 & (\text{Otherwise.}) \\
    \end{cases},
\]
and grid points $t_j:=-1+\frac{2j-1}{m} \in I$ ($j=1,2,\ldots,m$), we define the bi-Lipschitz function as
\begin{align*}
       \chi_{\theta}(\bs x) 
       = 
       \sum_{j_1=1}^m  \sum_{j_{2}=1}^m  \frac{\theta_{j_1,j_{2}} }{M} \Phi \left( m\left(x_1 - t_{j_1}\right),m\left(x_{2} - t_{j_{2}}\right) \right)
       :
       I^2 \to [0,1/M],
\end{align*}
parameterized by a binary matrix $\theta=(\theta_{j_1,j_2}) \in \Theta_m^{\otimes 2}$ ($\Theta_m:=\{0,1\}^m$). Using the function $\chi_{\theta}$, we define a function class:
\begin{align}
    \Xi^2_k
    :=
    \left\{
        \xi_{\theta}: I^2 \to I \mid \xi_\theta(\bs x)
        =
        x_k + \chi_{\theta}(\bs x), 
        \theta \in \Theta_m^{\otimes 2}
    \right\},
    \label{eq:Xidk}
\end{align}
for $k \in \{1,2\}$. See Figure~\ref{fig:xi_theta} for an illustration of the function $\xi_{\theta} \in \Xi^2_k$. Using the function set $\Xi^2_k$ defined in (\ref{eq:Xidk}), we define the function class as
\[
    \f(\{\Xi^2_k\}_k)
    :=
    \left\{
        \bs f=(f_1,f_2):I^2 \to I^2
        \mid 
        f_k \in \Xi^2_k, \: k =1,2
    \right\}.
\]

\begin{figure}[!ht]
\centering
\begin{minipage}{0.46\textwidth}
\centering
\includegraphics[width=0.7\textwidth]{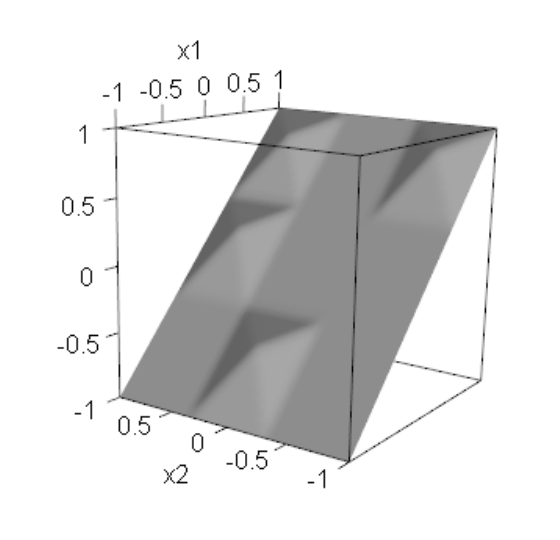}
\subcaption{$\xi_{\theta}(x_1,x_2)=x_1+\chi_{\theta}(x_1,x_2)$ for $k=1,m=3,M=6$. 
The entries in matrix $\theta \in \{0,1\}^{3 \times 3}$ are $\theta_{1,2}=\theta_{2,3}=\theta_{3,1}=\theta_{3,3}=1$, and $0$ otherwise.\label{fig:xi_theta}}
\end{minipage}
\hfill
\begin{minipage}{0.46\textwidth}
\centering
\vspace{1em}
\includegraphics[width=0.65\textwidth]{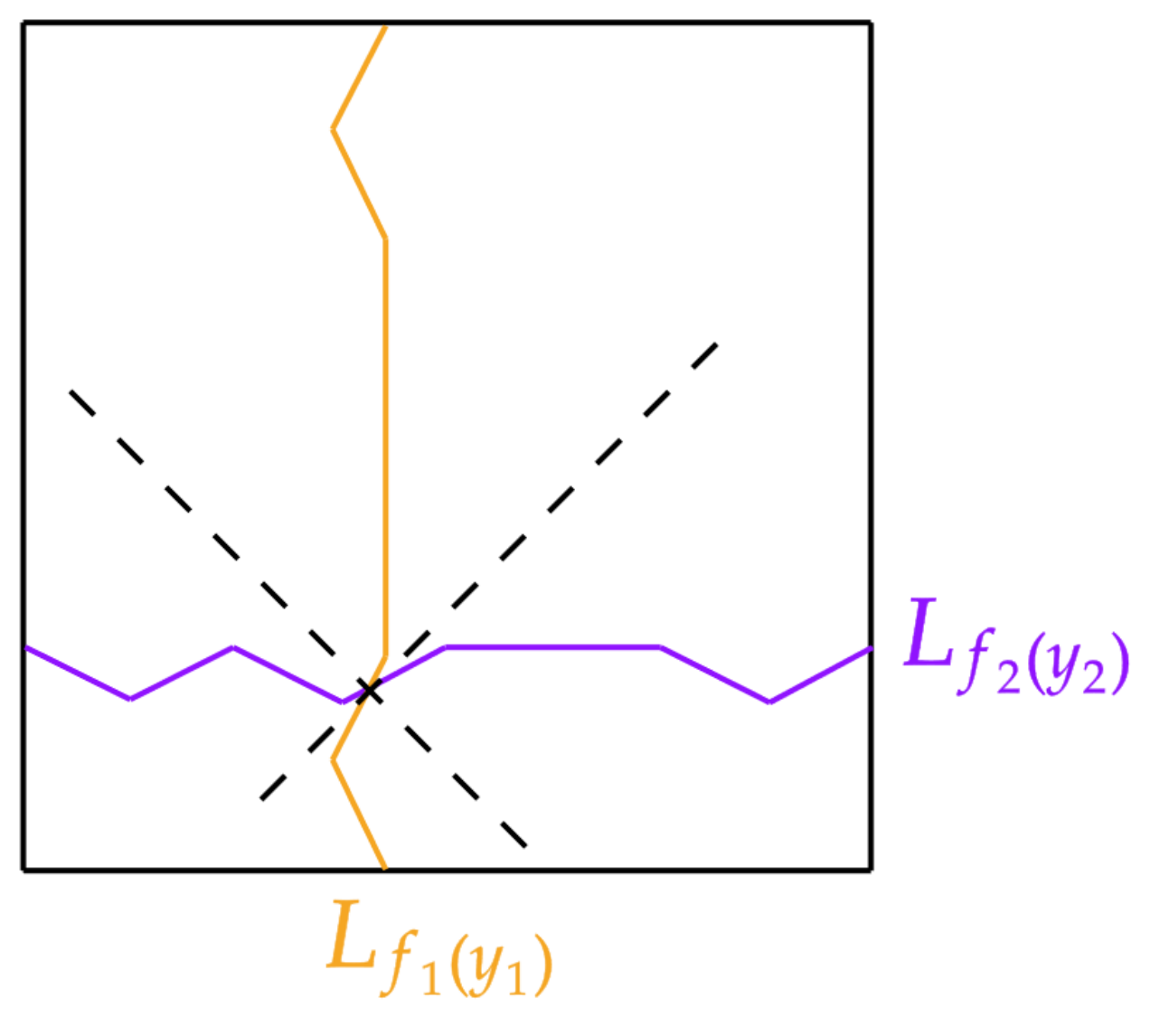}
\vspace{1em}
\subcaption{Level-sets $L_{f_1}$ and $L_{f_2}$ in $I^2$. Their slopes are  restricted so that the intersection is unique; hence, invertibility is guaranteed.\label{fig:levelset_finv}}
\end{minipage}
\end{figure}

We state the invertibility of $\bs f \in \f(\{\Xi^2_k\}_k)$ by the level-set representation in Proposition \ref{prop:equiv_invertible_levelset}. That is, using the fact that a function $f_k(\bs x) = x_k + \chi_{\theta}(\bs x) \in \Xi^2_k$ is piecewise linear, its level-set $L_{f_k}(y_k)$ is also piecewise linear with small slopes. Then, we can prove the uniqueness of the level-set representation $\bs f^{\dagger}(\bs x)$, which indicates the invertibility of $\bs f$ (see Figure~\ref{fig:levelset_finv}). 
We summarize the result as follows. 
\begin{proposition}
\label{prop:FXi_ivnertible}
$\f(\{\Xi^2_k\}_k) \subset \flipinv$. 
\end{proposition}

\section{Upper Bound Analysis}
\label{sec:upper_bound_analysis}

This section derives an upper bound of the minimax inverse risk, by developing an estimator $\hat{\bs f}_n$ which is almost everywhere invertible in the asymptotic sense. 
We first present the upper bound, and subsequently, describe the developed estimator. 

\subsection{Minimax Upper Bound of the Inverse Risk}
\label{subsec:upper_bound}


Using our developed estimator, 
we obtain the following upper bound on an inverse risk:


\begin{theorem} \label{thm:upper_d2}
Let $\psi(z)=z^4$. 
Consider $d=2$.
Suppose $\bs f_* \in \flipinv$ and Assumption \ref{asmp:uniform_estimator} hold. 
Then, for any $\beta>0$, there exists $C_* \in (0,\infty)$ such that
\begin{align*}
    \riskinv(\hat{\bs  f}_n, \bs f_*) \leq C_* n^{-2/(2+d)}(\log n)^{2\alpha+2\beta},
\end{align*}
holds for any sufficiently large $n$. 
\end{theorem}

\noindent 
See Appendix~\ref{sec:proofs_for_upper_bound_analysis} for the proof.
This result is consistent with the lower bound of the inverse minimax in Theorem \ref{thm:main_lower} up to logarithmic factors. We immediately obtain the following result:
\begin{corollary}
Let $\psi(z)=z^4$. 
Consider the setting in Theorem \ref{thm:upper_d2}.
Then, for any $\beta>0$, there exists $ \overline{C} \in (0, \infty)$ such that
\begin{align*}
    \inf_{\bar{\bs f}_n}  \sup_{\bs f_* \in \flipinv}
    \riskinv(\bar{\bs f}_n, \bs f_*)
     \leq  \overline{C} n^{-2/(2 + d)} (\log n)^{2\alpha + 2\beta}
\end{align*}
holds for any sufficiently large $n$.
\end{corollary}
\noindent
With this result, we achieve a tight evaluation of the minimax inverse risk in case $d=2$. This result implies that the difficulty of estimating invertible functions is similar to the case without invertibility, and that there are estimators that achieve the same rate up to logarithmic factors.

The penalty function $\psi(z)=z^4$ can be replaced to $\psi(z)=z^2$, by considering a function class $\flip(L) = \{\bs f:I^2 \to I^2 \mid \bs f,\bs f^{-1} \text{ is $L$-Lipschitz}\}$ with $L=2^{1/4} \approx 1.19$.

\begin{proposition}
\label{prop:upper_d2_phi2}
Let $\psi(z)=z^2$. 
Consider the setting in Theorem \ref{thm:upper_d2}.
Then, for any $\beta>0$, there exists $ \overline{C} \in (0, \infty)$ such that  
\begin{align*}
    \inf_{\bar{\bs f}_n}  \sup_{\bs f_* \in \flipinv \cap \flip(2^{1/4})}
    \riskinv(\bar{\bs f}_n, \bs f_*)
     \leq  \overline{C} n^{-2/(2 + d)} (\log n)^{2\alpha + 2\beta}
\end{align*}
holds for any sufficiently large $n$.
\end{proposition}

Next, we mention the immediate consequence of the above results. 
Considering the inequality $\prisklinv(\hat{\bs f}_n,\bs f_*) \le \priskinv(\hat{\bs f}_n,\bs f_*)$ with $\psi(z)=z^2$, 
Proposition~\ref{prop:upper_d2_phi2} leads to the following upper-bound without proof:
\begin{proposition}
\label{prop:sub_upper_d2_phi2}
Consider the setting in Theorem \ref{thm:upper_d2}.
Then, for any $\beta>0$, there exists $ \overline{C} \in (0, \infty)$ such that 
\begin{align*}
    \inf_{\bar{\bs f}_n}  \sup_{\bs f_* \in \flipinv \cap \flip(2^{1/4})}
    \risklinv(\hat{\bs f}_n,\bs f_*)
     \leq  \overline{C} n^{-2/(2 + d)} (\log n)^{2\alpha + 2\beta}
\end{align*}
holds for any sufficiently large $n$.
\end{proposition}

This constraint by $L=2^{1/4} \approx 1.19$ is essential and difficult to improve to larger constants. 
This is necessary so that the quadrilateral, which is a transformed small square in the domain $I^2$ by $\bs f_*$, does not become pathological with twists.
As shown in Remark~\ref{remark:twist_L=sqrt2} in Apppendix~\ref{subsec:Proof_of_Proposition_upper_d2_phi2}: even in the case $L=\sqrt{2} \approx 1.41$, there can be a pathological example that prohibits proving the minimax optimality with $\psi(z)=z^2$.

\subsection{Idea and Preparation for Invertible Estimator}
\label{subsec:estimator_outline}

We describe the developed invertible estimator, that attains the above upper-bound. 
The estimator is made by partitioning the domain $I^2$ and the range $I^2$ respectively, and combining local bijective maps between pieces of the partitions. To develop the partitions and bijective maps, we develop (i) a coherent rotation for $\bs f_*$ and (ii) two types of partitions of $I^2$ by squares and quadrilaterals. In this section, we introduce these techniques in preparation.

\subsubsection{Coherent Rotation}

First, we introduce an invertible function $\bs g_*: I^2 \to I^2$ whose endpoint level-sets correspond to endpoints of $I^2$, that is, $\bs g_*(\pm 1,I) = (\pm 1,I)$ and $\bs g_*(I,\pm 1)=(I,\pm 1)$ hold. Such $\bs g_*$ is utilized to define a partition of $I^2$ using quadrilaterals. See Figure~\ref{fig:gstar} for illustration.
To the aim, we define a bi-Lipschitz invertible function $\bs \rho: I^2 \to I^2$ for rotation, then obtain $\bs g_*$  as follows:
\begin{lemma}
\label{lemma:coherent_rotation}
    For every $\bs f_* \in \flipinv$, there exists an invertible map $\bs \rho \in \flipinv$ depending on $\bs f_*$, such that an invertible function
\begin{align}
\bs g_* = (g_{1},g_{2}) := \bs \rho \circ \bs f_* \in \flipinv
\end{align}
satisfies $\bs g_*(\pm 1,I) = (\pm 1,I)$ and $\bs g_*(I,\pm 1)=(I,\pm 1)$. 
\end{lemma}
\noindent
We refer to $\bs \rho$ as a \textit{coherent rotation}. We provide a specific form of $\bs \rho$ in Appendix~\ref{subsec:coherent_rotation}, and the proof of Lemma~\ref{lemma:coherent_rotation} is shown in Appendix~\ref{subsec:proof_of_lemma:coherent_rotation}. 

\begin{figure}[!ht]
\centering
\includegraphics[width=0.8\textwidth]{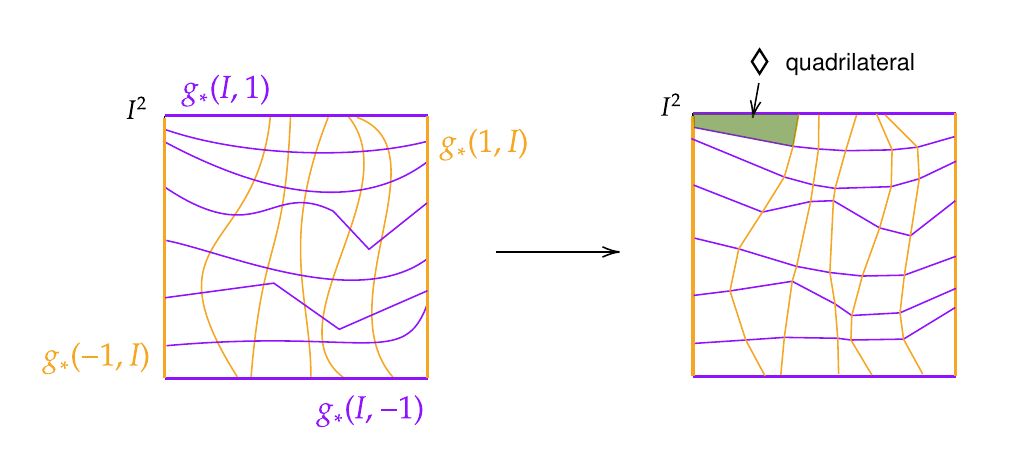}
\caption{(Left) Level-sets of $\bs g_*=\bs \rho \circ \bs f_*$, whose endpoints are aligned with a square $I^2$ by the coherent rotation. (Right) Partition of $I^2$ into quadrilaterals. Since the endpoint level-set of $\bs g_*$ is aligned to the endpoint of $I^2$, the partition is well-defined.
}
\label{fig:gstar}
\end{figure}

\subsubsection[Two Partitions of I2]{Two Partitions of $I^2$} \label{sec:partition}

We develop two types of partitions of $I^2$, in order to construct local bijective maps between pieces of the partitions, then combine them to develop an invertible function.

\begin{figure}[!ht]
\centering
\includegraphics[width=0.6\textwidth]{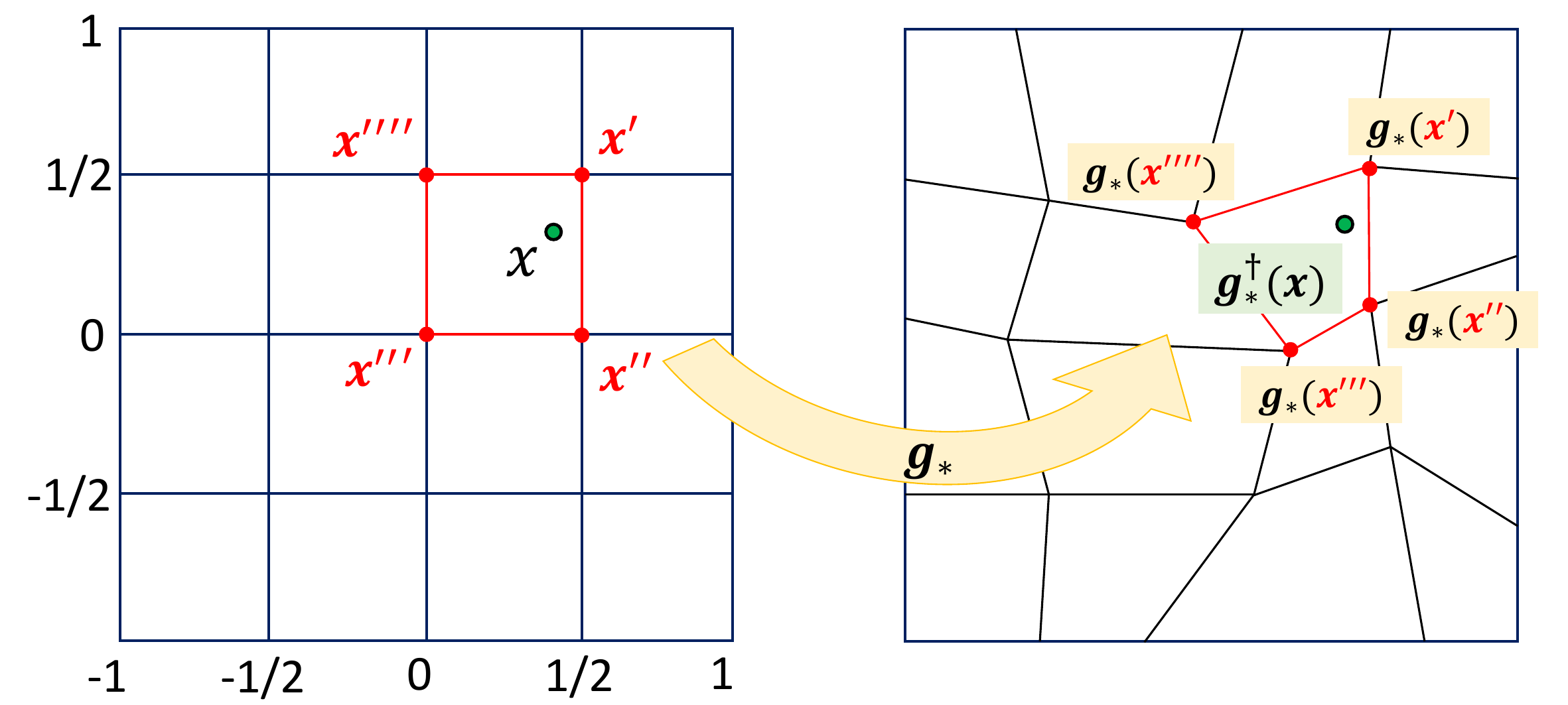}
\caption{Two partitions of $I^2$. The left $I^2$ is partitioned into squares $\shikaku$, and the right $I^2$ into quadrilaterals $\shikakuimage$. $\shikakuimage$ is defined by the vertices $\{\bs x',\bs x'',\bs x''',\bs x''''\}$ of $\shikaku$ mapped by $\bs g_*$.}
\label{fig:interpolation}
\end{figure}

The first partition is defined by grids in $I^2$. 
We consider a set of grids $\hat{I}^2:=\{0,\pm 1/t,\pm 2/t,\ldots,\pm (t-1)/t,\pm 1\}$ ($t \in \mathbb{N}$), then consider a square by the grids
\[
    \shikaku := [\tau_1/t,(\tau_1 + 1)/t] \times [\tau_2/t,(\tau_2 + 1)/t]  \subset I^2,
\]
for each $\tau_1,\tau_2 \in \{-t,-t+1,...,-1,0,1,...,t-2,t-1\}$. 

For each $\shikaku$, we choose four points $\nu (\shikaku) := \{\bs x', \bs x'', \bs x''', \bs x''''\} \subset \hat{I}^2$ such that they are vertices of $\shikaku$, and starting from the $\bs x'$ closest to $(1,1)$, we set the other vertices by a clockwise-path $\bs x' \to \bs x'' \to \bs x''' \to \bs x''''$ along with a boundary of $\shikaku$. A set of $\shikaku$ forms a straightforward partition of $I^2$.

The second partition is developed by the first partition and $\bs g_*$. Intuitively, using the level-set representation
$
    \bs g_*(\bs x)=\bs g_*(x_1,I) \cap \bs g_*(I,x_2)
$
in (\ref{eq:grid_representation}) and $\bs g_*(\pm 1,I) = (\pm 1,I), \bs g_*(I,\pm 1)=(I,\pm 1)$ in Lemma \ref{lemma:coherent_rotation}, we consider quadrilaterals in $I^2$ generated by $\{\bs g_*(x_1,I)\}_{x_1}$ and $\{\bs g_*(I,x_2)\}_{x_2}$ as shown in Figure~\ref{fig:gstar} (right). Formally, we define a quadrilateral $\shikakuimage$ corresponding to $\shikaku$ from the first partition as
\begin{align*}
    \shikakuimage := \mathrm{quadrilateral~whose~vertices~are~}\bs g_*(\nu(\shikaku)).
\end{align*}
Figure \ref{fig:interpolation} (right) illustrates the quadrilaterals. A set of $\shikakuimage$ works as a partition of $I^2$, if the quadrilaterals are not twisted (see Remark \ref{remark:twist}). Also, $\shikakuimage$ plays a role of approximation of $\bs g_*(\shikaku) \subset I^2$.

\begin{remark}[Twist of quadrilaterals $\shikakuimage$] \label{remark:twist}
If $\shikakuimage$ is twisted as Figure \ref{fig:twisted_quadrilateral} (left), the partition is not well-defined.
However, when the grids for $\shikaku$ is sufficiently fine, i.e. $t$ is sufficiently large, the twisted quadrilaterals vanish in the sense of the Lebesgue measure (see Figure~\ref{fig:twisted_quadrilateral} (right)). 
Since we will consider $t \to \infty$ as $n$ increases when developing an estimator, an effect of the twisted quadrilaterals are asymptotically ignored in the result of estimation.
Hence, we assume that there is no twist to simplify the discussion.
We provide details of the twist in Appendix~\ref{subsec:twist}. 
\end{remark}


\begin{figure}[!ht]
\centering
\includegraphics[width=0.8\textwidth]{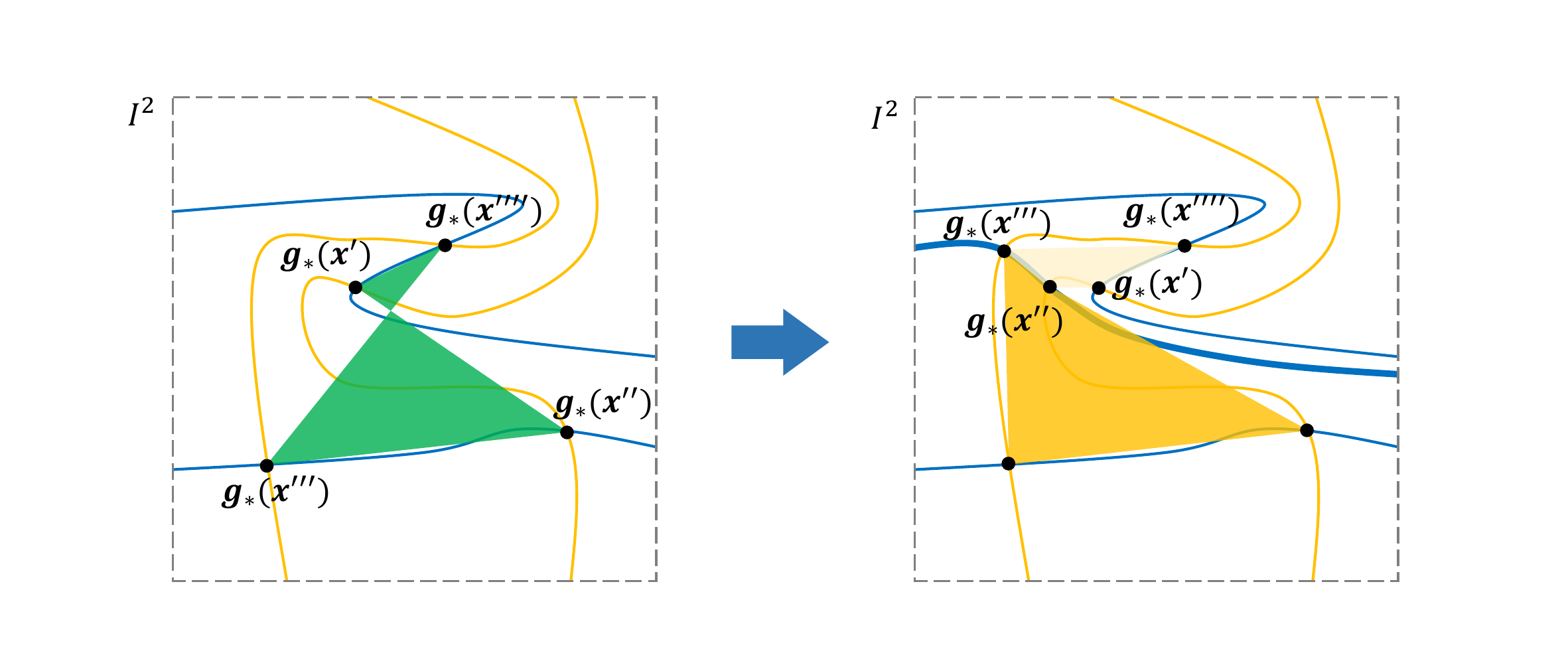}
\caption{
The twisted quadrilateral in $I^2$ (the green region in the left) disappears as the partition by squares become finer (the yellow quadrilateral in the right). The yellow and blue curves are level-sets by $\bs g_*$. As $t$ increases, the twists vanish or become negligibly small.
}
\label{fig:twisted_quadrilateral}
\end{figure}

Using the partitions, we can develop an invertible approximator for $\bs g_*$.
For each $\shikaku$ and its corresponding $\shikakuimage$, we can easily find a local bijective map $\bs g_\shikaku: \shikaku \to \shikakuimage$ (its explicit construction will be provided in Section \ref{sec:proof_upper_d2}). Then, we combine them and define an invertible function $\bs g_*^{\dagger}: I^2 \to I^2$ as $\bs g_*^{\dagger}(\bs x):= \bs g_{\shikaku_{\bs x}} (\bs x)$,
where $\shikaku_{\bs x}$ is a square $\square$ containing $\bs x$. Since the partitions satisfy $\cup_i \shikaku_i = \cup_i \shikakuimage_i=I^2$, $\bs g_*^{\dagger}$ is invertible. Furthermore, $\bs g_*^{\dagger}$ converges to $\bs g_*$ as $t$ increases to infinity. In the following section, we develop an invertible estimator through estimation of $\bs \rho$ and $\bs g_*^{\dagger}$.

\subsection{Invertible Estimator}
\label{subsec:estimator}

We develop an invertible estimator $\hat{\bs f}_n$ by the following two steps: 
(i) we develop estimators $\hat{\bs \rho}_n$ for $\bs \rho$ and $\hat{\bs g}_n^\dagger$ for $\bs g_*^{\dagger}$, by using a pilot estimator $\hat{\bs f}_n^{(1)}$ (e.g., kernel smoother) which is not necessarily invertible but consistent, and 
(ii) we define the developed estimator as $\hat{\bs f}_n :=\hat{\bs \rho}_n^{-1} \circ \hat{\bs g}_n^{\dagger}$. 
In preparation, we first introduce the following assumption on the pilot estimator:
\begin{assumption} \label{asmp:uniform_estimator} 
There exists an estimator $\hat{\bs f}_n^{(1)}: I^2 \to I^2$ and $C > 0$ such that 
\[
    \mathbb{P}\left(\vertinfty{ \hat{\bs f}_{n}^{(1)} - \bs f_{*} } \leq C(n^{-1/(2+d)}(\log n)^{\alpha})\right) \geq 1-\delta_n
\]
holds for sufficiently large $n$, with some $\alpha > 0$ and a sequence $\delta_n \searrow 0$ as $n \to \infty$.
\end{assumption}
\noindent
Several estimators are proved to satisfy this assumption, for example, using a kernel method (\cite{tsybakov2008introduction}), a nearest neighbour method (\cite{devroye1978uniform,devroye1994strong}) and a Gaussian process method (\cite{yoo2016supremum, yang2017frequentist}) with various $\alpha$. In some cases, it is necessary to restrict their ranges to $I^2$ by clipping. Note that this assumption does not guarantee invertibility of $\hat{\bs f}_n^{(1)}$ as follows:
\begin{proposition} \label{prop:inpossibility_first_step}
There exists an estimator $\hat{\bs f}_n^{(1)}$ satisfying Assumption \ref{asmp:uniform_estimator} such that $\priskinv(\hat{\bs f}_n^{(1)}, \bs f_*) > \exists c > 0$ holds for some $\bs f_* \in \flipinv$ and any $n \in \mathbb{N}$.
\end{proposition}

Herein, we develop the invertible estimator $\hat{\bs f}_n$ by leveraging the (not necessarily invertible) pilot estimator $\hat{\bs f}_n^{(1)}$ as follows:
\begin{enumerate}
\item[(i-a)]
\textbf{Estimator for $\bs \rho$}: We develop the invertible estimator $\hat{\bs \rho}_n$ for $\bs \rho$, such that $\hat{\bs \rho}_n(\hat{\bs f}_n^{(1)}(\tilde{\bs x}_j)) \approx \tilde{\bs x}_j$ for all vertices $\tilde{\bs x} \in \{\pm 1\} \times \{\pm 1\}$ of $I^2$ (formal definition of $\hat{\bs \rho}_n$ is provided in Appendix~\ref{subsec:estimator_rho}). 

\item[(i-b)] \textbf{Estimator for $\bs g_*$}: 
We define an estimator $\hat{\bs g}_n(\bs x):=\mathfrak{P}\hat{\bs \rho}_n(\hat{\bs f}_n^{(1)}(\bs x))$ for $\bs g_*$, where $\mathfrak{P}$ constrains $\hat{\bs g}_n(\bs x)$ to an edge of the range $I^2$, when $\bs x$ is an endpoint of the domain $I^2$: $\mathfrak{P}$ replaces $\tilde{y}_1$ in $\tilde{\bs y}=(\tilde{y}_1,\tilde{y}_2):=(\hat{\bs \rho}_n \circ \hat{\bs f}^{(1)}_n)(\bs x)$ with $\pm 1$ if $x_1=\pm 1$ and $\tilde{y}_2$ with $\pm 1$ if $x_2=\pm 1$. This operator $\mathfrak{P}$ is necessary for making $\hat{\bs g}_n$ to have a range $I^2$. Note that $\hat{\bs g}_n$ is not always invertible.

\item[(i-c)] \textbf{Invertible estimator for ${\bs g}_*^\dagger$}: We develop an invertible estimator $\hat{\bs g}_n^\dagger$ for ${\bs g}_*^\dagger$ by estimating the partition $\shikakuimage$ using $\hat{\bs g}_n$. For $\bs x \in I^2$, let $\shikaku = \shikaku_{\bs x}$ be a square containing $\bs x$ and $\nu(\shikaku) = \{\bs x',\bs x'',\bs x''',\bs x''''\}$ be a set of its vertices, and we estimate its corresponding quadrilateral $\shikakuimage$ by its estimator $\hat{\shikakuimage}$ using $\hat{\bs g}_n(\nu(\shikaku))$. Then, we develop a bijective map between $\shikaku$ and $\hat{\shikakuimage}$. 
Let $\bs s:=(\bs x'+\bs x''+\bs x'''+\bs x'''')/4 \in I^2$ and suppose $\bs x',\bs x''$ are the two closest vertices from $\bs x$. As there exists a unique $(\alpha',\alpha'') \in [0,1]^2$ satisfying $\alpha'+\alpha'' \le 1$ and $\bs x = \bs s + \alpha' \{\bs x'-\bs s\} + \alpha'' \{\bs x''-\bs s\}$, we define an estimator $\hat{\bs g}_n^{\dagger}$ for ${\bs g}_*^{\dagger}$ on $\hat{\shikakuimage}$ by a {triangle interpolation}:
\begin{align}
\hat{\bs g}_n^{\dagger}(\bs x):=\hat{\bs g}_n(\bs s) + \alpha' \{\hat{\bs g}_n(\bs x')-\hat{\bs g}_n(\bs s)\} + \alpha'' \{\hat{\bs g}_n(\bs x'')-\hat{\bs g}_n(\bs s)\}
\label{eq:interpolation}
\end{align}
if the quadrilateral $\hat{\shikakuimage}$ is not twisted. 
$\hat{\bs g}_n^{\dagger}$ is bijective within each pair of $\shikaku$ and $\hat{\shikakuimage}$, hence it is therefore invertible. See Figure~ \ref{fig:triangle_interpolation} for illustration. 
Note that, the quadrilateral $\hat{\shikakuimage}$ can be twisted: for the twisted quadrilateral (as an exceptional case), we define $\hat{\bs g}_n^{\dagger}(\bs x):=\hat{\bs g}_n(\bs x')$.

\item[(ii)] \textbf{Invertible estimator for ${\bs f}_*$}:
We define the estimator for $\bs f_*$ as
\[
    \hat{\bs f}_n
    :=
    \hat{\bs \rho}_n^{-1} \circ \hat{\bs g}_n^{\dagger}.
\]
Since $\hat{\bs \rho}_n$ and $\hat{\bs g}_n^{\dagger}$ are invertible, the invertibility of $\hat{\bs f}_n$ is assured. 
\end{enumerate}

\begin{figure}
\includegraphics[width=0.9\textwidth]{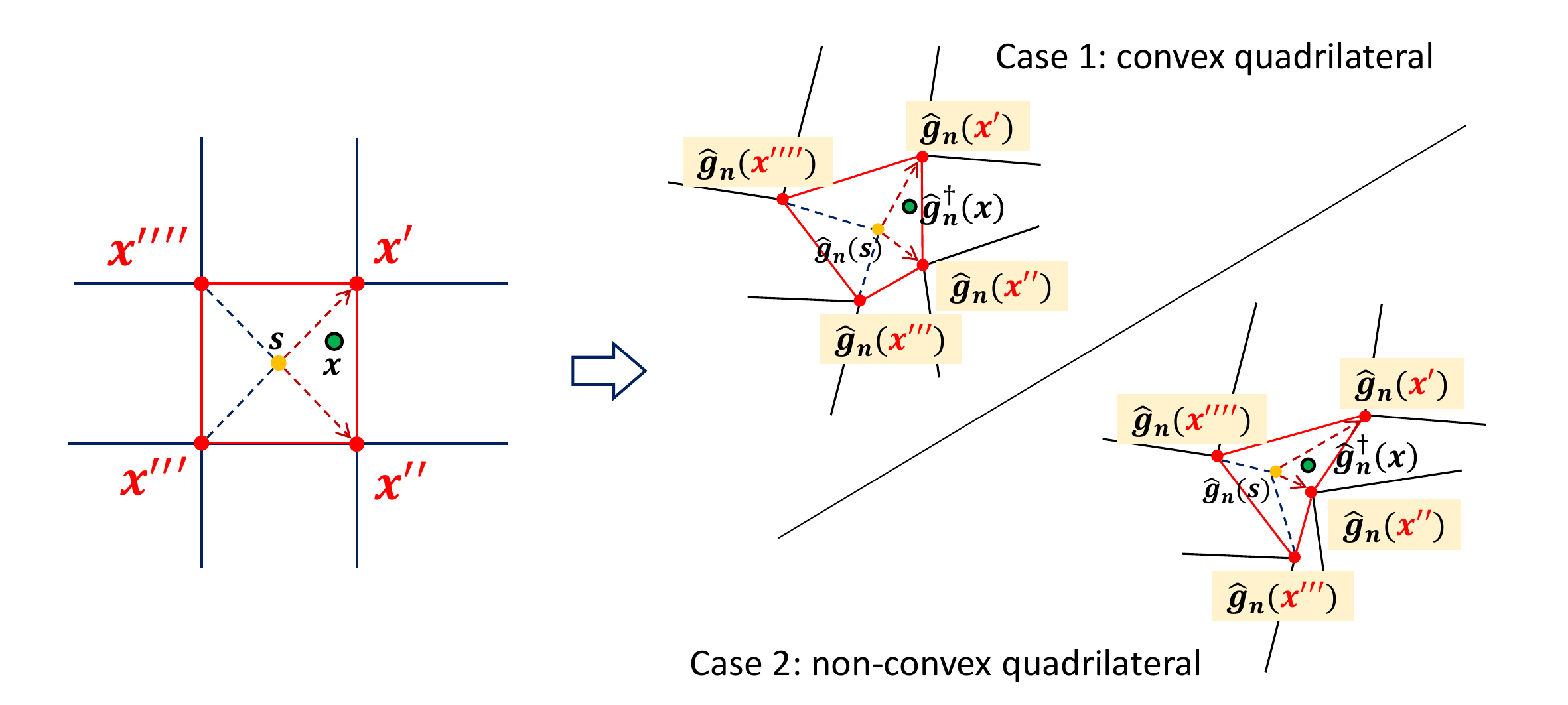}
\caption{Triangle interpolation $\hat{\bs g}_n^{\dagger}(\bs x):=\hat{\bs g}_n(\bs s) + \alpha' \{\hat{\bs g}_n(\bs x')-\hat{\bs g}_n(\bs s)\} + \alpha'' \{\hat{\bs g}_n(\bs x'')-\hat{\bs g}_n(\bs s)\}$ for $\bs x = \bs s + \alpha' \{\bs x'-\bs s\} + \alpha'' \{\bs x''-\bs s\}$. 
}
\label{fig:triangle_interpolation}
\end{figure}

\subsection{Numerical Demonstration of the Developed Estimator:} 
We experimentally demonstrate the developed estimator $\hat{\bs f}_n$. We set a true function 
\[
    \bs f_*(\bs x):=\bs v(\|\omega(\bs x)\|_2^{|\sin(\vartheta(\omega(\bs x)))|},\vartheta(\omega(\bs x))) \in \flipinv
\]
where the functions $\omega,\vartheta,\bs v$ are defined in Appendix~\ref{subsec:symbols}. We generated $n=10^4$ covariates $\bs x_i \overset{\text{i.i.d.}}{\sim} U(I^2)$ and outcomes $\bs y_i \overset{\text{i.i.d.}}{\sim} N(\bs f_*(\bs x_i),\sigma^2\bs I_2)$, and conduct the above estimation procedure with $\sigma^2 \in \{10^{-3}, 10^{-1}\}$. Especially, we employed $k$-nearest neighbor regression ($k=10$, clipped to restrict the range to $I^2$) for the pilot estimator $\hat{\bs f}_n^{(1)}$.
We note that we use bi-linear interpolation for calculating $\hat{\bs g}_n^{\dagger}$, which coincides with the triangle interpolation \eqref{eq:interpolation} in this setting.

\begin{figure}[!t]
    \centering
    \includegraphics[width=\textwidth]{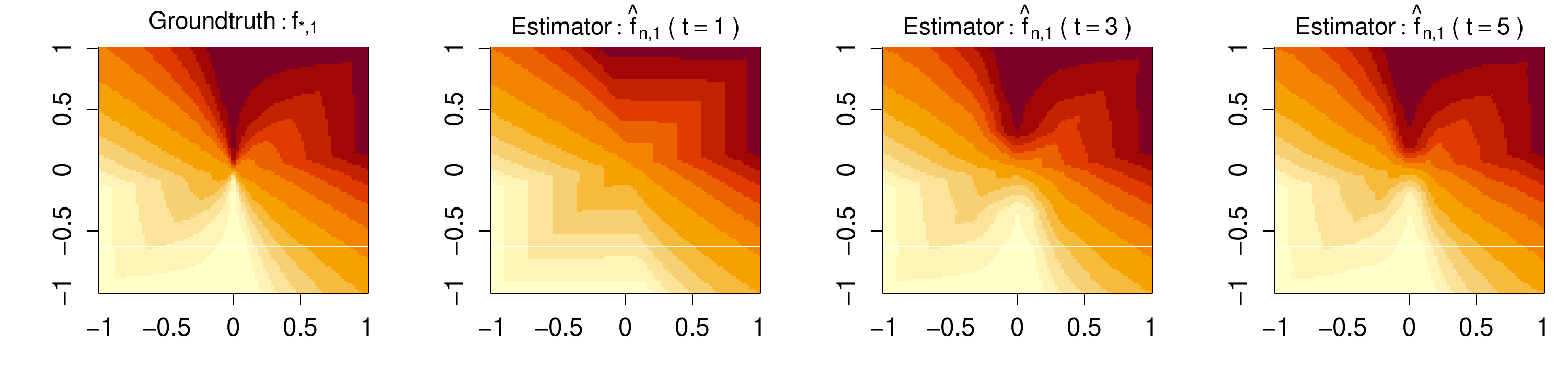}
    \includegraphics[width=\textwidth]{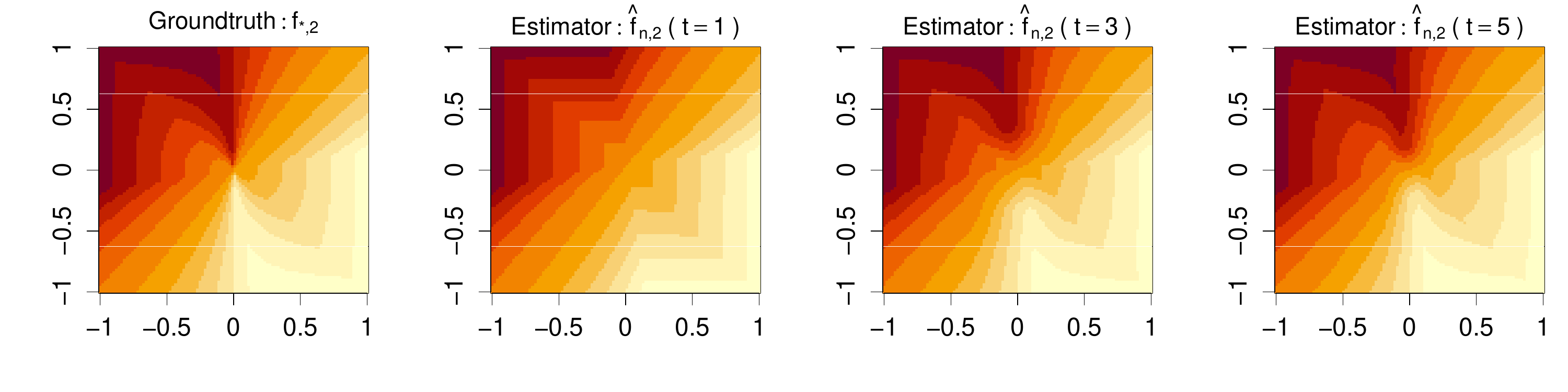}
    \caption{Heatmap of the true function $\bs f_*=(f_{*,1},f_{*,2})$ and its invertible estimator $\hat{\bs f}_n=(\hat{f}_{n,1},\hat{f}_{n,2})$ with $t=1,3,5$.}
    \label{figestimators_heatmap_for_different_t}
\end{figure}

We plot the heatmaps of $\bs f_*=(f_{*,1},f_{*,2})$ and $\hat{\bs f}_n=(\hat{f}_{n,1},\hat{f}_{n,2})$ for $t \in \{1,3,5\}$ and $\sigma^2 = 10^{-3}$ in Figure \ref{figestimators_heatmap_for_different_t}. We can see that $\hat{\bs f}_n$ approaches $\bs f_*$ as $t$ increases. 
We further plot the heatmaps of $f_{*,1} $, $\hat{f}_{n,1}^{(1)}$, $\hat{g}_{n,1}$, $\hat{g}_{n,1}^{\dagger}$ and $\hat{f}_{n,1}$ with $t=3$ and $\sigma^2 \in \{10^{-3}, 10^{-1}\}$, in Figure~\ref{fig:estimators_heatmap}. We can verify that (i) $\hat{g}_{n,1}$ and $\hat{g}_{n,1}^{\dagger}$ have level-sets aligned to the endpoints of $I^2$, and (ii) $\hat{g}_{n,1}^{\dagger}$ and $\hat{f}_{n,1}$ have level-sets with fewer slopes than those of $\hat{f}_{n,1}^{(1)}$ and $\hat{g}_{n,1}$, which is suitable for invertibility.

\begin{figure}[!ht]
\centering
$\sigma^2 = 10^{-3}$

\begin{minipage}{0.18\textwidth}
\includegraphics[width=\textwidth]{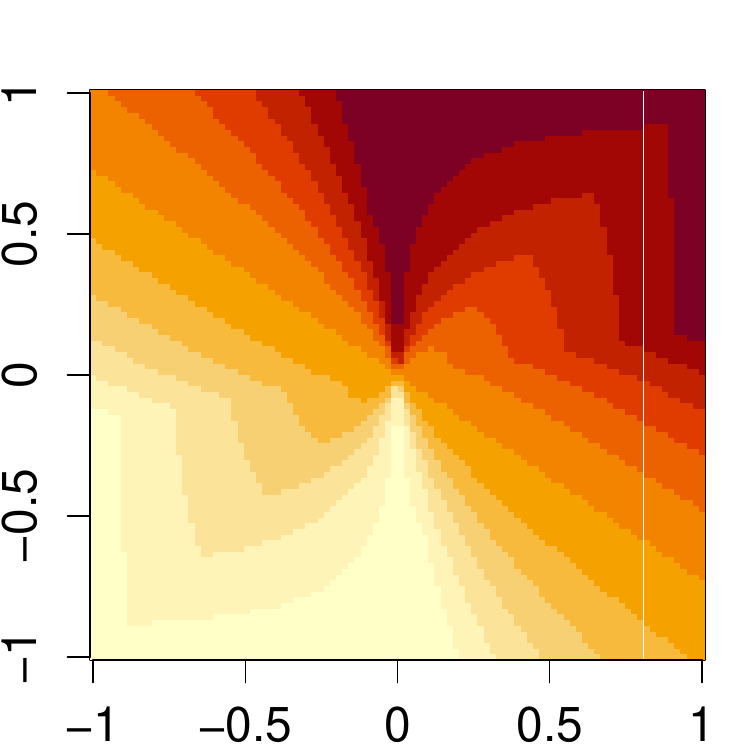}
\subcaption{$f_{*,1}$}
\label{fig:groundtruth}
\end{minipage}
\vline
\hspace{0.5em}
\begin{minipage}{0.18\textwidth}
\includegraphics[width=\textwidth]{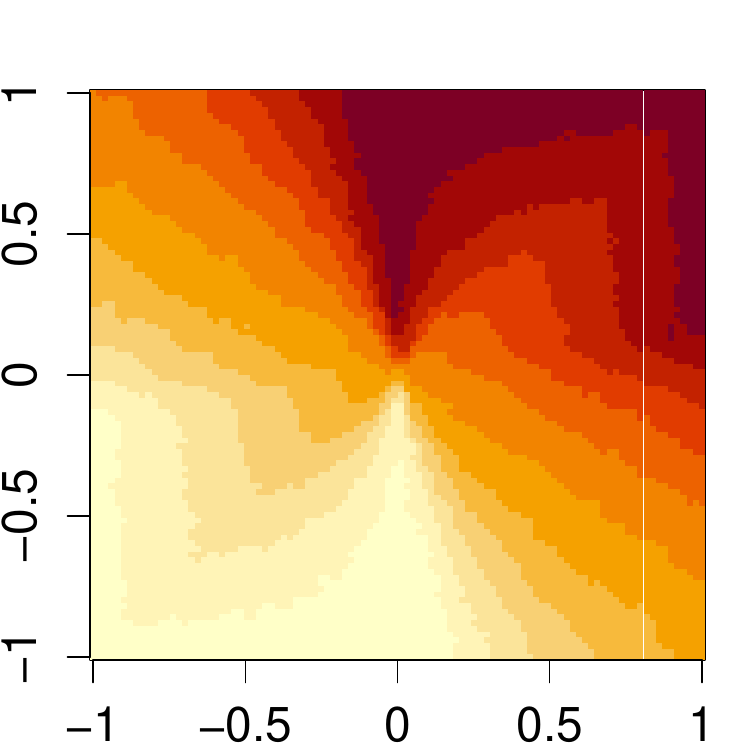}
\subcaption{$\hat{f}^{(1)}_{n,1}$}
\label{fig:first_step}
\end{minipage}
\begin{minipage}{0.18\textwidth}
\includegraphics[width=\textwidth]{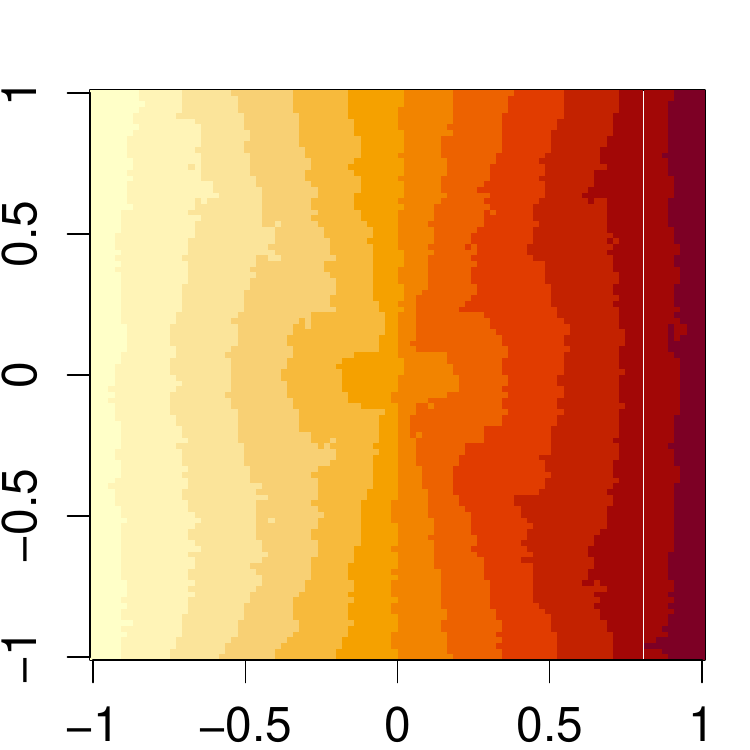}
\subcaption{$\hat{g}_{n,1}$}
\label{fig:rotated}

\end{minipage}
\begin{minipage}{0.18\textwidth}
\includegraphics[width=\textwidth]{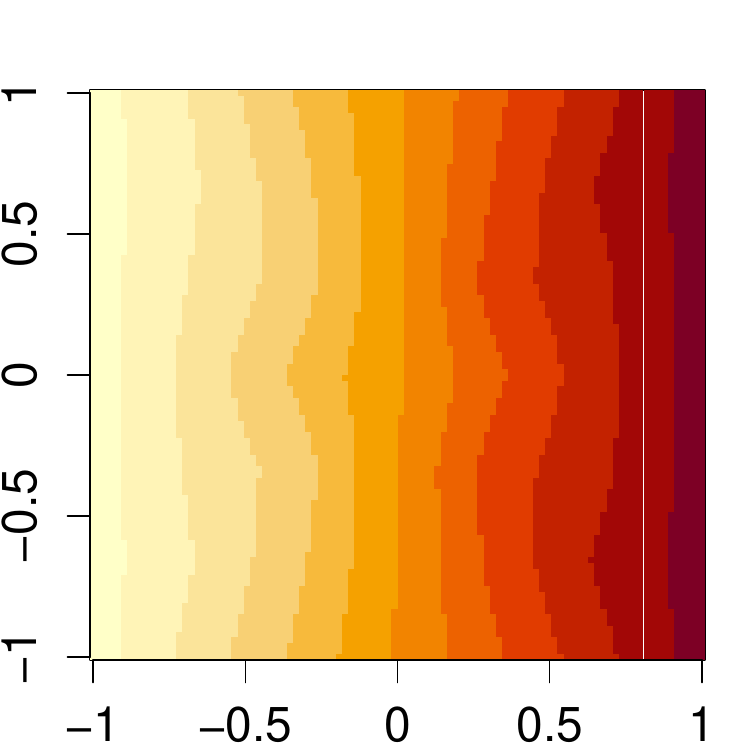}
\subcaption{$\hat{g}_{n,1}^{\dagger}$}
\label{fig:smoothed}
\end{minipage}
\vline \vspace{0.5em}
\begin{minipage}{0.18\textwidth}
\includegraphics[width=\textwidth]{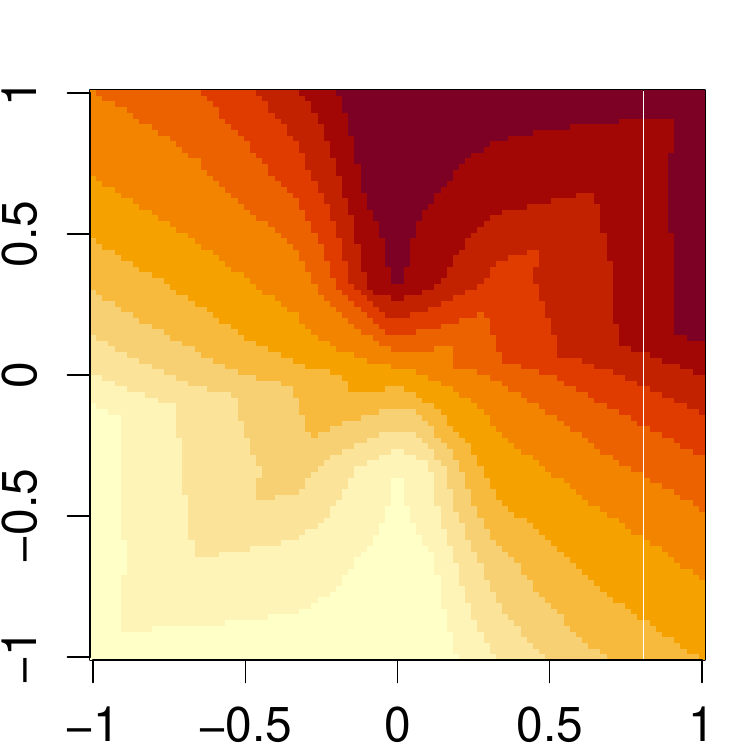}
\subcaption{$\hat{f}_{n,1}$}
\label{fig:second_step}
\end{minipage}
$\sigma^2 = 10^{-1}$

\begin{minipage}{0.18\textwidth}
\includegraphics[width=\textwidth]{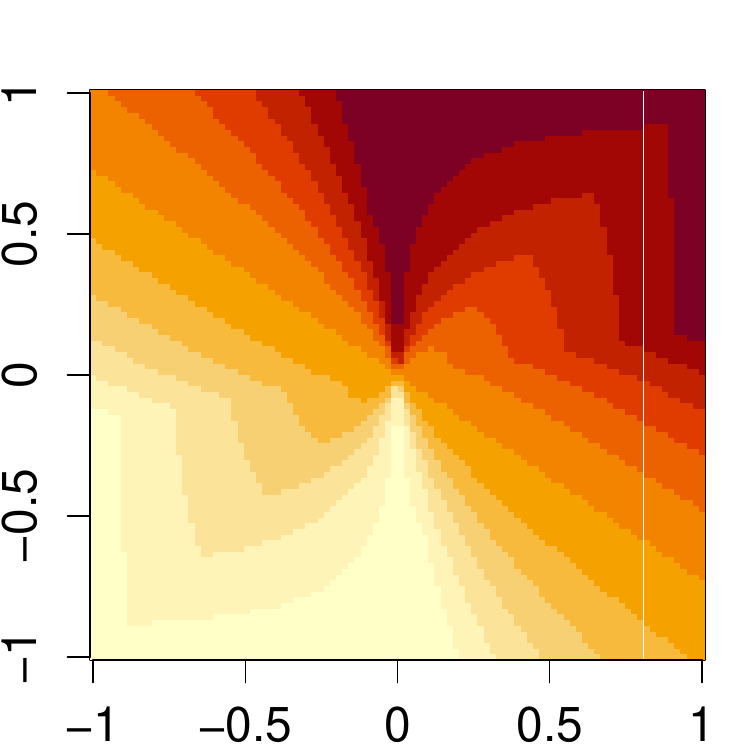}
\subcaption{$f_{*,1}$}
\label{fig:v01_groundtruth}
\end{minipage}
\vline
\hspace{0.5em}
\begin{minipage}{0.18\textwidth}
\includegraphics[width=\textwidth]{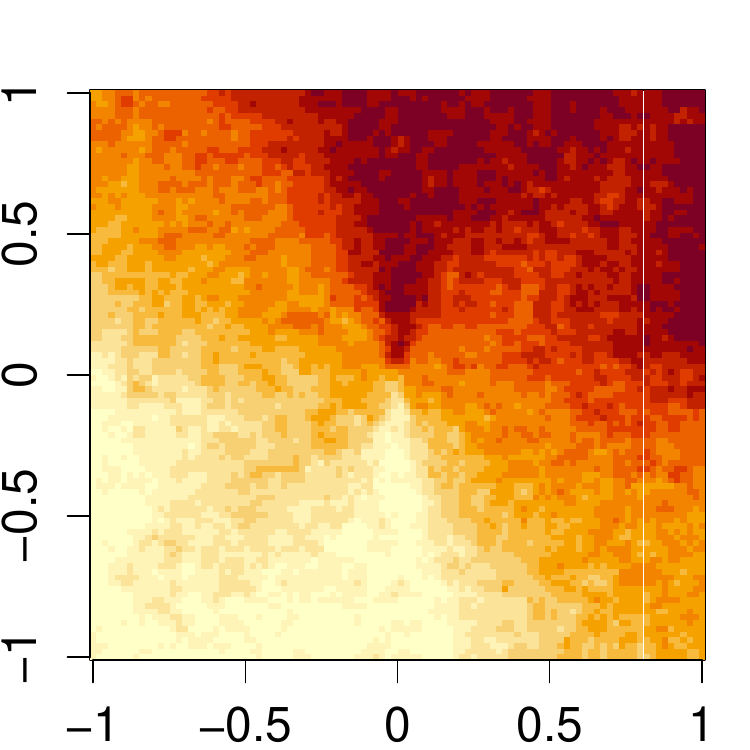}
\subcaption{$\hat{f}^{(1)}_{n,1}$}
\label{fig:v01_first_step}
\end{minipage}
\begin{minipage}{0.18\textwidth}
\includegraphics[width=\textwidth]{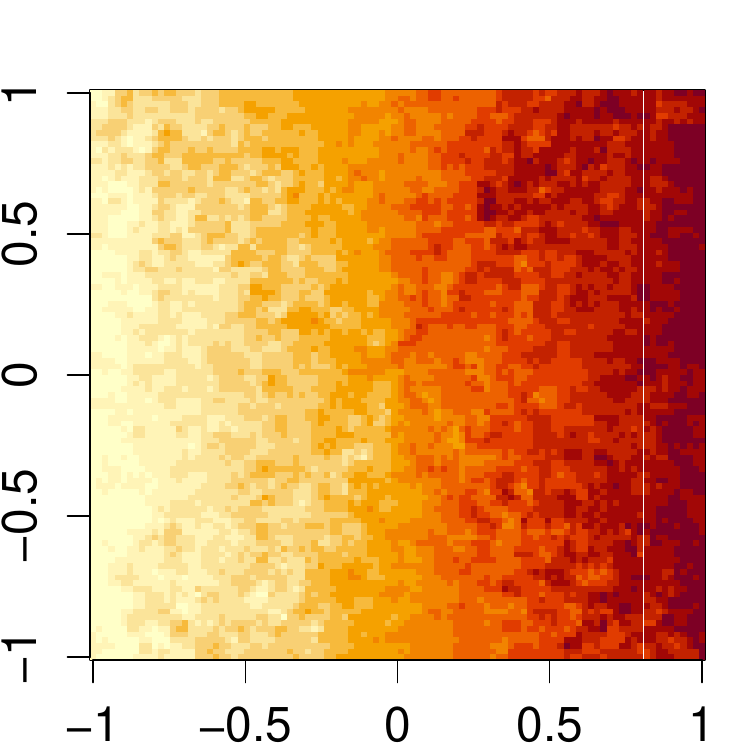}
\subcaption{$\hat{g}_{n,1}$}
\label{fig:v01_rotated}

\end{minipage}
\begin{minipage}{0.18\textwidth}
\includegraphics[width=\textwidth]{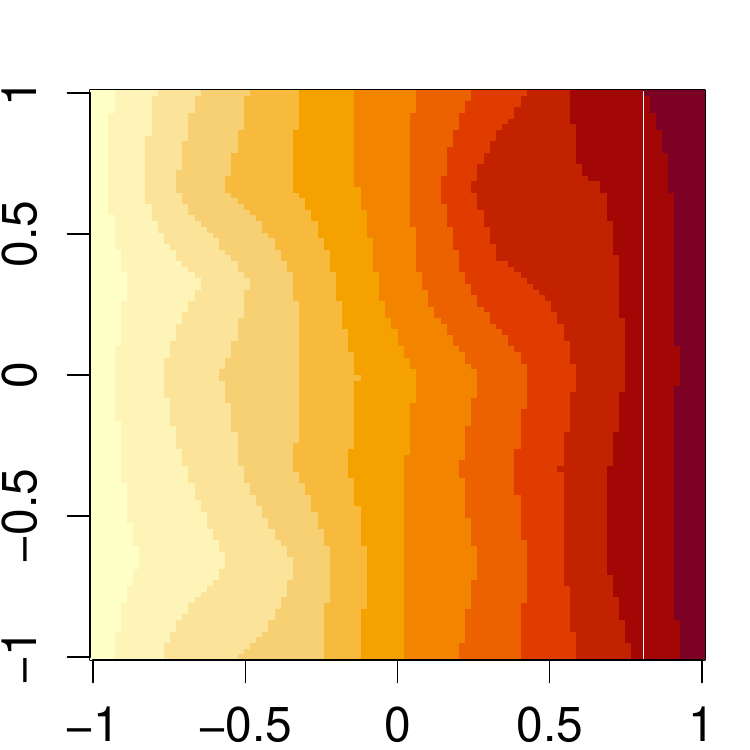}
\subcaption{$\hat{g}_{n,1}^{\dagger}$}
\label{fig:v01_smoothed}
\end{minipage}
\vline \vspace{0.5em}
\begin{minipage}{0.18\textwidth}
\includegraphics[width=\textwidth]{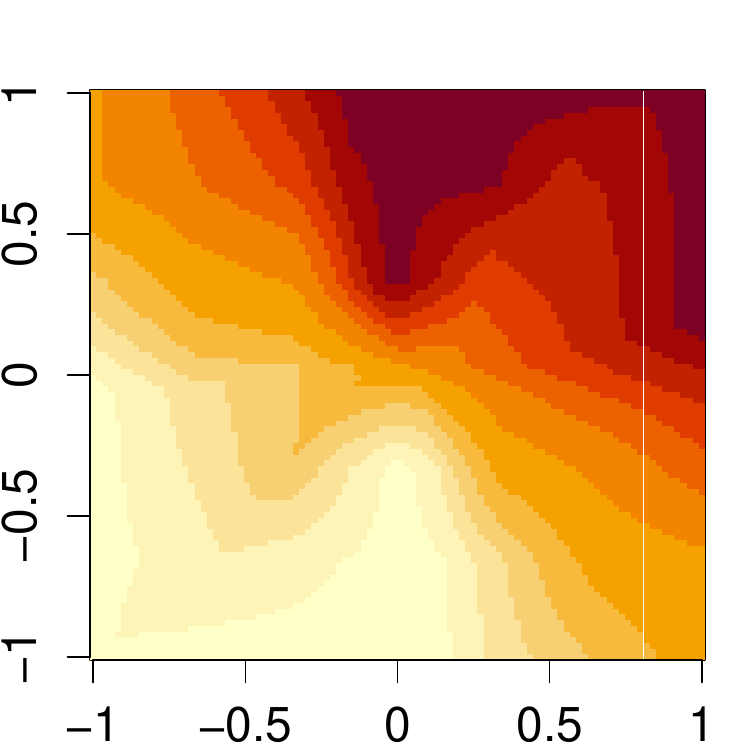}
\subcaption{$\hat{f}_{n,1}$}
\label{fig:v01_second_step}
\end{minipage}
\caption{(\subref{fig:groundtruth},\subref{fig:v01_groundtruth}) True function $\bs f_{*,1}$, (\subref{fig:first_step},\subref{fig:v01_first_step}) pilot estimator $\hat{f}_{n,1}^{(1)}$, (\subref{fig:rotated},\subref{fig:v01_rotated}) estimator $\hat{g}_{n,1}$ transformed by a coherent rotation, 
(\subref{fig:smoothed},\subref{fig:v01_smoothed}) invertible estimator $\hat{g}_{n,1}$ using biniliear interpolation, and 
(\subref{fig:second_step},\subref{fig:v01_second_step}) invertible estimator $\hat{f}_{n,1}$. The upper row is $\sigma^2 = 10^{-3}$ and the lower row is $\sigma^{2} = 10^{-1}$.}
\label{fig:estimators_heatmap} 
\end{figure}

\verb|R| source codes to reproduce the experimental results are provided in \url{https://github.com/oknakfm/NPIR}.

\section{Conclusions and Future Research Directions}
We studied the nonparametric planer invertible regression, which estimates invertible and bi-Lipschitz function $\bs f_* \in \flipinv$ between a closed square $[-1,1]^2$. For $d=2$, we defined inverse risk to evaluate the invertible estimators $\hat{\bs f}_n$: the minimax rate is lower bounded by $n^{-2/(2+d)}$. We developed an invertible estimator, which attains the lower bound up to logarithmic factors. This result implies that the estimation of invertible functions is as difficult as the estimation of non-invertible functions in the minimax sense. For this evaluation, we employed output-wise level-sets $L_{f_j}(y):=\{\bs x \in I^2 \mid f_j(\bs x)=y\}$ of the invertible function $\bs f=(f_1,f_2)$, as their intersection $L_{f_1}(y_1) \cap L_{f_2}(y_2)$ identifies the inverse $\bs f^{-1}(\bs y)$. We identified some important properties of the level-set $L_{f_j}$. This study is the first step towards understanding the multidimensional invertible function estimation problem.

However, there remain unsolved problems. For example,
\begin{enumerate}[{(i)}]
    \item We developed an invertible estimator only for a restricted case, $d=2$. A natural direction would be to extend our estimator and the minimax upper bound of the inverse risk to the general $d \ge 3$. However, theoretical extension to general $d \ge 3$ seems not straightforward by the following two reasons: (i) coherent rotation, which is used to align the endpoints in our estimator, cannot be defined even for $d=3$ and (ii) \citet{donaldson1989quasiconformal} proved that bi-Lipschitz homeomorphisms cannot be approximated by even piecewise Affine functions for $d=4$. Some additional assumptions seem needed. 
    Another ongoing work of ours studies the case $d \in \mathbb{N}$, by additionally imposing $C^2$ smoothness on $\bs f_*$ to eliminate the pathological cases.
    
    \item The discussions in this paper mostly rely on (a) the existence of the boundary and (b) the simple connectivity of set $[-1,1]^2$. It would be worthwhile to generalize our discussion to different types of domains, such as the open multidimensional unit cube $(-1,1)^2$ (e.g., \citet{kawamura1979invertibility} and \citet{pourciau1988global} for a characterization of nonsmooth invertible mappings between $\mathbb{R}^2$, where, $\mathbb{R}^2$ and $(-1,1)^2$ are homeomorphic) and some sets with different torus (see, \citet{hatcher2002topology} for the gentle introduction to torus, and \citet{pmlr-v119-rezende20a} for normalizing flow on tri and sphere surface). 
    \item {It is an important attempt to relax the bi-Lipschitz continuity setting. 
    In particular, omitting the restriction of lower-Lipschitz property is important. 
    If we omit the restriction, we can handle a wider class of functions such as polynomials.}
    
    \item Whereas the minimax rate is obtained for a supervised regression problem, one of the main applications of the multidimensional invertible function estimation is density estimation, which implicitly trains the invertible function in an unsupervised manner. 
    An interesting direction would be to extend the minimax rate to an unsupervised setting. 
\end{enumerate}

\section*{Acknowledgement}
A. Okuno is supported by JSPS KAKENHI (21K17718, 22H05106) and JST CREST (JPMJCR21N3). 
M. Imaizumi is supported by JSPS KAKENHI (18K18114) and JST Presto (JPMJPR1852). We would like to thank Keisuke Yano for the helpful discussion. 

\appendix

\section{Supporting Lemmas}
\label{sec:supporting_lemmas}

\begin{lemma}
\label{lem:inverse_of_lipschitz}
Let $\bs f:I^2 \to I^2$ be an invertible function and let $L>0$. The following statements are equivalent:
\begin{enumerate}[{(i)}] 
\item $\|\bs f^{-1}(\bs y)-\bs f^{-1}(\bs y')\|_2 \le L \|\bs y-\bs y'\|_2$ for any $\bs y,\bs y' \in I^2$, 
\item $L^{-1}\|\bs x-\bs x'\|_2 \le \|\bs f(\bs x)-\bs f(\bs x')\|_2$ for any $\bs x,\bs x' \in I^2$. 
\end{enumerate}
\end{lemma}

\begin{proof}[Proof of Lemma~\ref{lem:inverse_of_lipschitz}]
For any $\bs x,\bs x' \in I^2$, substituting $\bs y:=\bs f(\bs x),\bs y':=\bs f(\bs x')$ to (i) yields the inequality (ii). Conversely, substituting $\bs x:=\bs f^{-1}(\bs y),\bs x':=\bs f^{-1}(\bs y')$ to (ii) yields (i). 
\end{proof}

Lemma~\ref{lem:inverse_of_lipschitz} immediately proves the following Lemma~\ref{lem:bi_lipschitz}.

\begin{lemma}
\label{lem:bi_lipschitz}
    Let $\bs f:I^2\to I^2$ be an invertible function. 
    Both $\bs f,\bs f^{-1}$ are Lipschitz if and only if 
    $\bs f$ is bi-Lipschitz, i.e., there exists $L\geq 1$ such that $L^{-1}\|\bs x-\bs x'\|_2 \le \|\bs f(\bs x)-\bs f(\bs x')\|_2 \le L\|\bs x-\bs x'\|_2$. 
\end{lemma}

\begin{lemma}
\label{lemma:Hausdorff_Lipschitz}
It holds for $\bs f \in \flipinv$ that $\dhauss(L_{f_j}(y),L_{f_j}(y')) \le C |y-y'| \quad (y,y' \in I;j=1,2)$, for some $C \in (0,\infty)$. 
\end{lemma}
\begin{proof}[Proof of Lemma~\ref{lemma:Hausdorff_Lipschitz}]
Considering the representation in Lemma \ref{lem:parameterization}, the Lipschitz property of $\bs f^{-1}$ proved in Lemma~\ref{lem:inverse_of_lipschitz} leads to
\begin{align*}
    \dhauss(&L_{f_1}(y),L_{f_1}(y')) 
    \leq 
    2
    \sup_{\alpha \in I}
    \|\bs f^{-1}(y,\alpha) 
    - \bs f^{-1}(y',\alpha)\|_2
    \leq
    C|y-y'|
\end{align*}
for some $C \in (0,\infty)$, and $\dhauss(L_{f_2}(y),L_{f_2}(y')) \le C|y-y'|$ is proved in the same way.
\end{proof}

\begin{lemma}
\label{lem:homeomorphism_boundary}
Let $X,Y \subset \mathbb{R}^2$ be non-empty closed topological spaces and let $\bs f:X \to Y$ be a homeomorphism, i.e., invertible and continuous function. 
Then, $\bs f(\partial X)=\partial Y$.
\end{lemma}

\begin{proof}[Proof of Lemma~\ref{lem:homeomorphism_boundary}]
As it suffices to prove that $\bs f(\internal X)=\internal Y$ with $\internal X:=X \setminus \partial X$, 
(i) we first prove that $\bs f(\internal X) \subset \internal Y$. 
Taking any $\bs x \in \bs f(\internal X)$, we have $\bs f^{-1}(\bs x) \in \internal X$, i.e., $\bs f^{-1}(\bs x) \in U \subset X$ for some open-neighbourhood $U$ of $\bs x$. Applying $\bs f$ yields $\bs x \in \bs f(U) \subset \bs f(X)=Y$. As $\bs f$ is continuous, $f(U)$ is also open neighbourhood of $\bs x$: we have $\bs x \in \internal Y$, indicating that $\bs f(\internal X) \subset \internal Y$. 
(ii) We next prove that $\internal Y \subset \bs f(\internal X)$. Take $\bs y \in \internal Y$, i.e., $\bs y \in V \subset Y$ for some open-neighbourhood $V$ of $\bs y$. Applying $\bs f^{-1}$ yields $\bs f^{-1}(\bs y) \in \bs f^{-1}(V) \subset \bs f^{-1}(Y)=X$, indicating that $\bs f^{-1}(\bs y) \in \internal X$. Thus applying $\bs f$ proves $\bs y \in \bs f(\internal X)$, and we have $\internal Y \subset \bs f(\internal X)$.
\end{proof}

\section{Proofs for Lower Bound Analysis}

\label{sec:proofs_for_lower_bound_analysis}

\subsection{Preliminaries}
\label{subsec:preliminaries_for_lower_bound_analysis}

\begin{lemma}
\label{lem:xi_theta_monotone}
Fix any $k \in \{1,2\}, x_{\ell} \in I, \ell \in \{1,2\} \setminus \{k\}$ and $\xi_{\theta} \in \Xi^2_k=(\ref{eq:Xidk})$. 
The following hold for $\tilde{\xi}_{\theta}(x_k):=\xi_{\theta}(\bs x)$:

\begin{enumerate}[{(i)}]
\item $\tilde{\xi}_{\theta}$ is strictly increasing in $x_k$, i.e., $\tilde{\xi}_{\theta}(x_k) < \tilde{\xi}_{\theta}(x_k')$ for any $-1 \le x_k < x_k' \le 1$, and 
\item $\tilde{\xi}_{\theta}$ is surjective from $I$ to $I$. 
\end{enumerate}
\end{lemma}

\begin{proof}[Proof of Lemma~\ref{lem:xi_theta_monotone}] 
We prove (i) and (ii) as follows.

\begin{enumerate}[{(i)}]
\item We prove that a continuous function $\tilde{\chi}_{\theta}(x_k)=\chi_{\theta}(x_k,x_\ell)$ is piecewise linear whose slopes are greater than $-1$, as it immediately proves the strict monotonicity of $\tilde{\xi}_{\theta}(x_k)=x_k+\tilde{\chi}_{\theta}(x_k)$. 

Let $j \in [m]$ and let $j_\ell^* \in [m]$ be a minimum index satisfying $|x_{\ell}-t_{j_\ell^*}| \le 1/m$, for $l \in \{1,2\} \setminus \{k\}$. 
Recall the function $\Phi(\bs x)$ defined in Section 3.2.
Consider a function 
\begin{align*}
    \phi_j(x_k)
    :=
    \Phi\big(
        m(x_k-t_{j}),
        m(x_{\ell}-t_{j_{\ell}^*})
    \big)
    \: \text{ with a constant} \: 
    \tilde{\phi}:=\phi_j(t_j)=\Phi\left(
        0,m(x_{\ell}-t_{j_{\ell}^*})
    \right),
\end{align*}
and split $I$ into $I_j:=[\tilde{t}_{j-1},\tilde{t}_j)$ ($j=1,2,\ldots,m-1$) and $I_m:=[\tilde{t}_{m-1},1]$ with $\tilde{t}_j:=-1+2j/m$; an explicit formula of this function $\phi_j:I \to [0,1]$ is
\[
    \phi_j(x_k)
    =
    \begin{cases} 
    m(x_k - \tilde{t}_{j-1}) & (x_k \in [\tilde{t}_{j-1}, \, \tilde{t}_{j-1}+\tilde{\phi}/m) \\
    \tilde{\phi} & (x_k \in [\tilde{t}_{j-1}+\tilde{\phi}/m, \, \tilde{t}_j-\tilde{\phi}/m)) \\
    m(\tilde{t}_j-x_k) & (x_k \in [\tilde{t}_j-\tilde{\phi}/m, \, \tilde{t}_j)) \\
    0 & (\text{Otherwise.})
    \end{cases}.
\]
By the definition of the function $\Phi$, we have 
\begin{align}
    \tilde{\chi}_{\theta}(x_k)
    =
    \sum_{j=1}^m\frac{\theta_{j,j_{\ell}^*}}{M}
    \phi_j(x_k)
    \label{eq:explicit_chi}
\end{align}
is piecewise linear with slopes $m/M,0,-m/M$. Recalling that $M>2m$ and $\theta_{j_1,j_2} \in \{0,1\}$, the slope of $\tilde{\chi}_{\theta}$ is greater than $-1$ (and is less than $1$). 
The assertion (i) is proved. 
\item The explicit formula (\ref{eq:explicit_chi}) yields $\tilde{\xi}_{\theta}(-1)=-1$ and $\tilde{\xi}_{\theta}(1)=1$, and the strict monotonicity (shown in (i)) proves that $\tilde{\xi}_{\theta}:I \to I$ is surjective. The assertion (ii) is proved. 
\end{enumerate}
\end{proof}

\begin{lemma}[Varshamov--Gilbert bound; Lemma 2.9 in \cite{tsybakov2008introduction}] \label{lem:vg}
    Let $m \geq 8$ and $\theta^{(0)} = (0,0,...,0) \in \Theta_m$.
    Then, there exists $\{\theta^{(0)},...,\theta^{(M)}\} \subset \Theta_m$ such that $M \geq 2^{m/8}$ and $H(\theta^{(j)}, \theta^{(k)}) \geq m/8$ for any $0 \leq j \neq k \leq M$.
\end{lemma}

\begin{lemma}[Theorem 2.5 in  \cite{tsybakov2008introduction}] 
\label{thm:lower}
Let $F$ be a set of functions from $\mathbb{R}^d$ to $\mathbb{R}$, and $\{f_0,  f_1,..., f_M\} \subset F$ be its subset of size $M + 1 \geq 3$.
Let $P_j$ be a probability measure indexed by $ f_j \in F$, and suppose that $P_ j,j=1,...,M$ are absolutely continuous with respect to $P_0$.
Then, we have
\begin{align*}
    \inf_{\hat{ f}_n} \sup_{ f \in F}\mathbb{P}\left( \| \hat{ f}_n -  f\|_{L^2}^2 \geq \alpha \right) \geq \frac{\sqrt{M}}{1 + \sqrt{M}} \left(1 - 2 \beta - \sqrt{\frac{2 \beta}{\log M}} \right),
\end{align*}
where $\alpha = \min_{j\neq k}\| f_j -  f_k\|_{L^2}^2 / 2 >0$ and $\beta = (M\log M)^{-1}  \sum_{j = 1}^M \mathrm{KL}(P_j, P_0)\in (0,1/8)$.
The infimum is taken over all estimators which depend on the $n$ observations.
\end{lemma}

\subsection{Proof of Proposition~\ref{prop:FXi_ivnertible}}
\label{subsec:proof_lem:FXi_invertible}
As the bi-Lipschitz property is straightforwardly proved, we describe the invertibility of $\bs f=(f_1,f_2)\in \f(\{\Xi^2_k\}_k)$ in this proof. 
Considering Propositon~\ref{prop:equiv_invertible_levelset}, $\bs f$ is invertible if and only if its level-set representation
\begin{align*}
    \bs f^{\dagger}(\bs y)
    =
    \bigcap_{j=1}^{2} L_{f_j}(y_j)
\end{align*}
shown in (\ref{eq:inverse_intersection}) is a unique point for every $\bs y=(y_1,y_2) \in I^2$. Therefore, in this proof, it is sufficient to prove the uniqueness of $\bs f^{\dagger}$.

We first examine the level-set $L_{f_1}(y_1)$. Using a function $\iota(\bs x):=x_1$ and $\theta \in \Theta_m^{\otimes 2}$ such that $f_1=\xi_{\theta}=\iota+\chi_{\theta} \in \Xi^2_1$, we have
\[
    L_{f_1}(y_1)
    =
    \bigcup_{y_1'+y_1'' = y_1}
    \left(
        L_{\iota}(y_1') \cap L_{\chi_{\theta}}(y_1'')
    \right)
    =
    \bigcup_{y_1' \in [0,1/M]}
    \left(
        L_{\iota}(y_1 - y_1') \cap L_{\chi_{\theta}}(y_1')
    \right)
\]
where the last equality follows from $\chi_{\theta}(\bs x) \in [0,1/M]$. Considering $L_{\iota}(y)=\{(y,x_2) \mid x_2 \in I\}$ and the surjectivity of the functions $\chi_{\theta}$ proved in Lemma~\ref{lem:xi_theta_monotone}, the level-set $L_{f_1}(y_1)$ should be formed as Figure~\ref{fig:level_set_chi_d=2}(\subref{fig:level_set_slope_d=2}). The level-set is piecewise linear, and the slopes (along with the axis $2$) take values within $(-1,1)$ due to the assumption $M>2m$. 

Similarly, we obtain the level-set $L_{f_2}(y_2)$. 
See Figure~\ref{fig:level_set_chi_d=2}(\subref{fig:level_set_intersection_d=2}); two level-sets $L_{f_1}(y_1),L_{f_2}(y_2)$ have at least one intersection point, which we write as $\tilde{\bs x} \in I^2$. 
Consider two lines with slopes $\pm 1$ crossing at $\tilde{\bs x}$, shown as black dot lines in Figure~\ref{fig:level_set_chi_d=2}(\subref{fig:level_set_intersection_d=2}). As slopes of two level-sets $L_{f_1}(y_1),L_{f_2}(y_2)$ take values within $(-1,1)$ (along with axes $2,1$, respectively), they belong to each region divided by the two dot lines, meaning that the intersection $L_{f_1}(y_1) \cap L_{f_2}(y_2)$ is unique, i.e., the function $\bs f=(f_1,f_2)$ is invertible. 
\qed

\begin{figure}[!ht]
\centering
\begin{minipage}{0.50\textwidth}
\centering
\includegraphics[width=0.8\textwidth]{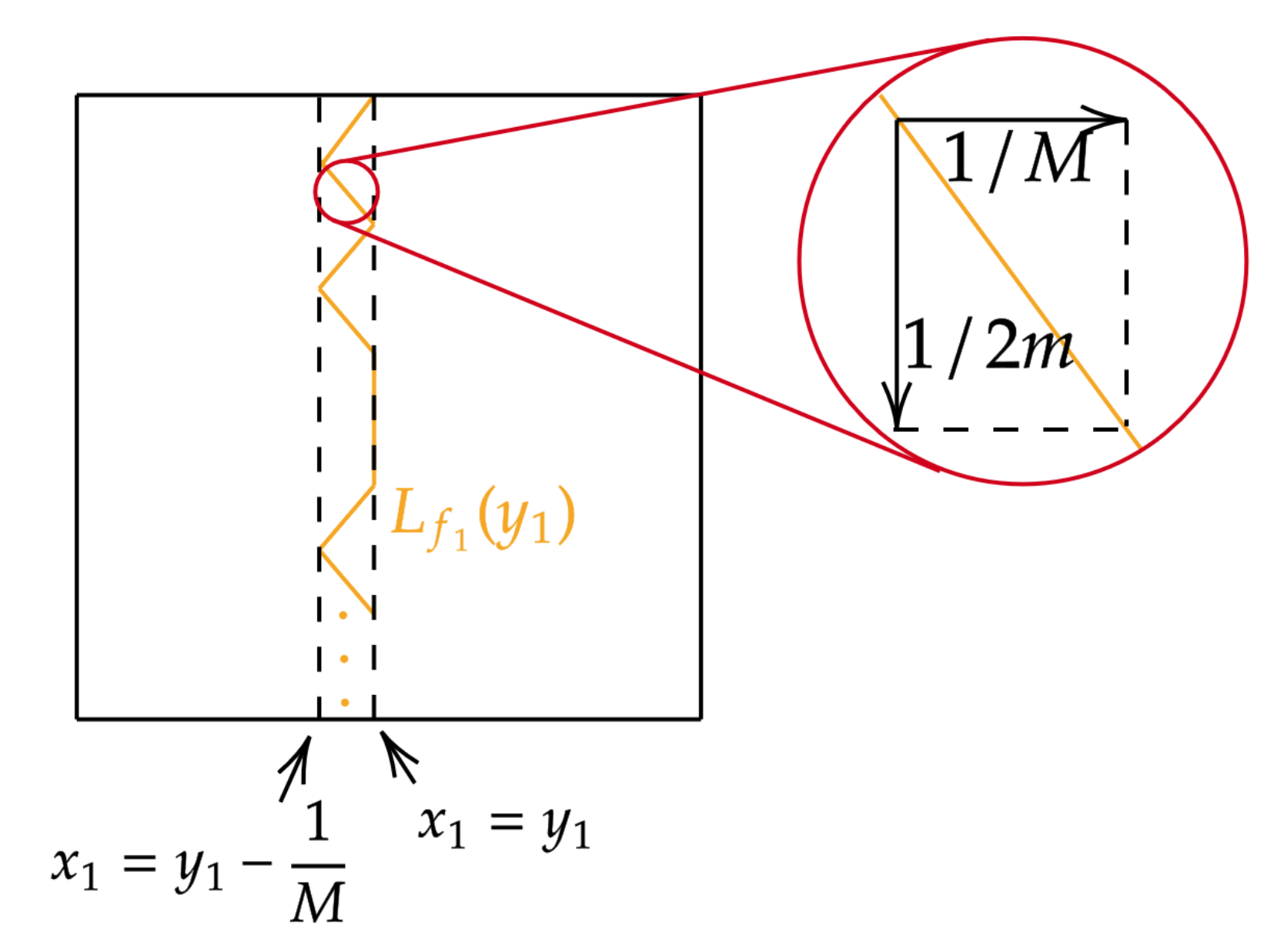}
\subcaption{Level-set $L_{f_1}(y_1)$, which is piecewise linear and its maximum slope is $2m/M$ (along with the axis $2$).}
\label{fig:level_set_slope_d=2}
\end{minipage}
\hfill
\begin{minipage}{0.47\textwidth}
\centering
\vspace{1em}
\includegraphics[width=0.6\textwidth]{level_set_intersection_d=2.pdf}
\vspace{1em}
\subcaption{Intersection of the level-sets $L_{f_1}(y_1),L_{f_2}(y_2)$ is a unique point.}
\label{fig:level_set_intersection_d=2}
\end{minipage}
\caption{Level-sets $L_{f_1},L_{f_2}$.}
\label{fig:level_set_chi_d=2}
\end{figure}

\subsection{Proof of Theorem \ref{thm:main_lower}.} 
\label{proof_thm:main_lower}
This proof consists of the following two steps: 
(step 1) we define an induced set $\tilde{\Xi}^2_k \subset \Xi^2_k$ and a function set $\f(\{\tilde{\Xi}^2_k\}_k) \subset \flipinv$, and
(step 2) we apply Lemma~\ref{thm:lower} to this (sufficiently complex) function set $\f(\{\tilde{\Xi}^2_k\}_k)$.

\paragraph{Step 1: Define $\f(\{\tilde{\Xi}^2_k\}_k)$, a sufficiently complex subset of $\flipinv$.} 
We define an induced subset of the function class $\mathcal{F}(\{\Xi^2_k\}_k)$. 
Using the Hamming distance $H(\theta,\theta')=\sum_{j_1,j_2=1}^{m}\mathbbm{1}\{\theta_{j_1,j_2} \ne \theta_{j_1,j_2}'\}$ defined for $\theta,\theta' \in \Theta_m^{\otimes 2}$, 
Lemma \ref{lem:vg} proves the existence of $T \subset \Theta_{m}^{\otimes 2}$ such that $|T| \ge m/8$ and $\min_{\theta \neq \theta' \in T}  H(\theta,\theta') \geq 2^{m/8}$. 
We define an induced function set 
\[
\tilde{\Xi}^2_k 
:=
\{ \xi_{\theta}(\bs x):=x_k+\chi_{\theta}(\bs x):I^2 \to I \mid \theta \in T \} \subset \Xi^2_k,
\]
where the definition of $\f(\cdot)$ and Lemma~\ref{prop:FXi_ivnertible} prove 
\begin{align}
\f(\{\tilde{\Xi}^2_k\}_k) \subset \f(\{\Xi^2_k\}_k) \subset \flipinv,
\label{eq:inclusion_relation}
\end{align}
indicating the invertibility of the functions $\bs f \in \f(\{\tilde{\Xi}^2_k\}_k)$.

For any different functions $\xi_{\theta},\xi_{\theta'} \in \tilde{\Xi}^2_k$, there exists $c_1 \in (0,1)$ such that we obtain
\begin{align}
    \|\xi_\theta - \xi_{\theta'}\|^2_{L^2} 
    &\ge 
    c_1 \|\theta - \theta'\|_1 m^{-2} m^{-2}
    \ge c_1 m^{-2}. 
    \label{ineq:dist_lower}
\end{align}
Hence, we obtain $\alpha$ in Lemma~\ref{thm:lower} (used in the next step) bounded below by $m^{-1}$.

\bigskip
\noindent \textbf{Step 2: Apply information-theoretic approach.} 
Finally, we develop a set of invertible functions. 
In this step, $P_j \in \mathcal{P}$ denotes a joint distribution of $\mathcal{D}_n$ associated to the probabilistic model \eqref{def:model}, equipped with the corresponding $\bs f_* = \bs f_j$. 

Considering the inclusion relation $\f(\{\tilde{\Xi}^2_k\}_k) \subset \flipinv$ shown in (\ref{eq:inclusion_relation}), we have
\begin{align*}
     \inf_{\bar{\bs f}_n} \sup_{\bs f \in \flipinv}\VERT \bar{\bs f}_n - \bs f\VERT_{L^2(P_{\bs X})}^2 \geq \inf_{\bar{\bs f}_n} \sup_{\bs f \in \f(\{\tilde{\Xi}^2_k\}_k)}\VERT \bar{\bs f}_n - \bs f\VERT_{L^2(P_{\bs X})}^2.
\end{align*}
By this form, it is sufficient to study a minimax rate with $\f(\{\tilde{\Xi}^2_k\}_k)$.

We apply the discussion for minimax analysis introduced in \cite{tsybakov2008introduction}, which is displayed as Lemma \ref{thm:lower}.
We check the conditions of Lemma \ref{thm:lower} one by one.
First, we check that $|\f(\{\tilde{\Xi}^2_k\}_k)| \geq (m^2 / 144)^2$.
Second, we consider different $\bs f = (f_1,f_2), \bs f' = (f'_1,f'_2) \in \f(\{\tilde{\Xi}^2_k\}_k)$ such that for some $k \in \{1,2\}$, $f_k = \xi_{\theta_k}$ for $f'_k = \xi_{\theta'_k}$ with $\theta_k, \theta_k' \in \tilde{\Xi}_k^2$ and $H(\theta_k,\theta_k') = 1$, and $f_{k'} = f'_{k'}$ for all $k' \in \{1,2\} \backslash \{k\}$.
Let $P_{\bs X_k}$ be a marginal measure of a $k$-th element of $\bs X \sim P_{\bs X}$.
Then, we obtain
\begin{align*}
    \VERT \bs f - \bs f' \VERT_{L^2(P_{\bs X})}^2 & \geq \|f_k - f_k'\|_{L^2(P_{\bs X_k})}^2 \geq c_2 m^{-2},
\end{align*}
by \eqref{ineq:dist_lower} with some constant $c_2>0$.
Here, the setting of the non-zero bounded density of $P_{\bs X}$ assures the existence of $c_2$.
Third, we apply an equation (2.36) in \cite{tsybakov2008introduction} which yields 
\begin{align*}
    \mathrm{KL}(P_j, P_0) \leq c_3 d n m^{-2d-2}
\end{align*}
which does not diverge when we set $m = n^{1/(2+d)}$.
Hence, we set $m = n^{1/(2+d)}$; 
there exists $C_* > 0$ such that, with a probability larger than $1/2$, we obtain
\begin{align*}
\inf_{\bar{\bs f}_n} 
    \sup_{\bs f_* \in \flipinv}
    \prisk(\bar{\bs f}_n, \bs f_*)
     \geq  C_* n^{-2/(2 + d)}.
\end{align*}
Finally, we apply the discussion of the minimax probability, which is described in (2.5) of \citet{tsybakov2008introduction}.
Since Markov's inequality gives that the risk $\risk(\bar{\bs f}_n, \bs f_*)$ is bounded below by $ C_*^{-1} n^{-2/(2+d)}\Pr(\prisk(\bar{\bs f}_n, \bs f_*) \geq C_*n^{-2/(2+d)}) $, we obtain the minimax lower bound of $\risk(\bar{\bs f}_n, \bs f_*)$ in the statement.
\qed

\subsection{Proof of Theorem~\ref{thm:sub_lower}}

As $\bs f^*$ is bi-Lipschitz, for any point $\bs y \in \Omega(\hat{\bs f}_n):=\{\bs y \in I^2:\hat{\bs f}_n\text{ is invertible at }\bs y\}$, we have an inequality 
\begin{align*}
    \|\hat{\bs f}_n^{\ddagger}(\bs y) - \bs f_*^{-1}(\bs y)\|_2
    &=
    \|\hat{\bs f}_n^{\ddagger}(\hat{\bs f}_n(\bs x)) - \bs f_*^{-1}(\hat{\bs f}_n(\bs x))\|_2 \\
    &=
    \|\bs f_*^{-1}(\bs f_*(\bs x)) - \bs f_*^{-1}(\hat{\bs f}_n(\bs x))\|_2 \\
    &\ge 
    L^{-1} \|\bs f_*(\bs x) - \hat{\bs f}_n(\bs x)\|_2,
\end{align*}
where $\bs x=\hat{\bs f}_n^{\ddagger}(\bs y)$.
Hence, we have 
\[
    \prisklinv(\hat{\bs f}_n,\bs f_*)
    =
    \VERT \hat{\bs f}_n^{\ddagger} - \bs f_*^{-1} \VERT_{L^2(P_X)}
    \ge 
    c L^{-1} \|\hat{\bs f}_n - \bs f_*\|_{L^2(P_X)} 
\]
with some constant $c \in (0,1)$.
By \citet{daneri2014smooth}, it can be shown that the Lebesgue measure of the non-invertible region $\lebesgue(I^2 \setminus \Omega(\hat{\bs f}_n))$ converges to $0$. 
Therefore, the assertion is proved by following the proof of Theorem~\ref{thm:main_lower}.
\qed

\section{Coherent Rotation}
\label{sec:coherent_rotation}

\subsection{Additional Symbol and Notation}
\label{subsec:symbols}

We define several functions and vectors with fixed $\bs f \in \flipinv$.
Recall that $\mathbb{D}$ is a unit ball in $\mathbb{R}^2$.

We develop a correspondence between the unit ball $\mathbb{D}$ and the square $I^2$ as the domain, by using some invertible maps.
We define a map $\omega: I^2  \to \mathbb{D}^2$
 as
\begin{align*}
     \omega(\bs x) := 
    \frac{\|\bs x\|_{\infty}}{\|\bs x\|_2}\bs x.
\end{align*}
Its inverse is explicitly written as
\begin{align*}
    \omega^{-1}(\bs y):=\frac{\|\bs y\|_2}{\|\bs y\|_{\infty}}\bs y,
\end{align*}
for $\bs y \in \mathbb{D}^2$.
Let $\tilde{\bs x}_{1}=(1,1),\tilde{\bs x}_{2}=(1,-1),\tilde{\bs x}_{3}=(-1,-1),\tilde{\bs x}_{4}=(-1,1) \in I^2$ be the vertices of the square $I^2$.
For each of the vertices, its corresponding point on $\mathbb{D}^2$ is defined as
$\tilde{\bs \zeta}_j=\omega(\bs f(\tilde{\bs x}_j)) \in \mathbb{D}^2$ $(j=1,2,3,4)$.

We also consider polar coordinates of elements in the unit ball $\mathbb{D}$.
For a radius $r \in [0,1]$ and an angle $\theta \in [0, 2\pi )$, $\bs v(r,\theta):=(r \sin \theta,r \cos \theta) \in \mathbb{D}$ is a transform from polar coordinate to the ordinary system.
For convenience, we define $\llb z \rrb$ denotes $z' \in [0,2\pi) $ satisfying $z-z' = 2\pi m$ for some $m \in \mathbb{Z}$, for any $z \in \mathbb{R}$.
Also, $\vartheta(\bs \zeta) := \llb \{\theta \mid (\sin \theta,\cos \theta)=\bs \zeta/\|\bs \zeta\|_2\} \rrb$ outputs a $[0,2\pi)$-valued angle between vectors $\bs \zeta \in \mathbb{D}^2 \setminus \{\bs 0\}$ and $\bs e=(0,1)$.

\subsection{Coherent rotation \texorpdfstring{$\bs \rho$}{Rho}}
\label{subsec:coherent_rotation}

For a function $\bs f_*=(f_1,f_2) \in \flipinv$, we define a coherent rotation $\bs \rho$ with functions and vectors defined in Section~\ref{subsec:symbols}:
\begin{align}
\bs \rho
:=
\omega^{-1} \circ R \circ \omega \, : \, I^2  \to I^2, 
\label{eq:rho}
\end{align}
where $R: \mathbb{D}^2 \to \mathbb{D}^2$ will be defined in the latter half of this section.

In preparation, we consider angles that correspond to the vertices of $I^2$ as
\begin{align*}
    \theta_j&:=\llb \vartheta(\tilde{\bs \zeta}_j) + \theta^{\dagger} \rrb \in [0,2\pi)
\quad (j=1,2,3,4),
\end{align*}
where $\theta^{\dagger}$ is a fixed angle $\theta^{\dagger} := \llb \llb 2\pi - \vartheta(\tilde{\bs \zeta}_1) \rrb + \frac{1}{2} \llb \vartheta(\tilde{\bs \zeta}_1)-\vartheta(\tilde{\bs \zeta}_4) \rrb \rrb \in [0,2\pi)$ for normalization.
The angles $\{\theta_j\}_{j=1}^{4}$ defined above satisfy $0<\theta_1<\theta_2<\theta_3<\theta_4<2\pi$.
Moreover, we define a strictly increasing function $\tau :[0,2\pi] \to [0,2\pi]$ as
\begin{align*}
    \tau(\theta)
&:=
\pi \cdot 
\begin{cases} 
        \frac{1}{4\theta_1} \theta & (\theta \in [0,\theta_1)) \\
        \frac{1}{2\theta_2-2\theta_1}\theta + \frac{\theta_2 - 3\theta_1}{4\theta_2-4\theta_1} & (\theta \in [\theta_1,\theta_2)) \\
        \frac{1}{2\theta_3-2\theta_2} \theta + \frac{3\theta_3-5\theta_2}{4\theta_3-4\theta_2} & (\theta \in [\theta_2,\theta_3)) \\
        \frac{1}{2\theta_4-2\theta_3} \theta + \frac{5\theta_4-7\theta_3}{4\theta_4-4\theta_3} & (\theta \in [\theta_3,\theta_4)) \\
        \frac{1}{8 \pi - 4\theta_4}\theta 
        +
        \frac{7\pi - 4\theta_4}{4\pi-2\theta_4} & (\theta \in [\theta_4,2\pi])
    \end{cases}.
\end{align*}
$\tau$ is a piecewise linear function which connects the angles $\{\theta_j\}_{j=1}^{4}$ defined above (shown in Figure~\ref{fig:tau}).
\begin{figure}
    \centering
    \includegraphics[width=0.4\hsize]{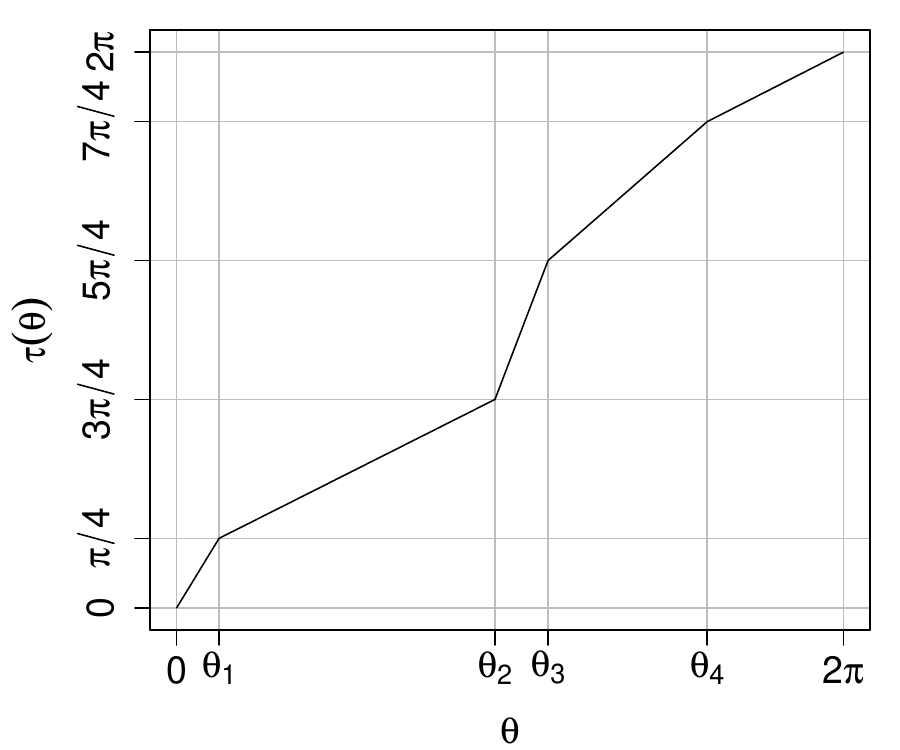}
    \caption{Strictly increasing and piecewise linear $\tau:[0,2\pi] \to [0,2\pi]$.}
    \label{fig:tau}
\end{figure}

Using the notions, we define the function $R: \mathbb{D}^2 \to \mathbb{D}^2$ as
\begin{align}
    R(\bs \zeta)
&:=
\begin{cases}
\bs v \left(
\|\bs \zeta\|_2 \, , \, 
\tau( \llb \vartheta(\bs \zeta) + \theta^{\dagger} \rrb)
\right)  & (\bs \zeta \ne \bs 0) \\
\bs0 & (\bs \zeta = \bs 0) \\
\end{cases}.
\label{eq:R}
\end{align}
$R$ has a role for rotating the points on the unit ball to make the points $\{\theta_j\}_{j=1}^4$  equally spaced on the boundary of $\mathbb{D}$.
Figure \ref{fig:rotation} shows the illustration of $\bs \rho$ including the role of $R$.

\begin{figure}
    \centering
\includegraphics[width=0.8\textwidth]{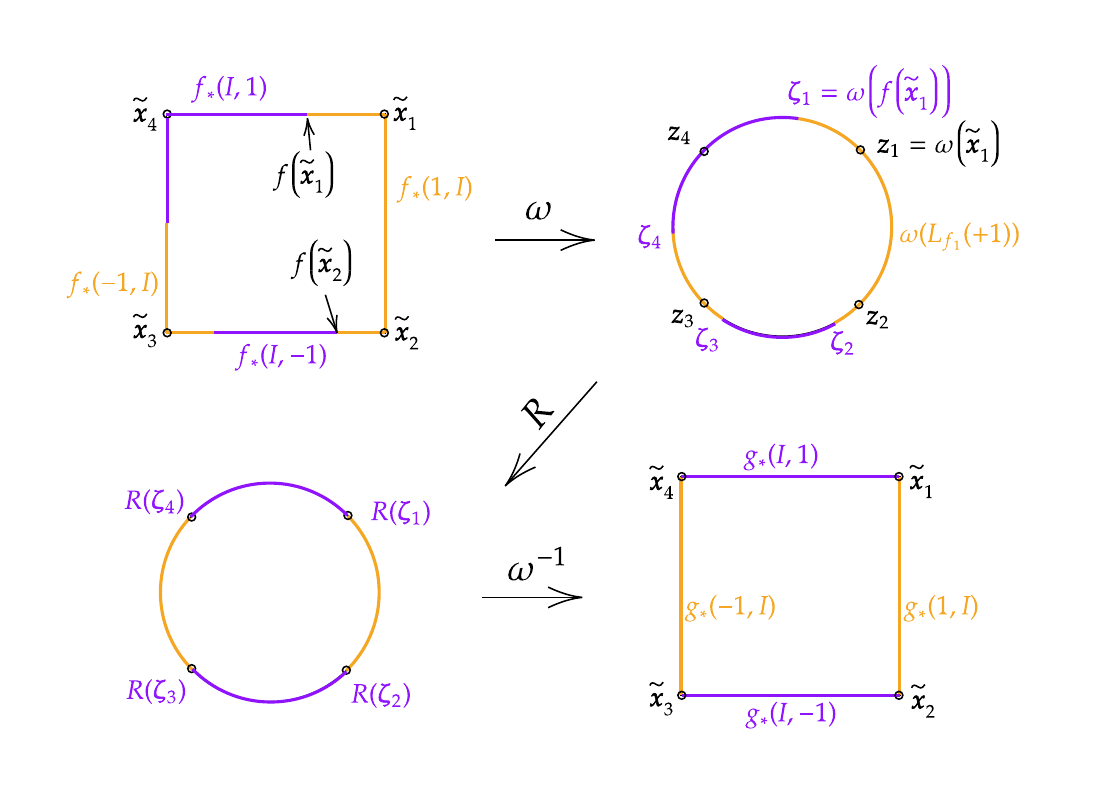}
\caption{The effect of the coherent rotation $\bs \rho:=\omega^{-1} \circ R \circ \omega:I^2 \to I^2$. 
$\omega$ converts the point in the square $I^2$ to the unit ball $\mathbb{D}$.
$R$ rotates the points in $\mathbb{D}$ to arrange $\{\theta_j\}_{j=1}^4$ equally spaced on the circle.
}
\label{fig:rotation}
\end{figure}

Lemma~\ref{lemma:rho_flipinv} below proves the invertibility and the bi-Lipschitz property of the function $\bs \rho$.
By the fact, we can define $\bs g_* := \bs \rho \circ \bs f_* \in \flipinv$.

\begin{lemma}
\label{lemma:rho_flipinv}
$\bs \rho \in \flipinv$.
\end{lemma}

\begin{proof}[Proof of Lemma~\ref{lemma:rho_flipinv}] 
$\omega$ is invertible from the definition, and 
$\omega,\omega^{-1}$ are Lipschitz as their directional derivatives are bounded on the compact set $I^2$. As $R$ is invertible from the definition, it suffices to show the bi-Lipschitz property of $R$.

If $\bs \zeta=\bs 0$, $\|R(\bs \zeta)-R(\bs \zeta')\|_2 =\|R(\bs \zeta')\|_2=\|\bs \zeta'\|_2 = \|\bs \zeta-\bs \zeta'\|_2$. 
Therefore, we herein prove the Lipschitz property of $R$ for $\bs \zeta,\bs \zeta' \in \mathbb{D}^2 \setminus \{\bs 0\}$. 
Without loss of generality, we assume that $\pi/2 \le \tilde{\theta}:= \vartheta(\bs \zeta) + \theta^{\dagger} < \tilde{\theta}':= \vartheta(\bs \zeta') + \theta^{\dagger} \le 3\pi/2$, indicating that $0 \le \tau(\tilde{\theta}) < \tau(\tilde{\theta}') \le 2\pi$. 

We first consider a case $|\tau(\tilde{\theta})-\tau(\tilde{\theta}')| \ge \pi$. 
Define a function $d(u,v;\theta):=\{u^2 + v^2 - 2uv \cos \theta\}^{1/2}$ ($u,v \ge 0,\theta \in \mathbb{R}$), which outputs $\|\bs \zeta-\bs \zeta'\|_2$ for two vectors $\bs \zeta=\bs v(u,0),\bs \zeta'=\bs v(v,\theta)$. 
Let $L_{\tau}>1$ be a Lipschitz constant of the bi-Lipschitz function $\tau$, where we obtain $|\tilde{\theta}-\tilde{\theta}'| > L_{\tau}^{-1}\pi$, and let $L_R:=\frac{1}{1-\cos(L_{\tau}^{-1}\pi)}$: we have
\begin{align*} 
\|R(\bs \zeta)-R(\bs \zeta')\|_2 
&\le 
\|\bs \zeta\|_2+\|\bs \zeta'\|_2 
\le 
L_R d(\|\bs \zeta\|_2,\|\bs \zeta'\|_2;L_{\tau}^{-1}\pi) 
\le 
L_R d(\|\bs \zeta\|_2,\|\bs \zeta'\|_2;|\tilde{\theta}-\tilde{\theta}'|) \\
&= 
L_R \|\bs \zeta-\bs \zeta'\|_2.
\end{align*}

For the remaining case $|\tau (\tilde{\theta})-\tau(\tilde{\theta}')| < \pi$, the
Lipschitz property of $\tau$ proves $|\tau(\tilde{\theta})-\tau(\tilde{\theta}')| \le L_{\tau}|\tilde{\theta}-\tilde{\theta}'|$, whereby we have
\begin{align*}
\|R(\bs \zeta)-R(\bs \zeta')\|_2
&=
d(\|\bs \zeta\|_2,\|\bs \zeta'\|_2; |\tau(\bs \zeta)-\tau(\bs \zeta')|)
\le 
d(\|\bs \zeta\|_2,\|\bs \zeta'\|_2; L_{\tau}|\bs \zeta-\bs \zeta'|)
\le 
L_{\tau} d(\|\bs \zeta\|_2,\|\bs \zeta'\|_2; |\bs \zeta-\bs \zeta'|) \\
&=
L_{\tau} \|\bs \zeta-\bs \zeta'\|_2.
\end{align*}

By replacing $\tau$ with $\tau^{-1}$, we can prove the Lipschitz property of $R^{-1}$ in the same way, thus both $\bs \rho=\omega^{-1} \circ R \circ \omega,\bs \rho^{-1}=\omega^{-1} \circ R^{-1} \circ \omega$ are Lipschitz. 
Lemma~\ref{lem:inverse_of_lipschitz} proves the bi-Lipschitz property of $\bs \rho$, which indicates the assertion $\bs \rho \in \flipinv$.
\end{proof}

\subsection{Proof of Lemma~\ref{lemma:coherent_rotation}}
\label{subsec:proof_of_lemma:coherent_rotation}

Using $\tau$, the function $R:\mathbb{D}^2 \to \mathbb{D}^2$ defined in (\ref{eq:R}) is designed to rotate vectors in $\mathbb{D}^2$ such that 
\[
    \vartheta(R(\tilde{\bs \zeta}_j))
    =
    \tau(\llb \vartheta(\tilde{\bs \zeta}_j)+\theta^{\dagger} \rrb)
    =
    \tau(\theta_j)
    =
    \frac{(2j-1)\pi}{4}
    =
    \vartheta(\omega(\tilde{\bs x}_j))
    \quad 
    (j=1,2,3,4). 
\]
Considering $\|R(\tilde{\bs \zeta}_j)\|_2=\|\tilde{\bs \zeta}_j\|_2=1=\|\omega(\tilde{\bs x}_j)\|_2$, we have 
\[
R(\omega(\bs f(\tilde{\bs x}_j)))=R(\tilde{\bs \zeta}_j)=\omega(\tilde{\bs x}_j),
\]
which indicates that
\begin{align}
    \bs g(\tilde{\bs x}_j)
    =
    (\bs \rho \circ \bs f)(\tilde{\bs x}_j)
    =
    \omega^{-1}( \, 
    \underbrace{R(\omega(\bs f_*(\tilde{\bs x}_j)))}_{=\omega(\tilde{\bs x}_j)} \, )
    =
    \tilde{\bs x}_j 
    \label{eq:g}
\end{align}
for $j=1,2,3,4$. 
The identity (\ref{eq:g}) together with Lemma~\ref{lem:homeomorphism_boundary} and Lemma~\ref{lemma:rho_flipinv} proves Lemma~\ref{lemma:coherent_rotation}, which is also illustrated in Figure~\ref{fig:rotation}.
\qed

\subsection[Estimator for rho]{Estimator \texorpdfstring{$\hat{\bs \rho}_n$}{rho hat} for \texorpdfstring{$\bs \rho$}{rho}}
\label{subsec:estimator_rho}
We obtain estimators $\hat{\theta}^{\dagger}_n,\hat{\theta}_{n,j},\hat{\tau}_n,\hat{R}_n$ for $\theta^{\dagger},\theta_j,\tau,R$ by substituting $\hat{\bs \zeta}_j:=\hat{\bs f}_n^{(1)}(\tilde{\bs x}_j)$ to $\tilde{\bs \zeta}_j$. 
Then, we develop an estimator for $\bs \rho$ and its inverse $\bs \rho^{-1}$ by 
\[
    \hat{\bs \rho}_n
    := 
    \omega^{-1} \circ \hat{R}_n \circ \omega, \mbox{~~and~~}
    \quad 
    \hat{\bs \rho}_n^{-1}
    :=
    \omega^{-1} \circ \hat{R}_n^{-1} \circ \omega,
\]
where the inverse of $\hat{R}_n$ is 
\[
    \hat{R}_n^{-1}(\bs \zeta)
    =
    \begin{cases} 
    \bs v(\|\bs \zeta\|_2,\llb \hat{\tau}_n^{-1}(\vartheta(\bs \zeta))-\hat{\theta}_n \rrb) & (\bs \zeta \ne \bs 0) \\
    \bs 0 & (\bs \zeta=\bs 0)
    \end{cases}.
\]
Note that we define $\hat{\tau}_n^{-1}(\bs \zeta):=\bs 0$ if $\hat{\tau}_n$ is not invertible (where such event occurs with probability approaching $0$, $n \to \infty$).
This case yields that $\hat{R}_n^{-1}(\bs \zeta)=\bs 0$ and $\hat{\bs \rho}_n^{-1}(\bs \zeta)=\bs 0$ for all $\bs \zeta$, hence $\hat{\bs f}_n(\bs x)=\bs 0$ for all $\bs x \in I^2$.

\section{Proofs for Upper Bound Analysis}
\label{sec:proofs_for_upper_bound_analysis}

Throughout this section, suppose $d=2$ and Assumption~\ref{asmp:uniform_estimator} with some fixed $\alpha>0$, i.e., 
\begin{align}
    \mathbb{P}(\vertinfty{\hat{\bs f}_n^{(1)}-\bs f_*} \le \exists C \gamma_n) \ge 1-\delta_n, 
    \, \text{ for } \,
    \gamma_{n} := n^{-1/4}(\log n)^{\alpha}.
    \label{eq:gamma}
\end{align}
$L_{\tau},L_R, L_{\omega},L_{\bs g} \ldots$ denote Lipschitz constants of the functions $\tau,R,\omega,\bs g,\ldots$, respectively. 
$\{t_n\}_{n \in \mathbb{N}}$ is a sequence defined as
\begin{align}
\tilde{\gamma}_n := (\log n)^{\beta} \gamma_n,
\quad 
t_n = \max\{t':=2^m \mid t' \le \tilde{\gamma}_n^{-1},m \in \mathbb{N}\}.
\label{eq:tn}
\end{align}
Note that, $t_n$ is of order $\tilde{\gamma}_n^{-1}$ and is the power of two: 
it indicates the monotonicity of $\hat{I}_n$, i.e., $\hat{I}_n \subset \hat{I}_{n'}$ for $n<n'$.


\bigskip 
In some proofs, we will employ a function $\bs g_*^{\dagger}$ defined as follows. 
For $\bs x \in I^2$ and symbols $\bs x',\bs x'',\bs x''',\bs x'''' \in \hat{I}^2,\bs s \in I^2,\alpha',\alpha'' \in [0,1]$ defined in Section~\ref{subsec:estimator} step (i-c), we define
\begin{align}
\bs g_*^{\dagger}(\bs x)
:=
\bs g_*(\bs s)
+
\alpha' \{\bs g_*(\bs x')-\bs g_*(\bs s)\}
+
\alpha'' \{\bs g_*(\bs x'')-\bs g_*(\bs s)\}.
\label{eq:gstar_dagger}
\end{align}



\subsection{Preliminaries}

\begin{lemma}
\label{lemma:negation_probability}
Let $\mathcal{A}_1,\mathcal{A}_2,\ldots,\mathcal{A}_J$ be events such that $\mathbb{P}(\mathcal{A}_j) \ge 1-\delta_{j}$ for $j \in [J]$, with a sequence $\{\delta_{j}\}_{ j \in [J]} \subset \mathbb{R}$. 
Then, $\mathbb{P}(\mathcal{A}_1 \text{ and } \mathcal{A}_2 \text{ and } \cdots \text{ and } \mathcal{A}_J) \ge 1-\sum_{j=1}^{J}\delta_{j}$. 
\end{lemma}

\begin{proof}[Proof of Lemma~\ref{lemma:negation_probability}]
With the logical negation $\lnot$, we have
$\mathbb{P}( \cap_{j \in [J]} \mathcal{A}_j) 
=
1 - \mathbb{P}(\cup_{j \in [J]} \lnot \mathcal{A}_j) 
\ge 
1 - \sum_{j=1}^{J}\mathbb{P}(\lnot \mathcal{A}_j)
\ge 
1 - \sum_{j=1}^{J} \delta_{j}$.
\end{proof}

\begin{lemma}
\label{lemma:convergence_via_decomposition}
Let $\{a_n\}_n,\{b_n\}_n,\{c_n\}_n$ be random sequences and let $\{d_n\}_n,\{\delta_n\}_n$ be deterministic sequences. 
Assume the existence of $C_a,C_b \in (0,\infty)$ such that
$\mathbb{P}(a_n \le C_a d_n)\ge 1-\delta_n/2$ and $\mathbb{P}(b_n \le C_b d_n) \ge 1-\delta_n/2$. 
If $c_n \le a_n+b_n$, then $\mathbb{P}(c_n \le C_c d_n) \ge 1-2\delta_n$ for $C_c:=C_a+C_b$.
\end{lemma}

\begin{proof}[Proof of Lemma~\ref{lemma:convergence_via_decomposition}]
Lemma~\ref{lemma:negation_probability} proves $\mathbb{P}(c_n \le C_c d_n)
\ge 
\mathbb{P}(a_n \le C_a d_n \text{ and } b_n \le C_b d_n) 
\ge 
1-2(\delta_n/2)$. 
\end{proof}

\begin{lemma}
\label{lemma:angle_perturbation}
Let $\bs \zeta_{n,1},\bs \zeta_{n,2}$ be $\mathbb{D}^2$-valued random variables satisfying the following conditions: 
with a sequence $\delta_n \searrow 0$, 
(i) $\mathbb{P}(\|\bs \zeta_{n,j}\|_2 \ge 1/2) \ge 1-\delta_n/3 \: (j=1,2)$, 
(ii) $\mathbb{P}(\llb \vartheta (\bs \zeta_{n,j}) \rrb \ne 0) \ge 1-\delta_n/3 \: (j=1,2)$, 
and 
(iii) there exists $C \in (0,\infty)$ such that 
$\mathbb{P}(\|\bs \zeta_{n,1}-\bs \zeta_{n,2}\|_{\infty}) \le C \gamma_n) \ge 1-\delta_n/3$. 
Then, 
\[
    \mathbb{P}\left(
        |\vartheta(\bs \zeta_{n,1})-\vartheta(\bs \zeta_{n,2})| \le 2 \sqrt{2} \pi C \gamma_n
    \right) \ge 1-\delta_n.
\]
\end{lemma}

\begin{proof}[Proof of Lemma~\ref{lemma:angle_perturbation}]
Assume that $\|\bs \zeta_{n,j}\|_2 \ge 1/2 \: (j=1,2)$ and let 
$A:=\left\|\frac{\bs \zeta_{n,1}}{\|\bs \zeta_{n,1}\|_2}-\frac{\bs \zeta_{n,2}}{\|\bs \zeta_{n,2}\|_2} \right\|_2$: 
then, $\llb \vartheta (\bs \zeta_{n,1}) - \vartheta (\bs \zeta_{n,2}) \rrb \le \pi A$ and $A < 2\|\bs \zeta_{n,1}-\bs \zeta_{n,2}\|_2$ hold, indicating that 
\[
    \llb \vartheta (\bs \zeta_{n,1}) - \vartheta (\bs \zeta_{n,2}) \rrb
    \le 
    2 \pi \|\bs \zeta_{n,1}-\bs \zeta_{n,2}\|_2
    \le 
    2\sqrt{2} \pi \|\bs \zeta_{n,1}-\bs \zeta_{n,2}\|_{\infty}.
\]
See Figure~\ref{fig:angle_perturbation}. 
Therefore, the assertion is proved by
\begin{align*}
    \mathbb{P}\left(
        |\vartheta(\bs \zeta_{n,1})-\vartheta(\bs \zeta_{n,2})| \le 2 \sqrt{2} \pi C \gamma_n
    \right)
&\ge 
    \mathbb{P}\left( 
        \|\bs \zeta_{n,j}\|_2 \ge 1/2,
        \llb \vartheta (\bs \zeta_{n,j}) \rrb \ne 0, \: 
        j=1,2,
        \|\bs \zeta_{n,1}-\bs \zeta_{n,2}\|_{\infty} \le C \gamma_n
    \right) \\
&\overset{\text{Lemma~\ref{lemma:negation_probability}}}{\ge} 
    1-\{\delta_n/3+\delta_n/3+\delta_n/3\}
=
    1-\delta_n. 
\end{align*}
\begin{figure}[!ht]
\centering
\includegraphics[width=0.4\textwidth]{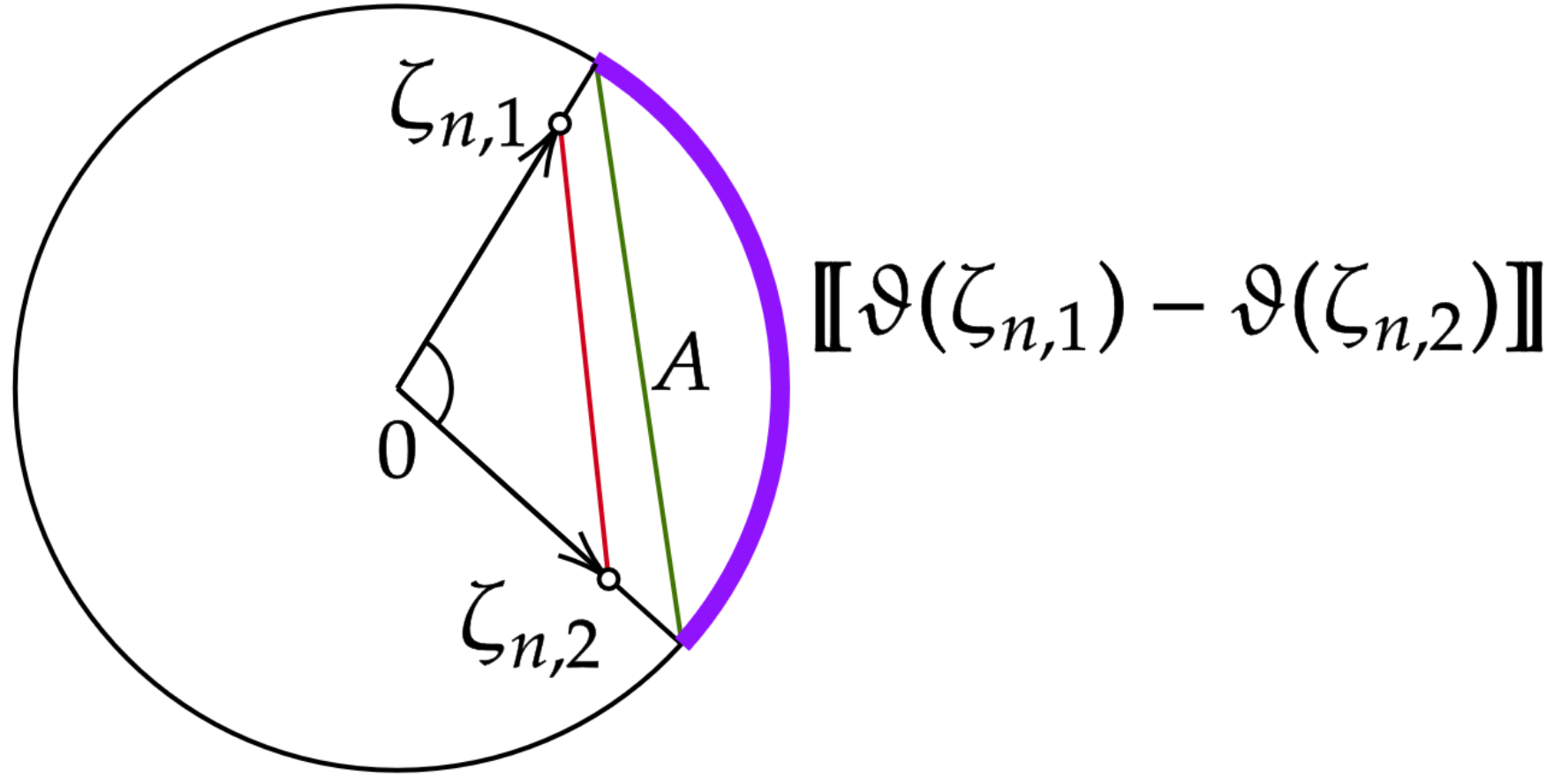}
\caption{$\llb \vartheta (\bs \zeta_{n,1})-\vartheta(\bs \zeta_{n,2}) \rrb$ is compatible with the blue arc of the unit ball $\mathbb{D}^2$, whose chord length is $A$.}
\label{fig:angle_perturbation}
\end{figure}
\end{proof}

\begin{proposition}
\label{prop:concentration_rho_g}
There exist $\{C_j\}_{j=1}^{7} \subset (0,\infty)$ such that
\begin{enumerate}[{(i)}]
\item $\mathbb{P}(\vertinfty{ \hat{\tau}_{n}^{-1} - \tau^{-1} } 
\le C_1 \gamma_{n}) \ge 1-\delta_{1,n}$, 
$\mathbb{P}(\vertinfty{ \hat{\tau}_{n} - \tau } 
\le C_2 \gamma_{n}) \ge 1-\delta_{2,n}$, 
\item $\mathbb{P}(\vertinfty{ \hat{R}_{n} - R } 
\le C_3 \gamma_{n}) \ge 1-\delta_{3,n}$, 
$\mathbb{P}(\vertinfty{ \hat{R}_{n}^{-1} - R } 
\le C_4 \gamma_{n}) \ge 1-\delta_{4,n}$, 
\item $\mathbb{P}(\vertinfty{ \hat{\bs \rho}_{n} - \bs \rho } 
\le C_5 \gamma_{n}) \ge 1-\delta_{5,n}$, 
$\mathbb{P}(\vertinfty{ \hat{\bs \rho}_{n}^{-1} - \bs \rho^{-1} } 
\le C_6 \gamma_{n}) \ge 1-\delta_{6,n}$,
\item $\mathbb{P}(\vertinfty{ \hat{\bs g}_{n} - \bs g_* } 
\le C_7 \gamma_{n}) \ge 1-\delta_{7,n}$, 
\end{enumerate}
with sequences $\delta_{j,n} \searrow 0$ as $n \to \infty$, $j=1,2,\ldots,7$. 
\end{proposition}

\begin{proof}[Proof of Proposition~\ref{prop:concentration_rho_g}]
In this proof, we first prove (i), and apply Lemma~\ref{lemma:convergence_via_decomposition} sequentially to prove (ii)--(v). Throughout this proof, for each $j \in [7]$, $\delta_{j,n}', \delta_{j,n}'',\ldots$ denote sequences approaching $0$ as $n \to \infty$, and $C_{j}',C_{j}'' \in (0,\infty)$ denote some constants. 

\begin{enumerate}[{(i)}]  
\item
Since Lemma~\ref{lemma:angle_perturbation} proves 
$\mathbb{P}(|\vartheta(\hat{\bs \zeta}_{n,j})-\vartheta(\tilde{\bs \zeta}_j)|\le C_1' \gamma_n) \ge 1-\delta_{1,n}'$ ($j=1,2,3,4$), the Lipschitz property of the maps ${\theta}^\dagger$ and ${\theta}$ in $\vartheta(\Tilde{\bs \zeta})$ yields
\begin{align*}
    &\mathbb{P}(|\hat{\theta}_n^{\dagger} - \theta^{\dagger}| \le C_1'' \gamma_n) \ge 1-\delta_{1,n}'', \mbox{~~and~~}
    \mathbb{P}\left( \max_{j=1,2,3,4}|\hat{\theta}_{n,j} - \theta_{j}| \le C_1''' \gamma_n \right)
    \ge 1-\delta_{1,n}'''.
\end{align*}
Further, since $0 < \theta_1 < \theta_2 < \theta_3 < \theta_4 < 2 \pi$ and the above convergence of $\hat{\theta}_{n,j}$ to $\theta_j$, we obtain
\begin{align*}
    \mathbb{P}(
        \underbrace{0
        <
        \llb \hat{\theta}_{n,1}+\hat{\theta}_n^{\dagger} \rrb 
        < 
        \llb \hat{\theta}_{n,2}+\hat{\theta}_n^{\dagger} \rrb
        <
        \llb \hat{\theta}_{n,3}+\hat{\theta}_n^{\dagger} \rrb
        <
        \llb \hat{\theta}_{n,4}+\hat{\theta}_n^{\dagger} \rrb
        <
        2\pi}_{(\star)}
    )
    \ge 
    1-\delta_{1,n}''''.
\end{align*}
In situations where this inequality ($\star$) holds, $\hat{\tau}_n$ is invertible since it is strictly monotone by its definition.
When ($\star$) holds, we obtain
\[
\vertinfty{ \hat{\tau}_n^{-1} - \tau^{-1} } 
\le 
    \max_{j=1,2,3,4}\left| \hat{\tau}_n^{-1} \left( \frac{2j-1}{4}\pi \right) - \tau^{-1} \left( \frac{2j-1}{4}\pi \right) \right|  
=
    \max_{j=1,2,3,4}| \hat{\theta}_{n,j} - \theta_j|.
\]
Hence, we obtain
\begin{align*}
    \mathbb{P}(\vertinfty{ \hat{\tau}_n^{-1} - \tau^{-1} } \le C_1 \gamma_n) 
    &\ge 
    \mathbb{P}\left( 
    \max_{j=1,2,3,4}|\hat{\theta}_{n,j} - \theta_{j}| \le C_1 \gamma_n \text{ and } (\star)
    \right) \\
    &\overset{\text{Lemma~\ref{lemma:negation_probability}}}{\ge} 
    1-\{\delta_{1,n}'''+\delta_{1,n}''''\}
    =:
    1-\delta_{1,n}
\end{align*}
by taking $C_1:=C_1'''$. 
An inequality $\vertinfty{\hat{\tau}_n-\tau}
=
\vertinfty{\hat{\tau}_n\circ \hat{\tau}_n^{-1} - \tau \circ \hat{\tau}_n^{-1}} 
= 
\vertinfty{\hat{\tau}_n\circ \hat{\tau}_n^{-1}-\tau \circ \hat{\tau}_n^{-1}+\tau \circ \tau^{-1} - \tau \circ \tau^{-1}}
=
\vertinfty{\tau \circ \hat{\tau}_n^{-1} - \tau \circ \tau^{-1} }
\le 
L_{\tau} \vertinfty{\hat{\tau}_n^{-1} - \tau^{-1}}
$ also proves 
\[
    \mathbb{P}(\vertinfty{ \hat{\tau}_n - \tau } \le C_2 \gamma_n) \ge 1-\delta_{n,2}
\]
with $C_2:=L_{\tau}C_1$ and $\delta_{n,2}:=\delta_{1,n}$. 

\item 
We apply Lemma~\ref{lemma:convergence_via_decomposition} and obtain
\begin{align*}
    \vertinfty{ \hat{R}_n - R }
&=
    \vertinfty{ 
        \bs v(\|\cdot\|_2, \: \hat{\tau}_n( \llb \vartheta (\cdot) + \hat{\theta}_n^{\dagger} \rrb) )
        -
        \bs v(\|\cdot\|_2, \: \tau( \llb \vartheta (\cdot) + \theta^{\dagger} \rrb) )
    } \\
&\le 
    \vertinfty{ 
        \bs v(1, \: \hat{\tau}_n( \llb \vartheta (\cdot) + \hat{\theta}_n^{\dagger} \rrb))
        -
        \bs v(1, \: \tau( \llb \vartheta (\cdot) + \theta^{\dagger} \rrb) )
    } \\
&\le 
    \vertinfty{
    \hat{\tau}_n( \llb \vartheta (\cdot) + \hat{\theta}_n^{\dagger} \rrb)
    -
    \tau( \llb \vartheta (\cdot) + \theta^{\dagger} \rrb)
    } \\
&\le 
    \vertinfty{
    \hat{\tau}_n( \llb \vartheta (\cdot) + \hat{\theta}_n^{\dagger} \rrb)
    -
    \tau( \llb \vartheta (\cdot) + \hat{\theta}_n^{\dagger} \rrb)
    }
    +
    \vertinfty{
    \tau( \llb \vartheta (\cdot) + \hat{\theta}_n^{\dagger} \rrb)
    -
    \tau( \llb \vartheta (\cdot) + \theta^{\dagger} \rrb)
    } \\
&\le 
    \vertinfty{\hat{\tau}_n - \tau}
    +
    L_{\tau} 
    |\llb \vartheta (\cdot) + \hat{\theta}_n^{\dagger} \rrb 
    - 
    \llb \vartheta (\cdot) + \theta^{\dagger} \rrb| \\
&\le 
    \vertinfty{\hat{\tau}_n - \tau}
    +
    L_{\tau} 
    |\hat{\theta}_n^{\dagger} - \theta^{\dagger} |.
\end{align*}
The inequality in the second line follows from the property of polar coordinates: $\|\bs v(r_1,\theta_1)-\bs v(r_2,\theta_2)\| \le \|\bs v(1,\theta_1)-\bs v(r_2,\theta_2)\| \le \|\bs v(1,\theta_1)-\bs v(1,\theta_2)\|$ for any $\theta_1,\theta_2 \in [0,2\pi)$. 
By the result of (i), we have 
\[
    \mathbb{P}(\vertinfty{\hat{R}_n - R} \le C_3 \gamma_n) \ge 1-\delta_{3,n}.
\]
Convergence of $\hat{R}_n^{-1}$ is proved in the same way. 

\item 
We apply Lemma~\ref{lemma:convergence_via_decomposition} and obtain
\begin{align*}
    \vertinfty{\hat{\bs \rho}_n-\bs \rho}
&=
    \vertinfty{ \omega^{-1} \circ \hat{R}_n \circ \omega - \omega^{-1} \circ R \circ \omega } 
\le 
    L_{\omega} \vertinfty{\hat{R}_n \circ \omega - R \circ \omega}
\le
    L_{\omega} \vertinfty{\hat{R}_n - R}.
\end{align*}
Then, by the result of (ii), we prove $\mathbb{P}(\vertinfty{\hat{\bs \rho}_n - \bs \rho} \le C_5 \gamma_n) \ge 1-\delta_{5,n}$. 
The result on $\hat{\bs \rho}_n^{-1}$ is proved in the same way.

\item We apply Lemma~\ref{lemma:convergence_via_decomposition} and obtain 
\begin{align*}
    \vertinfty{ \hat{\bs g}_n-\bs g_* }
&=
    \vertinfty{ \mathfrak{P}\hat{\bs \rho}_n \circ \hat{\bs f}_n^{(1)}-\bs \rho \circ \bs f_* } \\
&\le 
    \vertinfty{ \hat{\bs \rho}_n \circ \hat{\bs f}_n^{(1)}-\bs \rho \circ \bs f_* } \\
&\le 
    \vertinfty{ \hat{\bs \rho}_n \circ \hat{\bs f}_n^{(1)}-\bs \rho \circ \hat{\bs f}_n^{(1)} }
    +
    \vertinfty{ \bs \rho \circ \hat{\bs f}_n^{(1)} - \bs \rho \circ \bs f_* } \\
&\le
    \vertinfty{ \hat{\bs \rho}_n - \bs \rho }
    +
    L_{\bs \rho} \vertinfty{ \hat{\bs f}_n^{(1)} - \bs f_* }.
\end{align*}
By the result of (iii) and Assumption~\ref{asmp:uniform_estimator}, we obtain the statement of (iv). 
\end{enumerate}
\end{proof}

\subsection{Proof of Proposition \ref{prop:inpossibility_first_step}}

We fix $\bs f_* (\bs x) = (f_1(\bs x), f_2 (\bs x))$ as $f_1(\bs x) = x_1, \mbox{~and~} f_2(\bs x) = x_2$.
This is obviously Lipschitz continuous and invertible as $\bs f_*^{-1} (\bs x) = \bs x$.

For each $n\in \mathbb{N}$, we define an estimator $\hat{\bs f}_n^{(1)}(\bs x)=(\hat{f}_{n,1}^{(1)}(\bs x),\hat{f}_{n,2}^{(1)}(\bs x))$ as follows. 
We set $\hat{f}_{n,1}^{(1)}(\bs x)=x_1$.
For $\hat{f}_{n,2}^{(1)}$, with an arbitrary positive sequence $\{D_n\}_{n \in \mathbb{N}} \subset \mathbb{N}$, we define $\Delta_n:=2/D_n$ and 
$d_{n,m}:=-1+m \Delta_n$ ($m=0,1,2,\ldots,D_n$). 
Then, we define $\hat{f}_{n,2}^{(1)}$ as
\begin{align*}
    \hat{f}_{n,2}^{(1)}(\bs x) 
    = 
    \hat{f}_n^{\dagger}(x_2) 
    :=
    \begin{cases}
    d_{n,m} + 3(x_2-d_{n,m}) & (x_2 \in [d_{n,m},d_{n,m} + \Delta_n/3)) \\
    d_{n,m} + \Delta_n - 3(x_2-d_{n,m}-\Delta_n/3) & (x_2 \in [d_{n,m}+\Delta_n/3,d_{n,m} + 2\Delta_n/3)) \\
    d_{n,m} + 3(x_2-d_{n,m}-2\Delta_n/3) & (x_2 \in [d_{n,m}+2\Delta_n/3,d_{n,m+1})) \\
    1 & (x_2=1)
    \end{cases}.
\end{align*}
See Figure~\ref{fig:sawtooth_wave} for illustration of the function $\hat{f}_n^{\dagger}$. 
Then, we have 
$\|\hat{f}_{n,1}^{(1)}-f_1\|_{L^\infty}=0$ and $\|\hat{f}_{n,2}^{(1)}-f_2\|_{L^\infty} \le \Delta_n$, and these facts yield
\[ 
    \vertinfty{\hat{\bs f}_n^{(1)} - \bs f_*}
    \le \Delta_n.
\]
Hence, the estimator $\hat{\bs f}_n^{(1)}$ satisfies Assumption~\ref{asmp:uniform_estimator}, while $\hat{\bs f}_n^{(1)}$ converges to $\bs f_*$ arbitrarily fast by specifying large $D_n (= 2/\Delta_n) \in \mathbb{N}$. 

\begin{figure}[!ht]
\centering
\includegraphics[width=0.4\textwidth]{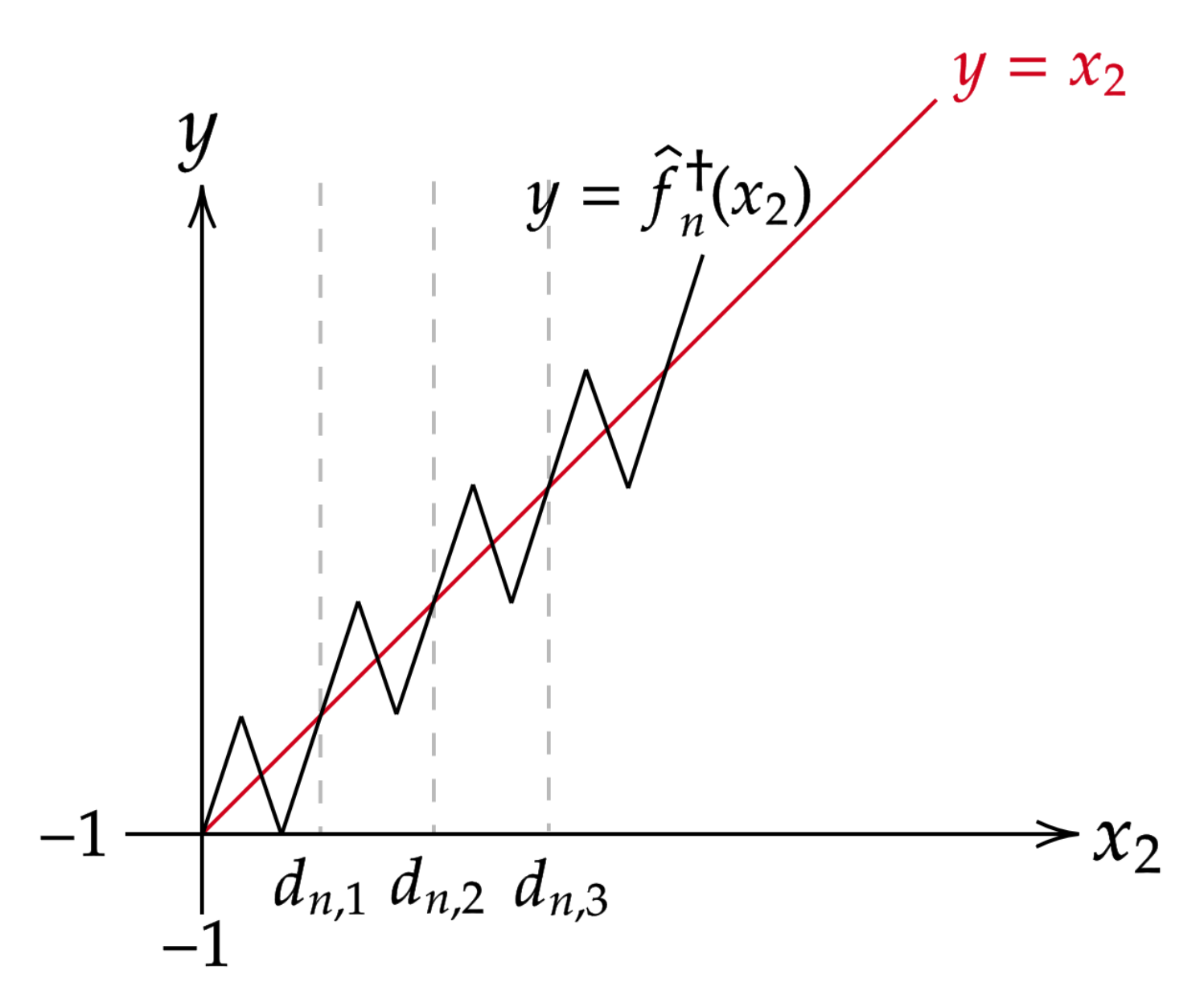}
\caption{Sawtooth-like estimator $\hat{f}_n^{\dagger}$}
\label{fig:sawtooth_wave}
\end{figure}

In the following, we prove that the estimator $\hat{\bs f}_{n}^{(1)}$ is not injective over the entire $I^2$. 
If $\hat{\bs f}_n$ is everywhere not invertible, i.e., $\hat{\bs f}_n^{-1}(\bs y)=\bs c \notin I^2$, it proves  the assertion as $\priskinv(\hat{\bs f}_n^{(1)},\bs f_*) \ge \|\bs c-\bs f_*^{-1}\|_{L^2} \ge \dhauss(\bs c,I^2) > 0$. 

Pick any $\bs x \in I^2$ and $n \in \mathbb{N}$. 
It is easy to derive $L_{\hat{f}_{n,1}^{(1)}}(x_1) = \{x_1\} \times I$ and 
\begin{align*}
    L_{\hat{f}_{n,2}^{(1)}}(x_2) = I \times \underbrace{\{y_2 \in I \mid \hat{f}_n^{\dagger}(x_2)=y_2\}}_{=L_{\hat{f}_n^{\dagger}}(y_2)},
\end{align*}
and their intersection is obtained as 
$L_{\hat{f}_{n,1}^{(1)}}(x_1) \cap L_{\hat{f}_{n,2}^{(1)}}(x_2) = \{x_1\} \times L_{\hat{f}_n^{\dagger}}(y_2)$. 
As the cardinality of 
$L_{\hat{f}_n^{\dagger}}(y_2)$ is greater than $1$, 
the intersection is not a unique point, indicating that the function $\hat{\bs f}_n^{(1)}$ is not injective at $\bs x \in I^2$. 
\qed

\subsection{Proof of Theorem~\ref{thm:upper_d2}} \label{sec:proof_upper_d2}

We review some notations.
As described in Introduction, an inverse function $\bar{\bs f}^{\ddagger}$ for a function $\bar{\bs f}:I^2 \to I^2$ is defined as
\[
    \bar{\bs f}^{\ddagger}(\bs y)
    := 
    \begin{cases} 
     \bs x & (\text{if } !\exists \bs x \text{ such that } \bar{\bs f}(\bs x)=\bs y) \\
     \bs c & (\text{otherwise})
    \end{cases},
\]
for some constant vector $\bs c \notin I^2$. 
We also define two sets 
\[
    \Omega(\bar{\bs f}) 
    :=
    \{\bs y \in I^2 \mid 
    \exists !\bs x \text{ such that } \bar{\bs f}(\bs x)=\bs y 
    \},
    \mbox{~~and~~} 
    \bar{\Omega}(\bar{\bs f}):=I^2 \setminus \Omega(\bar{\bs f}),
\]
that are used to measure a property of invertibility of functions.
For a set $\Bar{\Omega} \subset I^2$, $\lebesgue(\Bar{\Omega})$ denotes the Lebesgue measure of $\Bar{\Omega}$.

We develop an upper-bound of the inverse risk with the Lipschitz coefficient $L_{\bs f_*}$ of $\bs f_*$ (and $\bs f_*^{-1}$) and some constant $C_1,C_2,C_3 \in (0,\infty)$:
\begin{align}
    \priskinv(\hat{\bs f}_n,\bs f_*)
&=
    \VERT \hat{\bs f}_n - \bs f_* \VERT_{L^2(P_X)}^2
    +
    \psi\left(\VERT \hat{\bs f}_n^{\ddagger} - \bs f_*^{-1} \VERT_{L^2(P_X)}\right) \nonumber \\
&= 
    \VERT \hat{\bs f}_n - \bs f_* \VERT_{L^2(P_X)}^2
    +
    \left(\VERT \hat{\bs f}_n^{\ddagger} - \bs f_*^{-1} \VERT_{L^2(P_X)}^{2} \right)^2 \nonumber \\
&\le 
    \vertinfty{\hat{\bs f}_n - \bs f_*}^2
    +
    \left(
    \lebesgue(\Omega(\hat{\bs f}_n))
    \VERT \hat{\bs f}_n^{\ddagger} - \bs f_*^{-1} \VERT_{L^2(\Omega(\hat{\bs f}_n))}^2
    +
    C_1
    \lebesgue(\bar{\Omega}(\hat{\bs f}_n))
    \right)^2 \nonumber \\
&\overset{(\star)}{\le} 
    \vertinfty{\hat{\bs f}_n - \bs f_*}^2
    +
    \left(
    4L_{\bs f_*}^2
    \VERT \hat{\bs f}_n - \bs f_* \VERT_{L^2(\Omega(\hat{\bs f}_n))}^2
    +
    C_1
    \lebesgue(\bar{\Omega}(\hat{\bs f}_n))\right)^2 \nonumber \\
&\le 
    \left(C_2\vertinfty{\hat{\bs f}_n - \bs f_*}
    +
    C_3 \lebesgue(\bar{\Omega}(\hat{\bs f}_n))\right)^2.
    \label{eq:Rinv}
\end{align}
The inequality ($\star$) follows from $\lebesgue (\Omega(\bar{\bs f}_n)) \le \lebesgue(I^2)=4$ and the inequality
\begin{align*}
\|\hat{\bs f}_n^{\ddagger}(\bs y)-\bs f_*^{-1}(\bs y)\|_2
&=
\|\hat{\bs f}_n^{\ddagger}(\hat{\bs f}_n(\bs x))-\bs f_*^{-1}(\hat{\bs f}_n(\bs x))\|_2 
=
\|\bs f_*^{-1}(\bs f_*(\bs x))-\bs f_*^{-1}(\hat{\bs f}_n(\bs x))\|_2 \\
&\le 
L_{\bs f_*}\|\bs f_*(\bs x)-\hat{\bs f}_n(\bs x)\|_2 
\le 
L_{\bs f_*}\vertinfty{\bs f_*-\hat{\bs f}_n}, 
\quad (\bs y \in \bar{\Omega}(\hat{\bs f}_n)).
\end{align*}
Therefore, we herein evaluate 
$\lebesgue (\bar{\Omega}(\hat{\bs f}_n))$ and $\vertinfty{\hat{\bs f}_n-\bs f_*}$ in the following Propositions \ref{prop:area_twisted} and \ref{prop:concentration_f2}: 
applying Lemma~\ref{lemma:convergence_via_decomposition} with these Propositions to (\ref{eq:Rinv}) proves:
\[
    \mathbb{P}\left(
        \priskinv(\hat{\bs f}_n,\bs f_*)
        \le 
        \exists C \tilde{\gamma}_n^2
    \right) \ge 1-\delta_n, 
    \quad 
    \delta_n \searrow 0.
\]
By taking the expectation $\mathbb{E}_n$ with the decreasing 
$\delta_n \lesssim n^{-2/(2+d)}(\log n)^{2\alpha+2\beta}$, the statement is proved.
\qed

\begin{proposition}
\label{prop:area_twisted}
Suppose Assumption~\ref{asmp:uniform_estimator} holds. There exists $C \in (0,\infty)$ such that $\mathbb{P}\left( \lebesgue(\bar{\Omega}(\hat{\bs f}_n)) \le C \tilde{\gamma}_n \right) \ge 1-\delta_n$ with a sequence 
$\delta_n \searrow 0$ as $n \to \infty$.
\end{proposition}

\begin{proof}[Proof of Proposition~\ref{prop:area_twisted}]

Let $\bs g_*^{\dagger}$ be a function for triangle interpolation defined in \eqref{eq:gstar_dagger}. 
Then, Proposition~4.1 in \citet{daneri2014smooth} evaluates the Lebesgue measure of the squares, that cannot be linearly interpolated (so twisted): there exists $C_1 \in (0,\infty)$ such that
\begin{align}
    \lebesgue(\bar{\Omega}(\bs g_*^{\dagger})) \le 
    C_1 \tilde{\gamma}_n.
    \label{eq:daneri_pratelli}
\end{align}
(\ref{eq:daneri_pratelli}) is obtained by specifying that $r=\tilde{\gamma}_n$ and $\varepsilon$ is proportional to $\tilde{\gamma}_n$ in Proposition~4.1 in \citet{daneri2014smooth}. 

Here, we show $\bar{\Omega}(\hat{\bs \rho}_n^{-1} \circ \hat{\bs g}_n^{\dagger}) \subset \hat{\bs \rho}_n^{-1}(\bar{\Omega}(\hat{\bs g}_n^{\dagger}))$. 
We denote $A:=\bar{\Omega}(\hat{\bs \rho}_n^{-1} \circ \hat{\bs g}_n^{\dagger}) $ and $B:= \hat{\bs \rho}_n^{-1}(\bar{\Omega}(\hat{\bs g}_n^{\dagger}))$.
Every $\bs y \in A$ satisfies $\bs y = \hat{\bs \rho}_n^{-1} \circ \hat{\bs g}_n^{\dagger}(\bs x)= \hat{\bs \rho}_n^{-1} \circ \hat{\bs g}_n^{\dagger}(\bs x')$ for some $\bs x \ne \bs x'$.
This fact is rewritten as $\hat{\bs \rho}_n(\bs y)=\hat{\bs g}_n^{\dagger}(\bs x)=\hat{\bs g}_n^{\dagger}(\bs x')$, i.e., $\hat{\bs \rho}_n(\bs y) \in \bar{\Omega}(\hat{\bs g}_n^{\dagger})$. Applying $\hat{\bs \rho}_n^{-1}$ to both sides yields $\bs y \in B$, whereby we have $A \subset B$.

This inclusion relation $A \subset B$ yields 
\begin{align}
    \lebesgue(\bar{\Omega}(\hat{\bs f}_n))
    &=
    \lebesgue(\bar{\Omega}(\hat{\bs \rho}_n^{-1} \circ \hat{\bs g}_n^{\dagger})) 
    \le 
    \lebesgue(\hat{\bs \rho}_n^{-1}(\bar{\Omega}(\hat{\bs g}_n^{\dagger})))
    \le 
    L_{\hat{\bs \rho}_n}^2
    \lebesgue(\bar{\Omega}(\hat{\bs g}_n^{\dagger})) \nonumber \\
    &\le 
    L_{\hat{\bs \rho}_n}^2
    \underbrace{
    \lebesgue(\bar{\Omega}(\bs g_*^{\dagger}))
    }_{\le C_1 \tilde{\gamma}_n}
    +
    L_{\hat{\bs \rho}_n}^2
    \underbrace{|
        \lebesgue(\bar{\Omega}(\hat{\bs g}_n^{\dagger}))
        -
        \lebesgue(\bar{\Omega}(\bs g_*^{\dagger}))
    |}_{=: T_1}.
    \label{eq:evaluation_of_non_invertible_region}
\end{align}
Since the vertices of the squares converge in probability with the convergence rate $\gamma_n = o(\tilde{\gamma}_n)$, the term $T_1$ is of order $O_p(\gamma_n^2)=o_p(\tilde{\gamma}_n^2)=o_p(\tilde{\gamma}_n)$, i.e., 
$\mathbb{P}(T \le C_2\tilde{\gamma}_n) \ge 1-\delta_n$ for  some $C_2>0$ and $\delta_n \searrow 0$. Therefore, applying Lemma~\ref{lemma:convergence_via_decomposition} leads to the assertion. 
\end{proof}

\begin{proposition}[Approximation error]
\label{prop:concentration_f2}
There exist constants $C_1,C_2 \in  (0,\infty)$ that satisfies the followings:
\begin{enumerate}[{(i)}]
\item $\mathbb{P}(\vertinfty{ \hat{\bs g}_{n}^{\dagger} - \bs g_* } 
\le C_1 \tilde{\gamma}_{n}) \ge 1-\delta_{1,n}$, 
\item $\mathbb{P}(\vertinfty{ \hat{\bs f}_{n} - \bs f_* } 
\le C_2 \tilde{\gamma}_{n}) \ge 1-\delta_{2,n}$, 
\end{enumerate}
with some positive sequences $\delta_{j,n} \searrow 0$ as $n \to \infty$, $j=1,2$. 
\end{proposition}

\begin{proof}[Proof of Proposition~\ref{prop:concentration_f2}]
We prove (i) and (ii) by the uniform convergence of $\hat{\bs g}_n$, which is already proved in Proposition~\ref{prop:concentration_rho_g} (iv). 
\begin{enumerate}[{(i)}]
\item 
Consider symbols $\bs x',\bs x'',\bs x''',\bs x'''',\bs s,\hat{\shikakuimage},\hat{\bs g}_n^{\dagger}$ defined in Section~\ref{subsec:estimator}, and $\bs g_*^{\dagger}$ defined in eq.~(\ref{eq:gstar_dagger}). 

Firstly, we consider the case that $\hat{\shikakuimage}$ is not twisted: 
with the triangle $\triangle(\bs x)$ whose vertices are $\bs x',\bs x'',\bs s$, we have 
\begin{align*}
& \|\hat{\bs g}_n^{\dagger}(\bs x)-\bs g_*(\bs x)\|_{\infty} \\
&\le 
    \|\hat{\bs g}_n^{\dagger}(\bs x)-\bs g_*^{\dagger}(\bs x)\|_{\infty}
    +
    \|\bs g_*^{\dagger}(\bs x) - \bs g_*(\bs x)\|_{\infty} \\
&\le 
    \left\|
        \{\hat{\bs g}_n(\bs s) - \bs g_*(\bs s)\}
        +
        \alpha' \{\hat{\bs g}_n(\bs x')-\bs g_*(\bs x')\}
        +
        \alpha'' \{\hat{\bs g}_n(\bs x'') - \bs g_*(\bs x'') \}
        -
        (\alpha'+\alpha'') \{\hat{\bs g}_n(\bs s) - \bs g_*(\bs s)\}
    \right\|_{\infty} \\
&\quad + 
    \left\| 
        -(\bs g_*(\bs x)-\bs g_*(\bs s))
        +
        \alpha'\{\bs g_*(\bs x')-\bs g_*(\bs s)\}
        +
        \alpha'' \{\bs g_*(\bs x'') - \bs g_*(\bs s)\}
    \right\|_{\infty} \\
&\le
    5 \vertinfty{\hat{\bs g}_n - \bs g_*}
    +
    3 \sup_{\bs x \in I^2} \sup_{\bs s \in \triangle (\bs x)}
    \| \bs g_*(\bs x)-\bs g_*(\bs s) \|_{\infty} \\
&\le 
    5 \vertinfty{\hat{\bs g}_n - \bs g_*} + 3L_{\bs g_*} \underbrace{\sup_{\bs x \in I^2}\sup_{\bs s \in \triangle(\bs x)}\|\bs x-\bs s\|_{\infty}}_{\le 1/t_n} \\
&\le 
    5 \underbrace{\vertinfty{\hat{\bs g}_n - \bs g_*}}_{=O_p(\gamma_n)=o_p(\tilde{\gamma}_n)} + 3L_{\bs g_*} \tilde{\gamma}_n.
\end{align*}

Secondly, if $\hat{\shikakuimage}$ is twisted, 
\begin{align*}
    \|\hat{\bs g}_n^{\dagger}(\bs x)-\bs g_*(\bs x) \|_{\infty} 
&=
    \|\hat{\bs g}_n(\bs x')-\bs g_*(\bs x)\|_{\infty} \\
&\le 
    \|\hat{\bs g}_n(\bs x')-\bs g_*(\bs x')\|_{\infty}
    +
    \|\bs g_*(\bs x')-\bs g_*(\bs x) \|_{\infty} \\
&\le 
    \vertinfty{\hat{\bs g}_n-\bs g_*}
    +
    L_{\bs g_*}\tilde{\gamma}_n.
\end{align*}
Overall, we have obtained 
\[
    \vertinfty{\hat{\bs g}_n^{\dagger}-\bs g_*}
    \le 
    5 \vertinfty{\hat{\bs g}_n-\bs g_*}
    +
    3L_{\bs g_*}\tilde{\gamma}_n,
\]
and applying Lemma~\ref{lemma:convergence_via_decomposition} with Proposition~\ref{prop:concentration_rho_g} (iv) and the definition of $t_n$ as \eqref{eq:tn} proves the assertion.

\item We apply Lemma~\ref{lemma:convergence_via_decomposition} and obtain
\begin{align*}
    \vertinfty{\hat{\bs f}_n-\bs f_* }
&\le 
    \vertinfty{\hat{\bs \rho}_n^{-1} \circ \hat{\bs g}_n^{\dagger} - \bs \rho^{-1} \circ \bs g } \\
&\le 
    \vertinfty{ \hat{\bs \rho}_n^{-1} \circ \hat{\bs g}_n^{\dagger} - \bs \rho^{-1} \circ \hat{\bs g}_n^{\dagger} }
    +
    \vertinfty{ \bs \rho^{-1} \circ \hat{\bs g}_n^{\dagger}
    -
    \bs \rho^{-1} \circ \bs g } \\
&=
    \vertinfty{ \hat{\bs \rho}_n^{-1} - \bs \rho^{-1} }
    +
    L
    \vertinfty{ \hat{\bs g}_n^{\dagger}
    -
    \bs g_* }
\end{align*}
with the above (i) and Proposition~\ref{prop:concentration_rho_g} (iii) leads to (vi).
\end{enumerate}
\end{proof}

\subsection{Proof of Proposition~\ref{prop:upper_d2_phi2}}
\label{subsec:Proof_of_Proposition_upper_d2_phi2}

This proposition is obtained by slightly modifying the proof of Theorem~\ref{thm:upper_d2} (shown in Appendix~\ref{sec:proof_upper_d2}).
Specifically, we replace the penalty function $\psi(z)=z^4$ in the inequality~(\ref{eq:Rinv}) with $\psi(z)=z^2$ and obtain an inequality
\begin{align*}
    \priskinv(\hat{\bs f}_n,\bs f_*)
&= 
    \VERT \hat{\bs f}_n - \bs f_* \VERT_{L^2(P_X)}^2
    +
    \VERT \hat{\bs f}_n^{\ddagger} - \bs f_*^{-1} \VERT_{L^2(P_X)}^{2} \nonumber \\
&\le 
    \vertinfty{\hat{\bs f}_n - \bs f_*}^2
    +
    \lebesgue(\Omega(\hat{\bs f}_n))
    \VERT \hat{\bs f}_n^{\ddagger} - \bs f_*^{-1} \VERT_{L^2(\Omega(\hat{\bs f}_n))}^2
    +
    C_1
    \lebesgue(\bar{\Omega}(\hat{\bs f}_n))
    \nonumber \\
&\le
    \vertinfty{\hat{\bs f}_n - \bs f_*}^2
    +
    4L_{\bs f_*}^2
    \VERT \hat{\bs f}_n - \bs f_* \VERT_{L^2(\Omega(\hat{\bs f}_n))}^2
    +
    C_1
    \lebesgue(\bar{\Omega}(\hat{\bs f}_n)) \nonumber \\
&\le 
    C_2\vertinfty{\hat{\bs f}_n - \bs f_*}^2
    +
    C_3 \lebesgue(\bar{\Omega}(\hat{\bs f}_n))
\end{align*}
with some $C_1,C_2,C_3>0$. 
Recall that $\lebesgue( \cdot )$ denotes the Lebesgue measure, $\bar{\Omega}(\cdot)$ denotes the set of non-invertible points, and $L_{\bs f_*}$ denotes the Lipschitz coefficient of $\bs f_*$.
Since the first term in the rightmost side is $O_p(\gamma_n^2)=o_p(\tilde{\gamma}_n^2)$ by Assumption~\ref{asmp:uniform_estimator}, it suffices to prove the latter term $\lebesgue(\bar{\Omega}(\hat{\bs f}_n))$ to be $O_p(\tilde{\gamma}_n^2)$. 
Note that, \textit{without} the condition $\bs f_* \in \flip(2^{1/4})$,  Proposition~\ref{prop:area_twisted} shows $\lebesgue(\bar{\Omega}(\hat{\bs f}_n))=O_p(\tilde{\gamma}_n)$ but not $O_p(\tilde{\gamma}_n^2)$.

For the function $\bs g_*^{\dagger}$ defined in eq.~(\ref{eq:gstar_dagger}), we suppose the following condition 
\begin{align}
    \lebesgue(\bar{\Omega}(\bs g_*^{\dagger}))
    =
    0,
    \label{eq:no_twist_measure}
\end{align}
which will be proved later for $\bs f_* \in \flipinv \cap \flip(2^{1/4})$.
With \eqref{eq:no_twist_measure}, we replace the inequality (\ref{eq:daneri_pratelli}) in the proof of Proposition~\ref{prop:area_twisted} with (\ref{eq:no_twist_measure}) and obtain
\[
    \mathbb{P}(\lebesgue(\bar{\Omega}(\hat{\bs f}_n) \le C \tilde{\gamma}_n^2)) \ge 1-\delta_n
\]
with a decreasing sequence $\delta_n \searrow 0$ and $C>0$, and it completes the proof.

We prove (\ref{eq:no_twist_measure}) for $\bs f_* \in \flipinv \cap \flip(2^{1/4})$. 
Consider a square $\shikaku$ and its vertices $\nu(\shikaku):=\{\bs x',\bs x'',\bs x''',\bs x''''\} \subset I^2$ defined in Section~\ref{sec:partition}, and also consider its corresponding quadrilateral $\shikakuimage$ with vertices $\bs g_*(\bs x'),\bs g_*(\bs x''),\ldots,\bs g_*(\bs x'''')$.
We define another quadrilateral
\[
    \shikakuimage^{\sharp}:=\text{quadrilateral whose vertices are }\bs f_*(\nu(\shikaku)),
\]
which is a variant of $\Diamond = \bs g_*(\nu(\shikaku)) =\bs \rho \circ \bs f_* (\nu(\shikaku)) $ by reducing the coherent rotation $\bs \rho$. 
Here, we obtain an interesting fact: $\shikakuimage$ is twisted if and only if $\shikakuimage^{\sharp}$ is twisted.
This fact simply holds, since the coherent rotation $\bs \rho$ is only a rotation through the monotone function $\tau$ defined in Section \ref{subsec:coherent_rotation}. See Figure~\ref{fig:rho_applied_to_quadrilateral} illustrates it: the triangle with three vertices $\bs y'=\bs f_*(\bs x'),\bs y''=\bs f_*(\bs x''),\bs y'''=\bs f_*(\bs x''')$ cannot be inverted from the triangle of $\bs z'=\bs g_*(\bs x')=\bs \rho(\bs y'),\bs z''=\bs g_*(\bs x'')=\bs \rho(\bs y''),\bs z'''=\bs g_*(\bs x''')=\bs \rho(\bs y''')$ (and the same holds for triangles with vertices $\bs y',\bs y''',\bs y''''$ and $\bs z',\bs z''',\bs z''''$, respectively), whereby the quadrilateral $\hat{\shikakuimage}$ cannot be twisted if $\hat{\shikakuimage}^{\sharp}$ is not twisted. 

\begin{figure}[!ht]
\includegraphics[width=\textwidth]{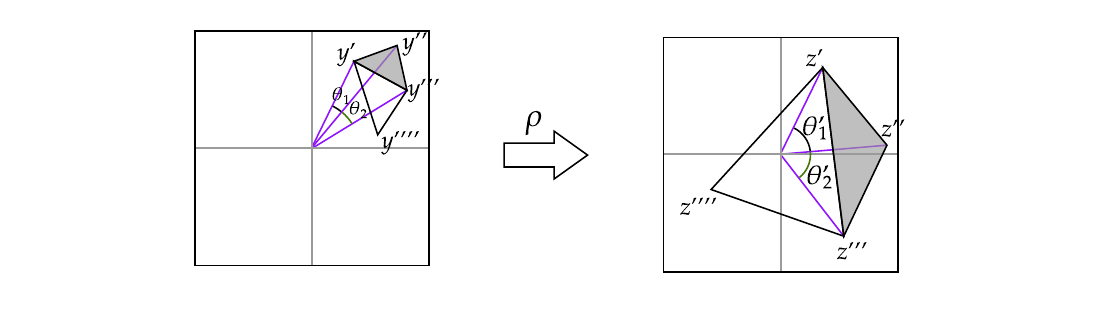}
\caption{The triangle with vertices $\bs y'=\bs f_*(\bs x'),\bs y''=\bs f_*(\bs x''),\bs y'''=\bs f_*(\bs x''')$ cannot be inverted by the coherent rotation $\bs \rho$ (cf. $\bs z'=\bs g_*(\bs x')=\bs \rho(\bs y'),\bs z''=\bs g_*(\bs x'')=\bs \rho(\bs y''),\bs z'''=\bs g_*(\bs x''')=\bs \rho(\bs y''')$).}
\label{fig:rho_applied_to_quadrilateral}
\end{figure}


\begin{figure}[!ht]
\centering 
\includegraphics[width=\textwidth]{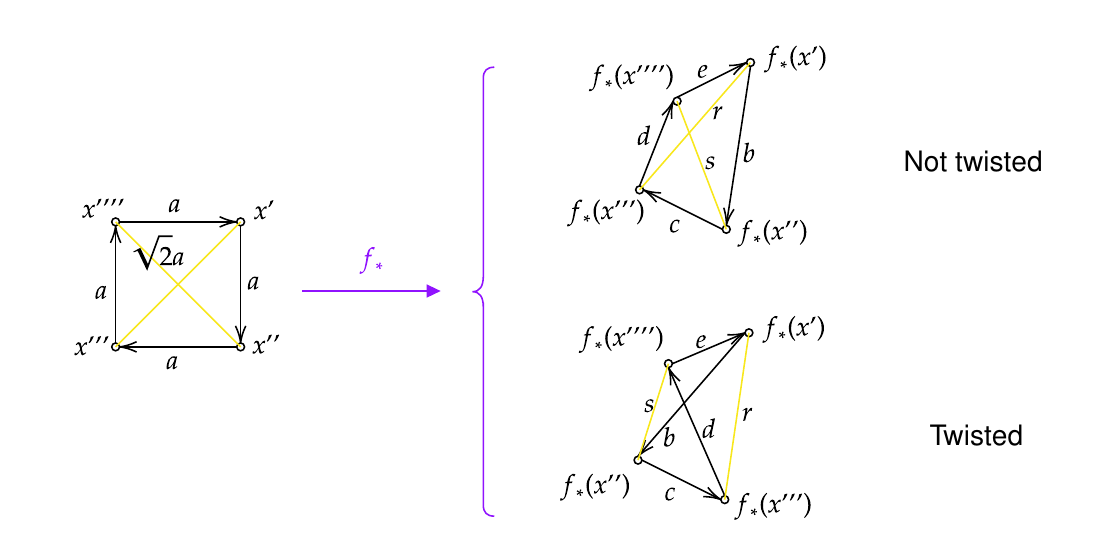}
\caption{The left is $\shikaku$ with its sides of length $a$ and diagonals of length $\sqrt{2}a$.
The right is $\shikakuimage^{\sharp}$, obtained by transforming $\shikaku$ with $\bs f_*$, showing both twisted and not twisted cases.}
\label{fig:twist_condition}
\end{figure}

By this fact, it is sufficient to show that $\shikakuimage^{\sharp}$ is not twisted.
Let $a > 0$ be the length of a side of the square $\shikaku$, then the length of a diagonal of $\shikaku$ is $\sqrt{2}a$.
Let $b,c,d,e> 0$ be length of lines obtained by transforming the sides of $\shikaku$ by $\bs f_*$, and $r, s > 0$ be length of lines obtained by transforming the diagonals of $\shikaku$
as shown in  Figure~\ref{fig:twist_condition}. 
The bi-Lipschitz property of $\bs f_*$ with the Lipschitz constant $L=2^{1/4}$ yields
\[
    2^{-1/4}a \le \min\{b,c,d,e\} \leq \max\{b,c,d,e\} \le 2^{1/4} a
    \quad 
\]
and
\[
    \quad 
    2^{1/4}a = 2^{-1/4} (\sqrt{2} a)
    \le \min\{r,s\}.
\]
By these facts, we obtain
\[
    \max\{b,c,d,e\} \, \le \, \min\{r,s\}.
\]
Then, $\shikakuimage^{\sharp}$ cannot be twisted, since the length of the diagonals of $\shikakuimage^{\sharp}$ is no less than those of the sides of $\shikakuimage^{\sharp}$.
Therefore, the pathological example of twists (shown in Appendix~\ref{subsec:twist}) does not appear, 
hence (\ref{eq:no_twist_measure}) is proved.
\qed

\begin{remark}
\label{remark:twist_L=sqrt2}
While the above proof considers the bi-Lipschitz function with $L \le 2^{1/4} \approx 1.19$, we here consider the case $L=2^{1/2} \approx 1.41$. 
Even in this case (that seems theoretically tractable), twist may appear as shown in Figure~\ref{fig:twist_counterexample}, and the above proof does not hold. 
\begin{figure}[!ht]
\centering 
\includegraphics[width=\textwidth]{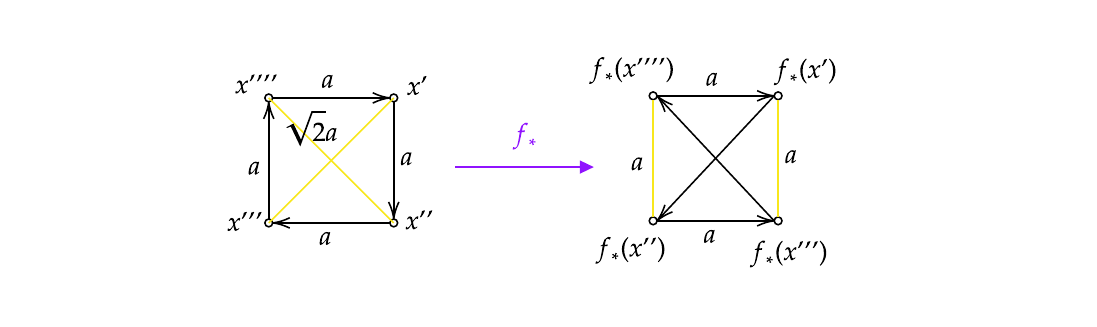}
\caption{Twist with $L=2^{1/2}$}
\label{fig:twist_counterexample}
\end{figure}
\end{remark}

\subsection{A Pathological Example of Twists}
\label{subsec:twist}

We defined an interpolation over the quadrilaterals as shown in Figure~\ref{fig:interpolation}. 
However, the quadrilateral connecting the four points $\bs u'=\hat{\bs g}_n(\bs x'),\bs u''=\hat{\bs g}_n(\bs x''),\bs u'''=\hat{\bs g}_n(\bs x''')$ and $\bs u''''=\hat{\bs g}_n(\bs x'''')$ can be twisted as shown in Figure~\ref{fig:twisted_quadrilateral} (left): this twist interrupts the estimator $\hat{\bs f}_n$ from being bijective over the quadrilateral, whereby $\hat{\bs f}_n$ is not entirely invertible over $I^2$. These twists can be eliminated by increasing the number of splits $t_n$ in most cases (see Figure~\ref{fig:twisted_quadrilateral} (right) and Proposition~\ref{prop:area_twisted}). 

\begin{figure}[!ht]
\centering
\includegraphics[width=0.7\textwidth]{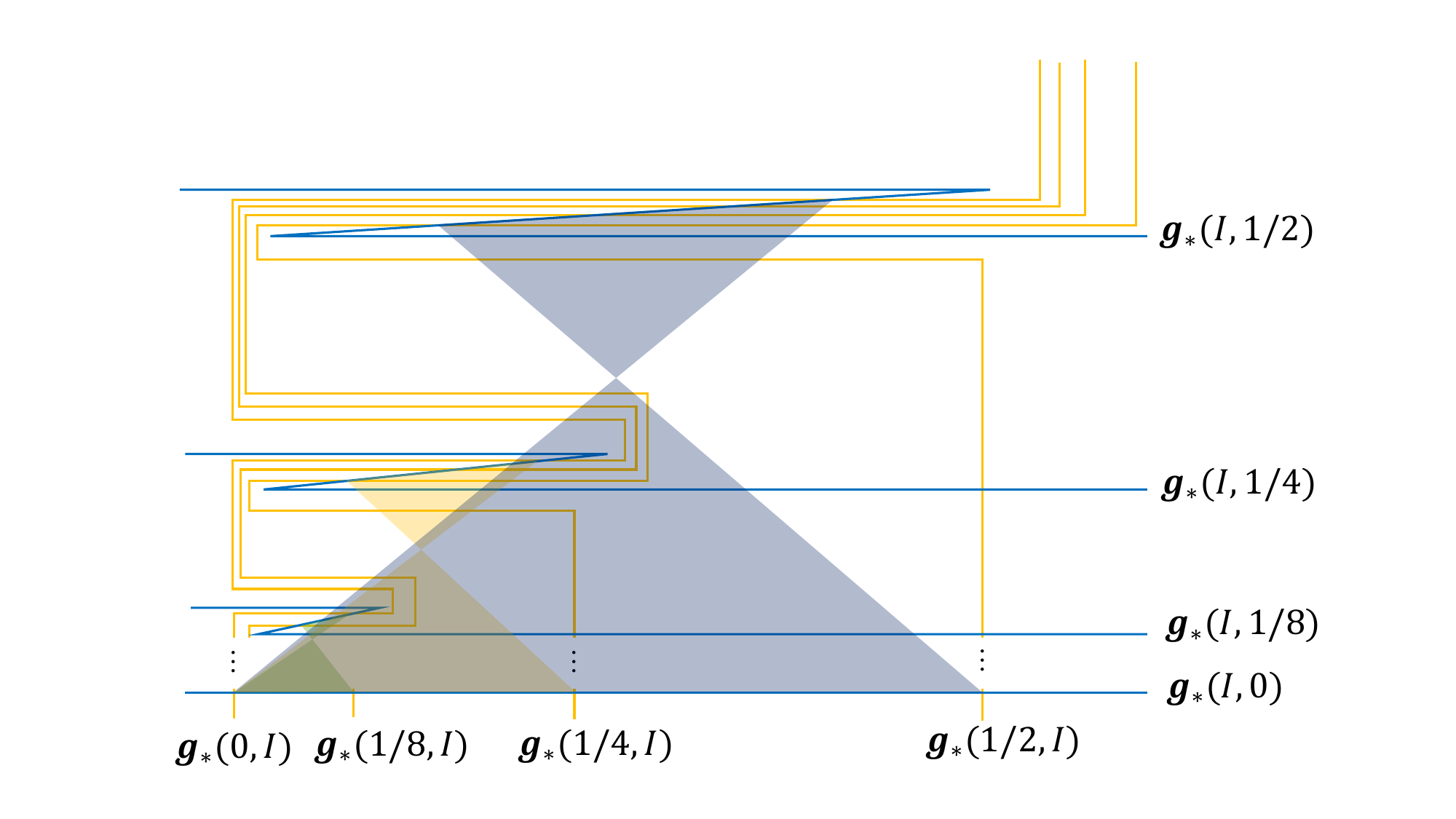}
\caption{A pathological example: for every $n \in \mathbb{N}$ with $t_n=2^n$, the quadrilateral which connects $\bs g_*(0,0),\bs g_*(0,1/t_n),\bs g_*(1/t_n,1/t_n),\bs g_*(1/t_n,0)$ in this order, is twisted. Although large twisted quadrilaterals are gradually decomposed into finer (smaller) quadrilaterals by increasing the number of split, smaller twists appear around $\bs g_*(0,0)$ indefinitely.}
\label{fig:pathological_example}
\end{figure}

Here, a natural question arises: can we further prove that the  developed estimator is entirely invertible on $I^2$? 
For most suitable $\bs f_*$, yes, our developed estimator is (asymptotically) entirely invertible as all the twists vanish as $t_n$ increases. 
Unfortunately, however, there exists a pathological example that such twist does not disappear even if $t_n$ increases. See Figure~\ref{fig:pathological_example} for such an pathological example. 
In this example, small twists (which can be ignored in the Lebesgue measure) appear indefinitely, and it prohibits our simple estimator from being entirely invertible for general $\bs f_* \in \flipinv$.

\end{document}